\documentclass[11pt]{article}

\usepackage{latexsym}
\usepackage{bbm}
\usepackage{amsmath}
\usepackage{amsfonts}
\usepackage{amssymb}
\usepackage{multirow}
\usepackage{latexsym}
\usepackage[dvips]{graphicx}
\usepackage{overpic}
\usepackage{graphicx}
\usepackage{relsize}
\usepackage{leftidx}
\newtheorem{theorem}{\bf Theorem}[section]
\newtheorem{lemma}[theorem]{\bf Lemma}
\newtheorem{prop}[theorem]{\bf Proposition}

\newtheorem{defn}{\bf Definition}[section]

\newtheorem{remark}{{\bf Remark}}[section]
\newenvironment{proof}{\noindent{\em Proof.}}{\quad \hfill$\Box$\vspace{2ex}}

\def \bZ {\Bbb Z}

\def \bR {\Bbb R}

\def \and {\, \mbox{\rm and}\, }

\def \supp {\,{\rm supp}\,}

\def \Re {\,{\rm Re}\,}
\def \Im {\,{\rm Im}\,}

\def \l {\left}
\def \r {\right}
\def \sup {\,{\rm sup}\,}
\makeatletter

\newcommand{\Rmnum}[1]{\expandafter\@slowromancap\romannumeral #1@}
\makeatother

\begin{document}
\title{\bf Uniform Estimates for Oscillatory Integral Operators with Polynomial Phases}
\author{Zuoshunhua Shi \thanks{~School of Mathematics and Statistics,~Central South University, Changsha Hunan 410083, P.R. China. E-mail address: {\it shizsh@163.com}.}}
\date{}
\maketitle{}
\begin{abstract}
In this paper, we shall establish uniform sharp $L^p$ decay estimates for oscillatory integral operators with weighted homogeneous polynomial phases. By this one-dimensional result, we can use the rotation method to obtain uniform sharp $L^p$ estimates of certain higher-dimensional oscillatory integral operators.
\end{abstract}
\section{Introduction}
In this paper, we mainly consider the stability of certain oscillatory integral operators. The issue of stability for oscillatory integrals includes two major cases: (i) stable estimates under a small perturbation of a given function; see Karpushkin \cite{Karpushkin86}, Phong-Stein-Sturm \cite{PSS1999}, Phong-Sturm \cite{PSturm2000}, Ikromov-Kempe-M\"{u}ller \cite{IKM1}, Ikromov-M\"{u}ller \cite{IM1} and Greenblatt \cite{greenblatt2011,greenblatt2014,greenblatt2016}; (ii) uniform estimates over a large class of phases satisfying certain nondegeneracy conditions; see Carbery-Christ-Wright \cite{carbery}, Carbery-Wright \cite{CaW}, Christ-Li-Tao-Thiele \cite{Christ-Li-Tao-Thiele}, Phong-Stein-Sturm \cite{PSS2001}, Greenblatt \cite{greenblatt3} and Gressman \cite{gressman}. The second case of stability will be investigated for certain oscillatory integral operators with polynomial phases.

Consider oscillatory integral operators of the following form:
\begin{equation}\label{General OIO}
T_{\lambda}f(x)=\int_{-\infty}^{\infty}e^{i\lambda S(x,y)}\varphi(x,y)f(y)dy,
\end{equation}
where $\lambda$ is a real parameter and $\varphi$ is a smooth cut-off function. For nondegenerate phases $S$, H\"{o}rmander \cite{hormander2} established the optimal $L^2$ decay estimate; see also Phong-Stein \cite{PS1997} and Phong-Stein-Sturm \cite{PSS2001} for uniform $L^2$ decay estimates. For general degenerate real-analytic phases $S$, the remarkable theorem in Phong-Stein \cite {PS1997} established the relation between the decay rate of the $L^2$ operator norm of $T_{\lambda}$ and the Newton distance of the phase $S$; see also Varchenko \cite{varchenko} for pioneering work on the relation between Newton distance and the optimal decay exponent of oscillatory integrals, and an illuminating simple proof of Phong-Stein Theorem in \cite {PS1997} due to Greeblatt \cite{greenblatt1}. Rychkov \cite{rychkov} extended this result to most smooth phases and full generalizations to
smooth phases were established by Greenblatt \cite{greenblatt2}. For higher dimensional analogues, Tang \cite{tangwan} and Greenleaf-Pramanik-Tang \cite{GPT} obtained sharp $L^2$ decay estimates under certain genericity assumptions. Recently, Xiao \cite{Xiao2017} proved the full range of sharp $L^p$ decay estimate for $T_{\lambda}$; see also \cite{yangchanwoo,yangchanwoo2,ShiYan,ShiXuYan} for some related work. We also refer the reader to Greenleaf-Seeger \cite{greenleafseeger2} for a comprehensive survey on degenerate oscillatory and Fourier integral operators.

By imposing uniform positive lower bounds on certain mixed derivatives of the phase $S$, Carbery-Christ-Wright \cite{carbery} established uniform sharp growth estimates for sublevel set operators associated with $S$. In \cite{carbery}, the authors also obtained uniform decay estimates for oscillatory integral operators with non-sharp decay exponents except for some especial cases; see also \cite{CaW}. Up to logarithmic terms, Phong-Stein-Sturm \cite{PSS2001} obtained uniform sharp $L^p$ estimates for a class of multilinear sublevel set operators and oscillatory integral operators. In this paper, we shall remove the logarithmic terms in these estimates of \cite{PSS2001} for $T_{\lambda}$. The phase $S$ is a weighted homogeneous polynomial in $\bR^2$ (see Definition \ref{sec2 def of weighted homo poly} below) which can be written as the following form:
\begin{equation}\label{Definition Polynomial phase}
S(x,y)=\sum_{k\eta+l=d}a_{k,l}x^ky^l,
\end{equation}
where $\eta,d>0$ are two fixed numbers. It is clear that $S$ has nonvanishing partial derivatives $\partial^k_x\partial^l_y S$ for positive integers $k,l$ with $a_{k,l}\neq 0$.

Now we state our main result in this paper.
\begin{theorem}\label{Theorem Main result uniform Lp decay}
Assume $S$ is a real-valued polynomial of form {\rm(\ref{Definition Polynomial phase})} in $\bR^2$. Let $T_{\lambda}$ be an oscillatory integral operator as in {\rm(\ref{General OIO})}. Then there exists a constant $C$, depending only on the cut-off $\varphi$ and the degree $deg(S)$ of phase $S$, such that for all positive integers $k,l$ satisfying $k\eta+l=d$,
\begin{equation}\label{Main Lp decay estimate}
\|T_{\lambda}f\|_{L^p}
\leq C |a_{k,l}|^{-\frac{1}{k+l}}
|\lambda|^{-\frac{1}{k+l}} \|f\|_{L^p},~~~p=\frac{k+l}{k}.
\end{equation}
\end{theorem}
Our proof of the theorem relies on some uniform damping estimates related to $T_{\lambda}$. Roughly speaking, one of the damping estimates is a uniform $L^2$ decay rate for oscillatory integral operators with damping factors related to the Hessian $S_{xy}''$. Another damping estimate consists of uniform $L^1\rightarrow L^{1,\infty}$, $H_E^1\rightarrow L^1$ and $L^1\rightarrow L^1$ boundedness for certain damped oscillatory integral operators. Here $H_E^1$ is a variant of the Hardy space $H^1$ associated with the phase $S$; see Phong-Stein \cite{PS1986}, Pan \cite{pan}, Greenleaf-Seeger \cite{greenleafseeger1} and Shi-Yan \cite{ShiYan}.

In this paper, the damped oscillatory integral operators with polynomial phases are of the form
\begin{equation}\label{Damped OIO}
W_{z}f(x)=\int_{\bR} e^{i\lambda S(x,y)} |D(x,y)|^z \varphi(x,y) f(y) dy,
\end{equation}
where $D$ is the damping factor and $z$ is the damping exponent. When the damping factor $D$ is taken as the Hessian $S''_{xy}$, Phong-Stein \cite{PS1998} proved the sharp $L^2$ decay estimate for $W_{\lambda}$ with the operator norm depending on upper bounds for $S''_{xy}$ together with its higher derivatives; see Seeger \cite{seeger1,seeger2} for decay $O(|\lambda|^{-1/2})$ with $\Re(z)>1/2$. As a consequence of our uniform $L^2$ damping estimates, we are able to establish the stability of the result in Phong-Stein \cite{PS1997} for $W_{z}$ with polynomial phases of the form (\ref{Definition Polynomial phase}). More precisely, we have the following theorem.
\begin{theorem}\label{Uniform Damped OIO}
Assume $S$ is a real-valued polynomial of the form {\rm(\ref{Definition Polynomial phase})}. Let $W_{z}$ be defined as in {\rm(\ref{Damped OIO})} with $D(x,y)=S''_{xy}(x,y)$. Then there exists a constant $C=C(deg(S))$, depending only on the degree of $S$, such that
\begin{equation*}
\|W_{z}f\|_{L^2}\leq C (1+|z|)^2
M |\lambda|^{-1/2}\|f\|_{L^2}
\end{equation*}
for all $z\in\mathbb{C}$ with $\Re(z)=\frac{1}{2}$ and all $\varphi\in C_0^{\infty}(\bR^2)$ satisfying
\begin{equation*}
\mathop{\sup}_{\Omega}
\sum_{k=0}^{2}
\left( \big(\delta^{}_{\Omega,h}(x) \big)^k |\partial^k_y\varphi(x,y)|
+
\big( \delta^{}_{\Omega,v}(y) \big)^k |\partial^k_x\varphi(x,y)|\right)
\leq M,
\end{equation*}
where $M$ is a positive number and $\Omega$ is a horizontally and vertically convex domain such that the cut-off $\varphi$ is supported in $\Omega$. Here $\delta_{\Omega,h}(x)$ and $\delta_{\Omega,v}(y)$ denote the length of the cross sections $\{y:(x,y)\in\Omega\}$ and $\{x:(x,y)\in\Omega\}$, respectively.
\end{theorem}

With the rotation method and a variant of Stein-Weiss interpolation with change of measures, we can apply Theorem \ref{Theorem Main result uniform Lp decay} to obtain some uniform $L^p$ estimates for higher dimensional oscillatory integral operators.

For an operator $W$ on $L^2(\bR)$, $\|W\|$ denotes its operator norm. If $G$ is a finite set, the notation $|G|$ means the number of elements in $G$. For two real numbers $a,b\geq 0$, $a\lesssim b$ implies that there exists an appropriate constant $C$ such that $a\leq Cb$, and $a\approx b$ means $a\lesssim b$ and $b\lesssim a$. For a general set $A$, we use $\chi_{A}$ to denote its characteristic function.

\section{Preliminaries}
In this section, we shall first present some basic notions concerning horizontally and vertically convex domains. With these useful notions, uniform $L^2$ estimates will be established for non-degenerate oscillatory integral operators which are supported in horizontally (vertically) convex domains. These domains with special geometric convexity were implicit in the work of Carbery-Christ-Wright \cite{carbery}, and defined explicitly by Phong-Stein-Sturum \cite{PSS2001}. These uniform $L^2$ decay estimates are often known as the operator version of van der Corput lemma; see \cite{PS1997,PSS2001}. Finally, we also include an interpolation result with change of power weights.

\begin{defn}
Let $\Omega\subseteq\bR^2$. We say that $\Omega$ is horizontally convex if $(x,z),(y,z)\in \Omega$ imply $(u,z)\in \Omega$ for all $x\leq u \leq y$. Similarly, $\Omega$ is said to be vertically convex if $(x,y),(x,z)\in\Omega$ imply $(x,v)\in\Omega$ for all $y\leq v \leq z$.
\end{defn}

We shall give a simple relation between horizontal convexity and a curved trapezoid.
\begin{defn}\label{sec2 curved trapezoid}
Assume $g$ and $h$ are two monotone functions on $[a,b]$ with $g\leq h$. Then we call
\begin{equation*}
\Omega=\{(x,y)\mid a\leq x\leq b,~g(x)\leq y \leq h(x)\}
\end{equation*}
a curved trapezoid.
\end{defn}
It is clear that a curved trapezoid is vertically convex. By the monotonicity of $g,h$, one can verify that a curved trapezoid is also horizontally convex. If the monotonicity assumption on $g,h$ is dropped, the horizontal convexity for $\Omega$ is not generally true. But we have the following result.

\begin{lemma}
Assume $g,h\in C[a,b]$ satisfy $g\leq h$ on $[a,b]$. If the set $\Omega$ in Definition {\rm \ref{sec2 curved trapezoid}} is horizontally convex, then $[a,b]$ can be divided into three intervals $I_1,I_2,I_3$ with disjoint interiors such that each domain
$$\Omega_i=\{(x,y)\mid x\in I_i,~g(x)\leq y \leq h(x)\}$$
is a curved trapezoid.
\end{lemma}

\begin{proof}
Choose a point $c\in [a,b]$ such that $h$ achieves its maximum at $c$. Then $h$ must be increasing on $[a,c]$ and decreasing on $[c,d]$. Otherwise, we can choose two points $a\leq x_1<x_2 \leq c$ such that $h(x_1)>h(x_2)$. By the continuity of $h$, we can choose $x_3\in [x_2,c]$ with $h(x_3)=h(x_1)$. Since $\Omega$ is horizontally convex, we find that $(u,h(x_1))\in\Omega$ for all $x_1\leq u \leq x_3$. But this contradicts the fact $(x_2,h(x_1))\notin\Omega$. Thus $h$ does not decrease on $[a,c]$. By a similar argument, we can prove that $h$ is decreasing on $[c,b]$.

Assume $g$ attains its minimum at some point $d\in [a,b]$. By the same argument as above, it follows that $g$ is decreasing on $[a,d]$ and is increasing on $[d,b]$.

Without loss of generality, assume $c\leq d$. Then $I_i$ are obtained by taking $a,c,d,b$ as the endpoints of these intervals. The proof is therefore complete.
\end{proof}

For horizontally and vertically convex domains, some notations will be used frequently in this paper.\\
$~~$$\bullet$ $\delta_{\Omega,h}(x)$:~~the length of the cross section $\{y:(x,y)\in\Omega\}$. Here the subscript $h$ means that $\delta_{\Omega,h}$ is a function of the horizontal variable.\\
$~~$$\bullet$ $\delta_{\Omega,v}(y)$:~~the length of the cross section $\{x:(x,y)\in\Omega\}$. The subscript $v$ suggests that $y$ is the vertical component.\\
$~~$$\bullet$ $\Omega^*_h$:~~a horizontally expanded domain of $\Omega$; see Definition \ref{HV expanded domain}. \\
$~~$$\bullet$ $\Omega^*_v$:~~a vertically expanded domain of $\Omega$; see Definition \ref{HV expanded domain}.\\
$~~$$\bullet$ $\Omega^*$:~~an expanded domain of $\Omega$ of form $\Omega^{\ast}=\Omega_h^{\ast}\cup\Omega_v^{\ast}$.\\
$~~$$\bullet$ $I_{\Omega,h}(x)$:~~the cross section $\{y:(x,y)\in\Omega\}$. The subscript $h$ means that $I_{\Omega,h}$ is taken with respect to the horizontal component. \\
$~~$$\bullet$ $I_{\Omega,v}(y)$:~~the cross section $\{x:(x,y)\in\Omega\}$. The subscript $v$ means that $I_{\Omega,v}$ is taken with respect to the vertical component. \\
$~~$$\bullet$ $a\wedge b$:~~the minimum of two real numbers $a$ and $b$.\\
$~~$$\bullet$ $a\vee b$:~~the maximum of $a$ and $b$.\\

Now we give the definition of polynomial type functions which behave like polynomials; see Phong and Stein \cite{PS1997,PS1998}.

\begin{defn}\label{sec2 def poly type fun}
Let $J$ be a bounded interval and $F\in C^{N}(J)$ for some $N\geq 1$. Then $F$ is said to be a polynomial type function with order $N$ if there exists a constant $C_F>0$ such that
\begin{equation*}
  \sup_{x\in J}|F^{(N)}(x)|
\leq
C_F \inf_{x\in J}|F^{(N)}(x)|.
\end{equation*}
For a polynomial type function $F$, the notation $N_F$ denotes the order of $F$.
\end{defn}

A polynomial $P\in\bR[x]$ of degree $N$ is of polynomial type with order $N$ on any bounded interval. Some properties of polynomial type functions were obtained by Phong and Stein \cite{PS1997,PS1998}.

\begin{lemma}\label{sec2 poly type func lemma}
Assume $F$ is a real-valued polynomial type function on a bounded interval $J$. Let $N_F$ be the order of $F$. We have the following two statements.\\
{\rm(i)} Then there exists a constant $C=C(N_F)$ such that for any bounded interval $I\subseteq J$,
\begin{equation}\label{sec2 esti of poly type func}
\sum_{k=0}^{N_F}|I^{\ast}|^k\mathop{\sup}\limits_{x\in I^{\ast}}
|F^{(k)}(x)|
\leq
C \mathop{\sup}\limits_{x\in I } |F(x)|,
\end{equation}
where $I^{\ast}$ is the double of $I$ in $J$, i.e., $I^{\ast}:=(c_I-2\delta,c_I+2\delta)\cap J$ with $I=(c_I-\delta,c_I+\delta)$.

\noindent{\rm (ii)} If, in addition, for two constants $\mu,A>0$, $F$ satisfies $\mu\leq |F(x)| \leq A\mu$ on $J$, then for all $z\in\mathbb{C}$ and intervals $I\subseteq J$, the above estimate is also true for $|F(x)|^z$ with a constant $C=C(N_F,C_F,A,z)$.
\end{lemma}
\begin{remark}
$~$\\
$~~${\rm $\bullet$} Phong and Stein {\rm\cite{PS1997}} proved the estimate {\rm(\ref{sec2 esti of poly type func})} only for $k=0,1$. By a simple iteration of Phong-Stein's estimates for derivatives of $F$, we can prove {\rm(\ref{sec2 esti of poly type func})}.\\
$~~${\rm $\bullet$} Under the assumptions in the second statement, Phong and Stein {\rm\cite{PS1998}} proved that $|F|^{1/2}$ satisfies the estimate {\rm (\ref{sec2 esti of poly type func})}.
\end{remark}
\begin{proof}
As explained in the above remark, we only need to prove that $|F|^z$ satisfies the estimate {\rm (\ref{sec2 esti of poly type func})}.
By our assumptions, $F$ does not change sign on $J$. Without loss of generality, we assume $F>0$. By the chain rule and an induction argument, we can show that the $N-th$ derivative of $|F|^z$ takes the following form:
\begin{equation*}
\frac{d^N}{dx^N}|F(x)|^z
=
|F(x)|^{z-N}
\sum_{} C_{k_1,k_2,\cdots,k_N}(z)
F^{(k_1)}(x)F^{(k_2)}(x) \cdots F^{(k_N)}(x)
\end{equation*}
where the summation is taken over all nonnegative integers $k_i$ satisfying $k_1+k_2+\cdots+k_N=N$. By repeated applications of (\ref{sec2 esti of poly type func}) to each term in the above summation, we can obtain the desired estimate.
\end{proof}

The following almost orthogonality principle turns out to be useful in this paper; see Phong-Stein-Sturm \cite{PSS2001} for its proof.
\begin{lemma}\label{simple almost orthogonality principle}
Let $K(x,y)$ be a Lebesgue measurable function in $\bR^2$, and let $T$ be the integral operator associated with $K$, i.e.,
$
Tf(x)=\int_{\bR}K(x,y)f(y)dy.
$
Assume there exist two sequences $\{I_s\}$ and $\{J_t\}$ of intervals such that $K(x,y)=\sum_{s,t}K(x,y)\chi_{I_s}(x)\chi_{J_t}(y)$ for almost every $(x,y)$ in $\bR^2$. Denote by $T_{s,t}$ the integral operator associated with $K(x,y)\chi_{I_s}(x)\chi_{J_t}(y)$. If $I_s\cap I_t=\emptyset$ and $J_s\cap J_t=\emptyset$ for $|s-t|\geq N_0$, then for all $1\leq p \leq \infty$ we have
\begin{equation*}
\|T\|_{L^p\rightarrow L^p}
\leq
N_0
\mathop{\sup}\limits_{s,t}\|T_{s,t}\|_{L^p\rightarrow L^p},
\end{equation*}
where $\|T\|_{L^p\rightarrow L^p}$ and $\|T_{s,t}\|_{L^p\rightarrow L^p}$ are the $L^p$ operator norms of $T$ and $T_{s,t}$, respectively.
\end{lemma}

Now we state the operator van der Corput lemma for nondegenerate oscillatory integral operators. The following lemma, with varying width of the ``curve box", was first established by Phong-Stein-Sturm in \cite{PSS2001}.

\begin{lemma}\label{operator van der Corput}
Assume $S$ is a real-valued polynomial in $\bR^2$. Let $T_{\lambda}$ be defined as in (\ref{General OIO}) with $\varphi$ supported in a curved trapezoid
$\Omega=\{(x,y)\mid a\leq x \leq b,~g(x)\leq y \leq h(x)\}.$ Suppose $S$ satisfies the following two conditions:

\noindent~~{\rm (i)} For some $\mu,A_1>0$,
$$\mu\leq |S''_{xy}(x,y)|\leq A_1\mu,~~~(x,y)\in\Omega.$$

\noindent~~{\rm (ii)} Let $\delta_{\Omega,h}(x)$ be the length of the cross section $I_{\Omega,h}(x)=\{y\mid (x,y)\in\Omega\}$. There exists a constant $A_2>0$ such that
\begin{equation*}
\sum\limits_{k=0}^{2}
\mathop{\sup}\limits_{(x,y)\in\Omega}
\Big( \delta_{\Omega,h}(x) \Big)^k |\partial_y^k\varphi(x,y)|\leq A_2.
\end{equation*}
Then there exists a constant $C=C(deg(S),A_1)$ such that
$$\|T_{\lambda}f\|_{L^2}
\leq CA_2~
|\lambda\mu|^{-1/2}\|f\|_{L^2}.$$
\end{lemma}

Let $\Omega$ be a horizontally and vertically convex domain in $\bR^2$. In this paper, we only consider those domains $\Omega$ for which all horizontal and vertical cross sections are closed intervals. Under this assumption, there exist two numbers $a,b$ and functions $g,h$ such that $\Omega$ can be written as
\begin{equation}\label{sec2 vertical convex dom}
\Omega=\{(x,y)\in\bR^2 \mid a\leq x \leq b,~g(x)\leq y \leq h(x)\}.
\end{equation}
Similarly, for some $c,d\in\bR$ and functions $u,v$, the domain $\Omega$ can also be given by
\begin{equation}\label{sec2 hori convex dom}
\Omega=\{(x,y)\in\bR^2 \mid c\leq y \leq d,~u(y)\leq x \leq v(y)\}.
\end{equation}
Throughout the rest of this paper, one will see that all horizontally and vertically convex domains considered below are of the above two forms.

Now we shall define an expanded domain $\Omega^{\ast}$ for which $\delta_{\Omega^*,h}$ and $\delta_{\Omega^*,v}$ are suitably larger than $\delta_{\Omega,h}$ and $\delta_{\Omega,v}$, respectively. Throughout the rest of this section, we always assume $\Omega$ is a horizontally and vertically convex domain in $\bR^2$.

\begin{defn}\label{HV expanded domain}
We say that $\Omega^*_h$ is a horizontally expanded domain of $\Omega$, if there exists a positive number $\epsilon>0$ and a nonnegative function $\eta$ on $[c,d]$ such that
$$\Omega^*_h=\{(x,y)\mid c\leq y \leq d,~u(y)-\eta(y)\leq x \leq v(y)+\eta(y)\}$$
and for each $y$
$$\delta_{\Omega_h^*,v}(y)\geq (1+2\epsilon)\delta_{\Omega,v}(y).$$
Similarly, $\Omega_{v}^{\ast}$ is said to be a vertically expanded domain of $\Omega$ if there exists a number $\epsilon>0$ and a nonnegative function $\gamma$ on $[a,b]$ such that
$$\Omega_{v}^{\ast}=\{(x,y)\mid a\leq x \leq b,~g(x)-\gamma(x)\leq y \leq g(x)+\gamma(x)\}$$
and there holds
$$\delta_{\Omega_v^{\ast},h}(x)\geq (1+2\epsilon)\delta_{\Omega,h}(x),~x\in [a,b].$$
\end{defn}

\begin{defn}\label{expanded domain}
The set $\Omega^{\ast}$ is said to be an expanded domain of $\Omega$, if there exist horizontally and vertically expanded domains $\Omega_h^*$ and $\Omega_v^*$, defined as in Definition {\rm \ref{HV expanded domain}}, such that $\Omega^*=\Omega_h^*\cup \Omega_v^*$.
\end{defn}

\begin{defn}\label{definition expanded intervals}
Let $I$ be a bounded interval in $\bR$. We use $I^*(B)$ to denote the concentric interval with length expanded by the factor $B>0$.
\end{defn}
We now state a useful lemma concerning expanded intervals.
\begin{lemma}\label{lemma expanded intervals}
Assume $I_1$ and $I_2$ are two bounded intervals in $\bR$. If $I_1\cap I_2\neq \emptyset$, then for every $B>1$ we have
$|I_1^{\ast}(B)\cap I_2^{\ast}(B)| \geq (B-1) |I_1|\wedge|I_2|$. Here $|I|$ is the length of the interval $I$.
\end{lemma}
\begin{proof}
Assume $|I_1|\leq |I_2|$ and $I_1\cap I_2=(a,b)$. Let $\delta_1=|I_1|$. By Definition \ref{definition expanded intervals}, we see that
$$ \left( a-\frac{B-1}{2}\delta_1,b+\frac{B-1}{2}\delta_1 \right)
\subseteq
I^*_1(B)\cap I^*_2(B)$$
for $B>1$. Hence $|I^*_1(B)\cap I^*_2(B)|\geq (B-1)\delta_1=(B-1)|I_1|$.
\end{proof}

With the above preliminaries, we can now present the oscillation estimate to treat the almost orthogonality between two oscillatory integral operators. It should be pointed out that operators considered here are supported on horizontally (vertically) convex domains. For operators supported on curved trapezoids, the corresponding oscillation estimate was obtained by Phong and Stein in \cite{PS1998}.

\begin{lemma}\label{Lemma almost orth van der Corput}
Let $S$ be a real-valued polynomial in $\bR^2$ with degree $deg(S)$. Assume $T_{\lambda}^{(1)}$ and $T_{\lambda}^{(2)}$ are defined as $T_{\lambda}$ in {\rm(\ref{General OIO})}, but with the cut-off $\varphi$ replaced by $\varphi_1$ and $\varphi_2$ respectively. Suppose that $\supp(\varphi_i)\subseteq \Omega_i$ for two horizontally and vertically convex domains $\Omega_1$ and $\Omega_2$. Assume that all of the following conditions are true.

{\rm(i)} For some $\mu,~A>0$, there exist expanded domains $\Omega_1^*$ and $\Omega_2^*$, defined as in {\rm Definition \ref{expanded domain}}, such that
$$\mu\leq |S''_{xy}(x,y)|\leq A\mu,~(x,y)\in\Omega_1^*.$$

{\rm(ii)} For any horizontal line segment $L$ joining one point in $\Omega_1^*$ and another one in $\Omega_2^*$, the Hessian $S''_{xy}(x,y)$ does not change sign on $L$ and $\sup_L|S_{xy}''(x,y)|\leq A\mu$.

{\rm(iii)} There exists a positive number $B$ such that for each $y$, we have
$$I_{\Omega_2^*,v}(y)\subseteq I^*_{\Omega_1^*,v}(y,B),$$
where $I_{\Omega_i^*,v}(y)$ is the horizontal cross-section of $\Omega_i^*$ at height $y$, i.e., $I_{\Omega_i^*,v}(y)=\{x:(x,y)\in\Omega_{i}^*\}$, and the interval $I^*_{\Omega_1^*,v}(y,B)$ has the same center as $I_{\Omega_1^*,v}(y)$, but its length is $B$ times as long as that of $I_{\Omega_1^*,v}(y)$; see also Definition \ref{definition expanded intervals}.

{\rm(iv)} For two positive numbers $M_1$  and $M_2$, it is true that
$$\sum_{k=0}^{2}\mathop{\sup}_{\Omega_i}\l(\delta_{\Omega_i,h}(x)\r)^k|\partial_y^k\varphi_i(x,y)|
\leq
M_i,~~~~~i=1,2.$$

\noindent Then there exists a constant $C$, depending only on $deg(S)$, $A$, $B$ and the expanded factors $\epsilon$ appearing in the definition of $\Omega_1^*$ and $\Omega_2^*$, such that
$$\l\|T_{\lambda}^{(1)}T_{\lambda}^{(2)\ast}\r\|_{L^2\rightarrow L^2}
\leq C M_1 M_2
|\lambda\mu|^{-1}.$$
\end{lemma}

\begin{proof}
Let $K$ be the kernel associated with $T_{\lambda}^{(1)}T_{\lambda}^{(2)\ast}$. Then $K$ can be written as
\begin{equation}
K(x,y)=\int_{\bR}e^{i\lambda[S(x,z)-S(y,z)]}\varphi_1(x,z)
\overline{\varphi_2}(y,z)dz.
\end{equation}
For those $z's$ satisfying $\varphi_1(x,z)\overline{\varphi_2}(y,z)\neq 0$, we deduce form the assumptions (i), (ii) and (iii) that
$\Phi(z)=\partial_zS(x,z)-\partial_zS(y,z)$
satisfies
$$|\Phi(z)|=\l|\int_y^x\partial_u\partial_zS(u,z)du\r|\geq C\mu|x-y|,~~~z\in I_{\Omega_1^*,h}(x)\cap I_{\Omega_2^*,h}(y),$$
where the constant $C$ depends only on the factor $B$ in the assumption (iii) and $\epsilon_1,\epsilon_2$ appearing in the definition of $\Omega_1^{\ast},\Omega_2^{\ast}$. By integration by parts, we have
\begin{eqnarray*}
K(x,y)
&=&
\int_{\bR}D_z^2\l(e^{i\lambda(S(x,z)-S(y,z))}\r)\varphi_1(x,z)
\overline{\varphi_2}(y,z)dz\\
&=&
\int_{\bR}e^{i\lambda(S(x,z)-S(y,z))}(D^t)^2\l(\varphi_1(x,z)
\overline{\varphi_2}(y,z)\r)dz,
\end{eqnarray*}
where $D$ is the differential operator $Df(z)=[i\lambda\Phi(z)]^{-1}f'(z)$ and $D^t$ is its transpose, i.e., $D^tf(z)=-\partial_z[(i\lambda\Phi(z))^{-1}f(z)].$

By the assumption (ii), we see that $|\Phi(z)|\leq A\mu|x-y|$ for $z\in I_{\Omega_1^*,h}(x)\cap I_{\Omega_2^*,h}(y)$. If $I_{\Omega_1,h}(x)\cap I_{\Omega_2,h}(y)=\emptyset$, then $K(x,y)=0$. Thus we need only to consider those $x,y$ for which $I_{\Omega_1,h}(x)\cap I_{\Omega_2,h}(y)\neq\emptyset$. By Lemma \ref{lemma expanded intervals}, we have
\begin{eqnarray*}
|I_{\Omega_1^*,h}(x)\cap I_{\Omega_2^*,h}(y)|
&\geq &
|I_{\Omega_{1,v}^*,h}(x)\cap I_{\Omega_{2,v}^*,h}(y)|\\
&\geq &
\epsilon |I_{\Omega_1,h}(x)|\wedge| I_{\Omega_2,h}(y)|
\end{eqnarray*}
with $\epsilon=2\epsilon_1\wedge\epsilon_2$. Here $\epsilon_1,\epsilon_2$ are the expanded factors for $\Omega_1^*$ and $\Omega_2^*$; see Definitions \ref{HV expanded domain} and \ref{expanded domain}. It follows from Lemma \ref{sec2 poly type func lemma} that
\begin{eqnarray*}
\l|\frac{d^k}{dz^k}\l(\frac{1}{\Phi(z)}\r)\r|
&\leq &
C(\mu|x-y|)^{-1} |I_{\Omega_1^*,h}(x)\cap I_{\Omega_2^*,h}(y)|^{-k}\\
&\leq &
C(\mu|x-y|)^{-1}
\Big(\delta_{\Omega_{1},h}(x)
\wedge\delta_{\Omega_{2},h}(y)\Big)^{-k}
\end{eqnarray*}
for $0\leq k \leq 2$ with a constant $C=C(deg(S),\epsilon_{1},\epsilon_2).$ Here $\epsilon_1$ and $\epsilon_2$ are factors associated with $\Omega_{1,h}^*$ and $\Omega_{2,h}^*$, respectively.

On the other hand, we see that
\begin{equation*}
\Big| (D^t)^2\l(\varphi_1(x,z)\overline{\varphi_2}(y,z)\r) \Big|
\leq
\frac{1}{|\lambda|^2}
\sum C_{ \overrightarrow{\mathbf{k}}  }
\left| \frac{d^{k_1}}{dz^{k_1}}\left( \frac{1}{\Phi(z)} \right) \right|
\left| \frac{d^{k_2}}{dz^{k_2}}\left( \frac{1}{\Phi(z)} \right) \right|
\left| \frac{\partial^{k_3}\varphi_1}{\partial z^{k_3}} \right|
\left| \frac{\partial^{k_4}\overline{\varphi}_2}{\partial z^{k_4}} \right|,
\end{equation*}
where the summation is taken over all $\overrightarrow{\mathbf{k}}\in \mathbb{Z}^4$ with nonnegative components $k_i$ satisfying $k_1+k_2+k_3+k_4=2$.
Combining above estimates together with the assumption (iv), we obtain
\begin{equation*}
|K(x,y)|
\leq
CM_1M_2
\l(1+|\lambda|\mu\delta_{\Omega_{1},h}(x)
\wedge\delta_{\Omega_{2},h}(y)|x-y|\r)^{-2}
\delta_{\Omega_{1},h}(x)
\wedge\delta_{\Omega_{2},h}(y).
\end{equation*}
Note that
\begin{equation*}
\frac{a\wedge b}{\l(1+|\lambda|\mu a \wedge b|x-y|\r)^2}
\leq
\frac{a}{\l(1+|\lambda| \mu a|x-y|\r)^2}
+\frac{b}{\l(1+|\lambda|\mu b|x-y|\r)^2}
\end{equation*}
with $a=\delta_{\Omega_{1},h}(x)$ and $b=\delta_{\Omega_{2},h}(y)$. Thus for any $f,g\in L^2$,
\begin{eqnarray*}
\int_{\bR^2}|K(x,y)||f(y)g(x)|dxdy
&\leq &
CM_1M_2
|\lambda\mu|^{-1}
\l(\int \mathbf{M}f(x)|g(x)|dy+\int|f(y)|\mathbf{M}g(y)dy\r)\\
&\leq &
CM_1M_2
|\lambda\mu|^{-1}
\|f\|_{L^2}\|g\|_{L^2},
\end{eqnarray*}
where $\mathbf{M}$ is the Hardy-Littlewood maximal operator, and the constant $C$ depends on $deg(S)$, $\epsilon_1$, $\epsilon_2$ and $A$, but not on $\lambda,\mu$ and $f,g$. Thus we complete the proof of the lemma.
\end{proof}

\begin{remark}\label{remark almost ortho esti}
If we change the role of $x$ and $y$ in the lemma, we also have a similar estimate for the $L^2$ operator norm of $~$$T_{\lambda}^{(1)\ast}T_{\lambda}^{(2)}$. Suppose the cut-off functions $\varphi_1$ and $\varphi_2$ are also supported in another two horizontally and vertically convex domains $\widetilde{\Omega}_1$ and $\widetilde{\Omega}_2$, respectively. Then we can improve the bound $M_i$ by replacing $\Omega_i$ by $\Omega_i\cap\widetilde{\Omega}_i$ in the assumption $(iv)$. For the proof of this improved result, we can follow the above argument line by line, but with $\Omega_i$ replaced by $\Omega_i\cap\widetilde{\Omega}_i$.
\end{remark}

The following lemma is useful in the interpolation with change of power weights. Its formulation and  proof are in many ways like the Stein-Weiss interpolation with change of measures; see Stein-Weiss \cite{stein-weiss},  Pan-Sampson-Szeptycki \cite{pansampson} and Shi-Yan \cite{ShiYan}.

\begin{lemma}\label{Lemma interpolaiton with change of weights}
Assume $T$ is a sublinear operator defined for simple functions in $\mathbb{R}$ with Lebesgue measure. Suppose that there exist two constants $A,B>0$ such that

\noindent $~$ {\rm(i)} $\| Tf \|_{L^{\infty}} \leq A \| f \|_{L^1}$ for all simple functions f;

\noindent $~$ {\rm (ii)} $\l\| |x|^a Tf  \r\|_{L^{p_0}} \leq B\| f \|_{L^{p_0}},$
where $a \in \mathbb{R}$, $1< p_0 <\infty$ and $a\neq -1/p_0$.

\noindent Then for all $0<\theta<1$, there exists a constant $C=C(a,p_0,\theta)$ such that
\begin{equation*}
\Big\| |x|^{\gamma} Tf \Big\|_{L^p}
\leq
CA^{1-\theta} B^{\theta} \|f\|_{L^p},~~~~~
\gamma=-(1-\theta)+\theta a,~~~\frac{1}{p}=(1-\theta)+\frac{\theta}{p_0},
\end{equation*}
where $f$ is an arbitrary simple function.
\end{lemma}
\begin{proof}
By the assumption {\rm (i)}, it is easy to see that $|x|^{-1} T$ is bounded from $L^1(dx)$ into $L^{1,\infty}(dx)$. Since $|x|^a T$ is bounded on $L^{p_0}(dx)$, we can define an operator $Sf=|x|^bTf$ and a measure $d\mu=|x|^{c}dx$ such that $S$ is bounded from $L^1(dx)$ into $L^{1,\infty}(d\mu)$ and maps $L^{p_0}(dx)$ continuously into $L^{p_0}(d\mu)$.

Now we need the simple fact that $|x|^{b}$ belongs to $L^{1,\infty}(|x|^{-1-b}dx)$ for any nonzero real number $b$. For any $\lambda>0$, $|x|^b>\lambda$ implies $|x|>\lambda^{1/b}$. Hence
\begin{eqnarray*}
d\mu(\{x\in\bR: |x|^b>\lambda\})
=
\int_{|x|>\lambda^{1/b}} |x|^{-1-b}dx
=
C(b)\lambda^{-1}.
\end{eqnarray*}
Hence we obtain the claim for $b>0$. The claim for $b<0$ can be proved in the same way.

By the above fact, we see that $S$ is bounded from $L^1(dx)$ into $L^{1,\infty}(d\mu)$ provided that $b+c=-1$ and $b\neq 0$. On the other hand, using assumption {\rm (ii)}, we see that $S$ is bounded from $L^{p_0}(dx)$ into $L^{p_0}(d\mu)$ if $p_0b+c=p_0a$. With this two equations for $b$ and $c$, we obtain
\begin{equation*}
b=\frac{p_0a+1}{p_0-1},~~~~~
c=-1-\frac{p_0a+1}{p_0-1}.
\end{equation*}
Recall that we have assumed $a \neq -1/p_0$ in the assumption (ii). This implies $b\neq 0$. By the Marcinkiewicz interpolation theorem, we have a constant $C=C(a,p_0,\theta)$ such that
\begin{equation*}
\||x|^bTf\|_{L^p(|x|^cdx)}
\leq
C A^{1-\theta} B^{\theta} \|f\|_{L^p(dx)}
\end{equation*}
with $\frac{1}{p}=1-\theta+\frac{\theta}{p_0}$. Note that $$b+\frac{c}{p}=b+c\left(1-\theta+\frac{\theta}{p_0}\right)
=(1-\theta)(b+c)+\theta\left( b+\frac{c}{p_0} \right)=-(1-\theta)+\theta a.$$
Then the desired inequality in the lemma follows immediately.
\end{proof}

Now we introduce the notion of weighted homogeneous polynomials considered in this paper.
\begin{defn}\label{sec2 def of weighted homo poly}
A polynomial $Q(x,y)=\sum_{k,l\geq 0}a_{k,l}x^ky^l$ is said to be $(p,q)$-weighted homogeneous, if there exist two positive integers $p$ and $q$, relatively prime, such that $pk+ql$ is a constant independent of $k,l$ for which $a_{k,l}\neq0$.
\end{defn}

\begin{lemma}\label{sec2 factorization Rxy}
Assume $Q\in\bR[x,y]$, the ring of polynomials in $x$ and $y$ with real coefficients. If $Q$ is a nonzero $(p,q)$-weighted homogeneous, $p$ and $q$ relatively prime, then $Q$ can be factorized as follows:
\begin{equation*}
Q(x,y)=c x^m y^n \prod_{k=1}^{N_1}\l(x^q-\alpha_k y^p\r)
\prod_{l=1}^{N_2}\left(x^{2q}+2\beta_lx^qy^p+\gamma_l y^{2p}\right),
\end{equation*}
where $c\in \bR$, $\alpha_k\in\bR\backslash\{0\}$, and $\beta_l,\gamma_l\in\bR$ satisfy $\gamma_l>\beta_l^2$.
\end{lemma}
\begin{proof}
If $Q$ is a monomial, then the above factorization is trivial. Assume now $Q$ contains at least two nonzero terms. Let $m$ be the least integer $k$ such that there exists a coefficient $a_{k,l}\neq 0$ for some $l$, and $n$ be the least integer $l$ such that there exists an $a_{k,l}\neq 0$ for some $k$. Then we can write $Q(x,y)=x^my^n\widetilde{Q}(x,y)$ for some $\widetilde{Q}\in\bR[x,y]$. Here $\widetilde{Q}$ satisfies $\widetilde{Q}(x,0)\neq 0$ and $\widetilde{Q}(0,y)\neq 0$. Note that $\widetilde{Q}$ is also $(p,q)$-weighted homogeneous. For this reason, we assume $m=n=0$ throughout our proof.

Write now $Q(x,y)=\sum_{k,l\geq 0}a_{k,l}x^ky^l$. Let $D:=pk+ql$ for $a_{k,l}\neq 0$. The assumption $m=n=0$ implies that $a_{k,0}\neq 0$ and $a_{0,l}\neq 0$ for two positive integers $k,l$. Hence $p|D$ and $q|D$. Since $p$ and $q$ are relatively prime, we have $pq|D$. It follows immediately that if $a_{k,l}\neq 0$ then $q|k$ and $p|l$. Set $u=x^q$ and $v=x^p$. Since $\frac{k}{q}+\frac{l}{p}=d:=\frac{D}{pq}$ for all $k,l$ such that $a_{k,l}\neq 0$, we see that $P(u,v):=Q(x,y)$ is a homogeneous polynomial in $u$ and $v$. Observe that $P(u,v)=v^dP(\frac{u}{v},1)$ and $P(0,1)\neq 0$. By the factorization theorem in $\bR[x]$, we have
$$
P(u,v)
=cv^d\prod_{k=1}^{N_1}\left(\frac{u}{v}-\alpha_k \right)
\prod_{l=1}^{N_2}\l( \frac{u^2}{v^2}+2\beta_l\frac{u}{v}+\gamma_l \r),
$$
where $c,\alpha_k\in\bR\backslash\{0\}$, and $\beta_l,\gamma_l$ are real numbers satisfying $\gamma_l>\beta_l^2$.
\end{proof}

\begin{remark}\label{sec2 remark on factorization}
Assume $Q\in\bR[x,y]$ is a nonzero $(p,q)$-weighted homogeneous polynomial as in the above lemma. Then we can decompose $Q$ in $\mathbb{C}[x,y]$ as
$$Q(x,y)=cx^my^n\prod_{k=1}^N(x-\alpha_ky^{\eta}),~~~~~
\eta=\frac{p}{q},$$
with $c\in\bR\backslash\{0\}$, $\alpha_k\in\mathbb{C}\backslash\{0\}$.
This factorization is conjugate invariant in the following sense:

\noindent$\bullet$ For each $k$ and each $q$-th root $\epsilon$ of the unit, the factors $x-\alpha_ky^{\eta}$, $x-\overline{\alpha}_ky^{\eta}$ and $x-\epsilon\alpha_ky^{\eta}$
appear in the above factorization with the same multiplicity, where $\overline{\alpha}_k$ is the complex conjugate of $\alpha_k$.
\end{remark}

\section{Damped Oscillatory Integral Operators on $L^2$}
Assume $S\in\bR[x,y]$ is a $(p,q)$-weighted homogeneous polynomial. Then $S$ is of the form (\ref{Definition Polynomial phase}) with $\eta=\frac{p}{q}$. Without loss of generality, we may assume the Hessian $S_{xy}''$ is nonzero. By Lemma \ref{sec2 factorization Rxy} and Remark \ref{sec2 remark on factorization}, $S_{xy}''$ can be factorized as
\begin{equation}\label{Hessian of phase}
S_{xy}''(x,y)=c_0x^my^n\prod\limits_{i=1}^N(x-\alpha_iy^{\eta})
\end{equation}
with $c_0\in \bR \backslash\{0\}$, $\alpha_1,\alpha_2,\cdots,\alpha_N\in\mathbb{C}\backslash\{0\}$ and $\eta>0$. Since $c_0$ can be absorbed in the parameter $\lambda$, we may assume $c_0=1$. By choosing $s$ indices $i_1<i_2<\cdots <i_s$, we define a damping factor $D$ of form
\begin{equation}\label{damping factor}
D(x,y)=x^m(x-\alpha_{i_1}y^{\eta})(x-\alpha_{i_2}y^{\eta})
\cdots (x-\alpha_{i_s}y^{\eta}).
\end{equation}
For our purpose, we shall further assume that this factorization is conjugate invariant in the sense of Remark \ref{sec2 remark on factorization}. On the other hand, we see that $D$ belongs to $\bR[x,y]$ if and only if $D$ is conjugate invariant.

Let $W_z$ be the damped oscillatory integral operator
\begin{equation}\label{damping OIO}
W_zf(x)=\int_{\bR}e^{i\lambda S(x,y)}|D(x,y)|^z\varphi(x,y)f(y)dy,
\end{equation}
where $z\in \mathbb{C}$ lies in an appropriate strip.
Now we shall give the main result in this section.

\begin{theorem}\label{Theorem L2 damping estimate}
Let $S,D\in\bR[x,y]$ be defined as above with $\eta\geq 1$. If $W_z$ is defined as in {\rm (\ref{damping OIO})}, then there exists a constant $C=C(deg(S),\varphi)$ such that
\begin{equation}\label{damping decay estimate}
\|W_zf\|_{L^2}
\leq
C(1+|z|^2)\l(|\lambda|\prod_{k\notin \{i_1,\cdots,i_s\}}|\alpha_k|\r)^{-\gamma}\|f\|_{L^2},
~~~\gamma=\frac{1}{2(n+(N-s){\eta}+1)}
\end{equation}
for $z\in \mathbb{C}$ with real part
\begin{equation}\label{L2 damping exponent}
\Re(z)=\frac{m+s-n-(N-s)\eta}{2(n+(N-s)\eta+1)}\cdot\frac{1}{m+s}.
\end{equation}
More precisely, the constant $C$ can take the following form:
\begin{equation}\label{uniform constant}
C(deg(S))
\mathop{\sup}_{\Omega}
\sum_{k=0}^{2}
\left( \big(\delta^{}_{\Omega,h}(x) \big)^k |\partial^k_y\varphi(x,y)|
+
\big( \delta^{}_{\Omega,v}(y) \big)^k |\partial^k_x\varphi(x,y)|\right)
\end{equation}
for all $\varphi$ supported in a horizontally and vertically convex domain $\Omega$. Here $C(deg(S))$ is a constant depending only on the degree of $S$.
\end{theorem}
\begin{remark}\label{Remark Theorem L2 damping estimate}
The assumption $\eta\geq 1$ is crucial in our proof for technical reasons. If $\eta<1$, we shall instead change the role of $x$ and $y$. For earlier work related to damping estimates of oscillatory integral operators, we refer the reader to {\rm \cite{PS1997,pyang,ShiYan}} with non-uniform bounds.
\end{remark}
\begin{proof}
Choose a smooth nonnegative bump function $\Phi$ such that $\supp(\Phi)\subseteq [1/2,2]$ and $\sum_j\Phi(x/2^j)=1$ for $x>0$. Define $W_{j,k}$ as $W_z$ by insertion of $\Phi(x/2^j)\Phi(y/2^k)$ into the cut-off of $W_z$, i.e.,
\begin{equation}\label{definition of Ujk}
W_{j,k}f(x)=\int_{\bR}e^{i\lambda S(x,y)}|D(x,y)|^z\Phi(x/2^j)\Phi(y/2^k)\varphi(x,y)f(y)dy.
\end{equation}
Here we only consider the operator $W_z$ in the first quadrant. Estimates in other quadrants can be treated similarly. We assume $z$ has the real part in (\ref{L2 damping exponent}) and $|\alpha_1|\leq |\alpha_2|\leq \cdots \leq |\alpha_N|$ throughout this proof.
$~$\\

$\mathbf{Step~1.~~All~\alpha_{i_1},~\alpha_{i_2},~\cdots,~\alpha_{i_s}~ are~real~numbers.}$\\

We first prove the theorem under the additional assumption $\alpha_{i_t}\in\mathbb{R}$. One can see that our arguments are also applicable without essential change in presence of complex $\alpha_{i_t}$; see Step 2.\\

\textrm{$\mathbf{Case~(i)~~|\alpha_1|2^{(k-1)\eta}\geq 2^{j+2}}$}.\\

Assume first $m=0$. For fixed $k$, we define $W_k=\sum_{j}W_{j,k}$ with the summation taken over all $j$ satisfying $\mathbf{Case~(i)}$. Then $W_k$ is supported in the rectangle $R_k$: $|x|\leq |\alpha_1|2^{(k-1)\eta-1}$, $2^{k-1}\leq y \leq 2^{k+1}$. For $|k-l|\geq 2$, $W_k^{}W_l^{\ast}=0$. Thus it suffices to estimate $W_k^{\ast}W_l^{}$ to get almost orthogonality. Since $W_k^{\ast}W_l^{}$ and $W_l^{\ast}W_k^{}$ have the same operator norm on $L^2$, we may assume $k\geq l$. Observe that the Hessian of $S$ has uniform upper and lower bounds on an expanded rectangle $R_k^{\ast}$ of $R_k$. More precisely, we have
$$|S_{xy}''(x,y)|\approx 2^{kn}\l(\prod_{i=1}^N|\alpha_i|\r)2^{Nk{\eta}}$$
on the expanded rectangle $R_k^{\ast}:~|x|\leq |\alpha_1|2^{(k-1)\eta-1/2}, ~2^{k-1-\epsilon}\leq y \leq 2^{k+1+\epsilon}$ with $\epsilon>0$ satisfying $\epsilon\eta\leq \frac{1}{4}$. It is easily verified that $|x-\alpha_i y^{\eta}|\approx |\alpha_i|2^{k\eta}$ with bounds depending only on $\eta$. Moreover $S_{xy}''(x,y)$ does not change sign on all vertical line segments joining two points in $R_k^{\ast}$ and $R_l^{\ast}$. We can apply Lemma \ref{Lemma almost orth van der Corput} to get
\begin{equation*}
\|W_k^{\ast} W_l\|
\leq
C\l[  |\lambda|  2^{kn}\l(\prod_{i=1}^N|\alpha_i|\r)2^{Nk{\eta}}     \r]^{-1}
\l( 2^{ks\eta} \prod_{t=1}^{s}|\alpha_{i_t}|  \r)^{\Re(z)}
\l( 2^{ls\eta} \prod_{t=1}^{s}|\alpha_{i_t}| \r)^{\Re(z)}.
\end{equation*}
By the size estimate for each $W_k$, we deduce from $\|W_k^{\ast} W_l^{}\|\leq \|W_k^{\ast}\|\| W_l^{}\|$ that
\begin{equation*}
\|W_k^{\ast} W_l\|
\leq
C \l( 2^{ks\eta} \prod_{t=1}^{s}|\alpha_{i_t}|  \r)^{\Re(z)}
( |\alpha_1|  2^{k{\eta}} )^{1/2}  2^{k/2}
\l( 2^{ls\eta} \prod_{t=1}^{s}|\alpha_{i_t}| \r)^{\Re(z)}
( |\alpha_1|  2^{l\eta} )^{1/2}  2^{l/2}.
\end{equation*}
Taking a convex combination of the above two estimates, we obtain
\begin{eqnarray}\label{Proof WjWk estimate}
\|W_k^{\ast} W_l\|
& \leq &
C\l[  |\lambda|  2^{kn}\l(\prod_{i=1}^N|\alpha_i|\r)2^{Nk{\eta}}     \r]^{-\theta}
\l( 2^{ks\eta} \prod_{t=1}^{s}|\alpha_{i_t}|  \r)^{\Re(z)}
\l( 2^{ls\eta} \prod_{t=1}^{s}|\alpha_{i_t}| \r)^{\Re(z)} \nonumber \\
& &~~
( |\alpha_1|  2^{k{\eta}} )^{(1-\theta)/2}  2^{k(1-\theta)/2}
( |\alpha_1|  2^{l\eta} )^{(1-\theta)/2}  2^{l(1-\theta)/2}.
\end{eqnarray}
Let $\theta=2\gamma$. The following relation between $\theta$ and $\Re(z)$ as in (\ref{L2 damping exponent}) will be used frequently throughout the rest of the proof.
\begin{equation}\label{relation theta and gamma}
\Re(z)=\frac{\theta}{2}-\frac{1-\theta}{2(m+s)},
~~~~~~
1-\theta=\Big( n+(N-s)\eta \Big)\theta.
\end{equation}
We collect terms involving $\alpha_i$ in the above inequality and obtain
\begin{eqnarray*}
\l(  \prod_{i=1}^N|\alpha_i|  \r)^{-\theta}
\l(  \prod_{t=1}^{s}|\alpha_{i_t}|  \r)^{2\Re(z)}|\alpha_1|^{1-\theta}
&\leq &
\l(  \prod_{i=1}^N|\alpha_i|  \r)^{-\theta}
\l(  \prod_{t=1}^{s}|\alpha_{i_t}|  \r)^{2\Re(z)+(1-\theta)/s}\\
&\leq&
\l(  \prod_{k\notin\{i_1,\cdots,i_s\}}|\alpha_i|  \r)^{-\theta},
\end{eqnarray*}
where we have used the assumptions $|\alpha_1|\leq |\alpha_2|\leq \cdots |\alpha_N|$, $m=0$ and the equation (\ref{relation theta and gamma}) to obtain
$$2\Re(z)+\frac{1-\theta}{s}
=2\left(\theta-\frac{1-\theta}{2s} \right)
+
\frac{1-\theta}{s}
=\theta.$$
On the right side of (\ref{Proof WjWk estimate}), we add all exponents of $2^k$ and obtain
\begin{eqnarray*}
& &-(n+N{\eta})\theta+s\eta\Re(z)
+\frac{1-\theta}{2}\eta+\frac{1-\theta}{2}\\
&=&
-(1-\theta)-s\eta\theta+\frac{s\eta}{2}\theta
-\frac{1-\theta}{2}\eta
+\frac{1-\theta}{2}\eta
+\frac{1-\theta}{2}\\
&=&
-\frac{1-\theta}{2}-\frac{s\eta}{2}\theta.
\end{eqnarray*}
The exponent of $2^l$ in (\ref{Proof WjWk estimate}) equals
\begin{eqnarray*}
s\eta\Re(z)+\frac{1-\theta}{2}\eta+ \frac{1-\theta}2
&=&
s\eta\left( \frac{\theta}{2}-\frac{1-\theta}{2s} \right)
+\frac{1-\theta}{2}\eta
+\frac{1-\theta}{2}\\
&=&
\frac{s\eta}{2}\theta
+\frac{1-\theta}{2}
\end{eqnarray*}
Then the inequality (\ref{Proof WjWk estimate}) becomes
\begin{equation*}
\|W_k^{\ast} W_l\|
\leq
C (1+|z|^2) |\lambda|^{-\theta}\l( \prod_{k\notin\{i_1,\cdots,i_s\}} |\alpha_k|  \r)^{-\theta} 2^{-|k-l|\delta}
\end{equation*}
with $\theta=2\gamma$ and $\delta=\frac{s\eta}{2}\theta
+\frac{1-\theta}{2}>0$.

Now we address \textrm{$\mathbf{Case~(i)}$} with $m>0$. Observe that $W_{j,k}$ is supported in the rectangle $R_{j,k}$: $x\in [2^{j-1},2^{j+1}]$, $y\in [2^{k-1},2^{k+1}]$. For some $\epsilon>0$ satisfying $\epsilon\leq \frac{1}{4\eta}$, we define the expanded rectangle
$R^{\ast}_{j,k}:~2^{j-3/2}\leq x \leq 2^{j+3/2},~2^{k-1-\epsilon}\leq y \leq 2^{k+1+\epsilon}.$ Then we can verify that all assumptions in Lemma \ref{Lemma almost orth van der Corput} are true. Since $W_{j,k}^{}W_{j',k'}^{\ast}=0$ for $|k-k'|\geq 2$, we may assume $|k-k'|\leq 1$. By Lemma \ref{Lemma almost orth van der Corput}, $W_{j,k}^{}W_{j',k'}^{\ast}$ satisfies, assuming $j\geq j'$,
\begin{eqnarray*}
\|W_{j,k}^{}W_{j',k'}^{\ast}\|
&\leq&
C\l[ |\lambda| 2^{jm} 2^{kn} \l( \prod_{i=1}^N|\alpha_i| \r) 2^{kN{\eta}}  \r]^{-1}
\l[ 2^{jm} \l( \prod_{u=1}^s|\alpha_{i_u}| \r) 2^{ks\eta}   \r]^{\Re(z)}\times\\
& &~~~\l[ 2^{j'm} \l( \prod_{u=1}^s|\alpha_{i_u}| \r) 2^{k's\eta}   \r]^{\Re(z)}.
\end{eqnarray*}
Since $\|W_{j,k}^{}W_{j',k'}^{\ast}\| \leq  \|W_{j,k}\|  \|W_{j',k'}^{\ast}\|$, it follows from size estimates for each operator that
\begin{eqnarray*}
\|W_{j,k}^{}W_{j',k'}^{\ast}\|
&\lesssim&
\l[ 2^{jm} \l( \prod_{u=1}^s|\alpha_{i_u}| \r) 2^{ks\eta}   \r]^{\Re(z)}2^{j/2} 2^{k/2}
\l[ 2^{j'm} \l( \prod_{u=1}^s|\alpha_{i_u}| \r) 2^{k's\eta}   \r]^{\Re(z)} 2^{j'/2} 2^{k'/2}.
\end{eqnarray*}
We take a convex combination of the above two inequalities and obtain
\begin{eqnarray}\label{Oscillation esti WjkWjk}
\|W_{j,k}^{}W_{j',k'}^{\ast}\|
&\lesssim&
\l[ |\lambda| 2^{jm} 2^{kn} \l( \prod_{i=1}^N|\alpha_i| \r) 2^{kN{\eta}}  \r]^{-\theta}
\l[ 2^{jm} \l( \prod_{u=1}^s|\alpha_{i_u}| \r) 2^{ks\eta}   \r]^{\Re(z)}2^{j(1-\theta)/2} \times \nonumber \\
& &~2^{k(1-\theta)/2}
\l[ 2^{j'm} \l( \prod_{u=1}^s|\alpha_{i_u}| \r) 2^{k's\eta}   \r]^{\Re(z)} 2^{j'(1-\theta)/2} 2^{k'(1-\theta)/2}.
\end{eqnarray}
Put $\theta=2\gamma$. Since $|k-k'|\leq 1$, we can identify $k'$ with $k$ in this estimate. Then the exponent of $2^k$ is equal to
\begin{eqnarray*}
\tau
&:=&
-(n+N{\eta})\theta+2s\eta\Re(z)+1-\theta\\
&=&
-(1-\theta)-s\eta\theta
+2s\eta\left( \frac{\theta}{2}-\frac{1-\theta}{2(m+s)}\right)
+1-\theta\\
&=&
\frac{s\eta}{(m+s)}(1-\theta).
\end{eqnarray*}
Recall that $|\alpha_1|2^{(k-1)\eta}\geq 2^{j+2}$ in \textrm{$\mathbf{Case~(i)}$}. Inserting $2^{k\tau}$, $2^{k'\tau}\lesssim (|\alpha_1|^{-1}2^j)^{\tau/\eta}$ into (\ref{Oscillation esti WjkWjk}) and collecting terms concerning $\alpha_i$, we obtain
\begin{eqnarray*}
\l(  \prod_{i=1}^N|\alpha_i| \r)^{-\theta}
\l(  \prod_{u=1}^s|\alpha_{i_u}|  \r)^{2\Re(z)}
|\alpha_1|^{-\tau/\eta}
& \leq &
\l(  \prod_{i=1}^N|\alpha_i| \r)^{-\theta}
\l(  \prod_{u=1}^s|\alpha_{i_u}|  \r)^{2\Re(z)}
\l(  \prod_{u=1}^s|\alpha_{i_u}|  \r)^{-\frac{\tau}{s\eta}}\\
&\leq &
\l(  \prod_{i=1}^N|\alpha_i| \r)^{-\theta}
\l(  \prod_{u=1}^s|\alpha_{i_u}|  \r)^{\theta}\\
&\leq &
\l(  \prod_{k\notin \{i_1,i_2,\cdots,i_s\} }|\alpha_k| \r)^{-\theta},
\end{eqnarray*}
where we have used the equality
\begin{eqnarray*}
2\Re(z)-\frac{\tau}{s\eta}
&=&
\theta-\frac{1-\theta}{m+s}
+\frac{1-\theta}{m+s}
=\theta.
\end{eqnarray*}
The exponent of $2^j$ becomes, after inserting $2^{k\tau}$, $2^{k'\tau}\lesssim (|\alpha_1|^{-1}2^j)^{\tau/\eta}$ into (\ref{Oscillation esti WjkWjk}),
\begin{eqnarray*}
\frac{\tau}{\eta}-m\theta+m\Re(z)+\frac{1-\theta}{2}
&=&
-\frac{s}{m+s}(1-\theta)-m\theta+\frac{m}{2}\theta
-\frac{m}{2(m+s)}(1-\theta)
+\frac{1}{2}(1-\theta)\\
&=&
-\frac{s}{2(m+s)}(1-\theta)-\frac{m}{2}\theta.
\end{eqnarray*}
By direct calculation, we see that the exponent of $2^{j'}$ in (\ref{Oscillation esti WjkWjk}) is
\begin{eqnarray*}
m\Re(z)+\frac{1-\theta}{2}
&=&
\frac{m}{2}\theta
-\frac{m}{2(m+s)}(1-\theta)
+\frac{1-\theta}{2}\\
&=&
\frac{s}{2(m+s)}(1-\theta)+\frac{m}{2}\theta.
\end{eqnarray*}
Combining these estimates, we get
\begin{equation*}
\|W_{j,k}^{}W_{j',k'}^{\ast}\|
\lesssim (1+|z|^2)
|\lambda|^{-\theta}
\l(  \prod_{k\notin \{i_1,i_2,\cdots,i_s\}} |\alpha_k| \r)^{-\theta}
2^{-|j-j'|\delta}
\end{equation*}
with $\theta=2\gamma$ and $\delta$ defined by
$$\delta=\frac{s}{2(m+s)}\cdot(1-\theta)+\frac{m}{2}\cdot\theta>0.$$

By the same argument as above, we can show that $W_{j,k}^{\ast}W_{j',k'}^{}$ satisfies a similar estimate. We omit the details here.
$ $\\

\textrm{$\mathbf{Case~(ii)~~2^j \geq |\alpha_N|2^{(k+1)\eta+2}}$}.\\

As in \textrm{$\mathbf{Case~(i)}$}, the argument is slightly different depending on whether $n=0$. First consider the case $n=0$. For each $j$, define $W_j=\sum_{k}W_{j,k}$ with the summation taken over all $k$ satisfying the condition in \textrm{$\mathbf{Case~(ii)}$}. Then $W_j$ is supported in the rectangle $R_j:$ $2^{j-1}\leq x \leq 2^{j+1}$, $0\leq y \leq (|\alpha_N|^{-1}2^{j-2})^{1/\eta}$. Since $W_j^{\ast}W_{j'}^{}=0$ for $|j-j|\geq 2$, it is enough to estimate $W_j^{}W_{j'}^{\ast}$. Observe that $\|W_j^{}W_{j'}^{\ast}\|=\|W_{j'}^{}W_{j}^{\ast}\|$. We may assume $j\geq j'$ in the following proof.

It should be pointed out that the almost orthogonality estimate in Lemma \ref{Lemma almost orth van der Corput} is not applicable here since $D(x,y)$ may not be a polynomial type function in $y$ on $0\leq y \leq (|\alpha_N|^{-1}2^{j-2})^{1/\eta}$. However, if $s=N$ then $D$ is a polynomial in $\bR^2$. Assume first $s<N$. Since both $D(x,y)$ and $|D(x,y)|^z$ satisfy the estimate (\ref{sec2 esti of poly type func}) with respect to $x\in [2^{j-1},2^{j+1}]$ for $N_F=2$, we can apply Lemma \ref{operator van der Corput} to obtain the following oscillation estimate
\begin{equation*}
\|W_j\|
\leq
C(1+|z|^2) \left( |\lambda| 2^{jm} 2^{jN} \right)^{-1/2}
\left( 2^{jm} 2^{js} \right)^{\Re(z)}.
\end{equation*}
By the Schur test, we obtain the size estimate
\begin{equation*}
\|W_j\|
\leq
C \left( 2^{jm} 2^{js} \right)^{\Re(z)}
2^{j/2} \left( |\alpha_N|^{-1} 2^{j-2} \right)^{\frac{1}{2\eta}}.
\end{equation*}
In the oscillation estimate, the exponent of $2^j$ is
\begin{eqnarray*}
-\frac{m+N}{2}+(m+s)\Re(z)
&=&-\frac{m+N}{2}
+
(m+s)\l(\frac{\theta}{2}-\frac{1-\theta}{2(m+s)}\r)\\
&=&-\frac{m+N}{2}-\frac{1-\theta}{2}+(m+s)\frac{\theta}{2}<0
\end{eqnarray*}
with $\theta=2\gamma$. Here the last inequality holds since $s<N$ and $0\leq \theta \leq 1$. On the other hand, the exponent of $2^j$ in the size estimate is
\begin{eqnarray*}
(m+s)\Re(z)+\frac{1}{2}+\frac{1}{2\eta}
&=&
(m+s)\left(\frac{\theta}{2}-\frac{1-\theta}{2(m+s)} \right)
+\frac{1}{2}+\frac{1}{2\eta}\\
&=&
(m+s)\frac{\theta}{2}+\frac{\theta}{2}+\frac{1}{2\eta}>0.
\end{eqnarray*}
Balancing the oscillation and size estimates, we will obtain
\begin{equation*}
\Big\| \sum_jW_j \Big\|
\leq
\sum_j\Big\| W_j \Big\|
\leq
C(1+|z|^2)\left( |\lambda| \prod_{k\notin\{i_1,i_2,\cdots,i_s\} } |\alpha_k|\right)^{-\gamma}.
\end{equation*}

Now we turn to the case when $D(x,y)$ is a polynomial type function in $y$ on $[0,(|\alpha_N|^{-1}2^{j-2})^{1/\eta}]$; for example $s=N$. One can also verify the assumptions in Lemma \ref{Lemma almost orth van der Corput}. Hence $W_j^{}W_{j'}^{\ast}$ satisfies
\begin{eqnarray*}
\|W_j^{}W_{j'}^{\ast}\|
&\leq&
C\Big[ |\lambda|  2^{jm}  2^{jN} \Big]^{-1}
\Big( 2^{jm} 2^{js} \Big)^{\Re(z)} \Big(  2^{j'm} 2^{j's}  \Big)^{\Re(z)}.
\end{eqnarray*}
By the Schur test, $\|W_j^{}W_{j'}^{\ast}\|$ is bounded by
\begin{eqnarray*}
\|W_j^{}\|  \|W_{j'}\|
\lesssim
\l[ \Big(  2^{jm} 2^{js}  \Big)^{\Re(z)} 2^{j/2} \Big(|\alpha_N|^{-1}2^{j-2}  \Big)^{\frac1{2\eta}} \r]
\l[  \Big(  2^{j'm} 2^{j's}  \Big)^{\Re(z)} 2^{j'/2}
\Big(  |\alpha_N|^{-1}2^{j'-2}  \Big)^{\frac1{2\eta}} \r].
\end{eqnarray*}
A convex combination of these two estimates yields, setting $\theta=2\gamma$,
\begin{eqnarray}\label{Estimate WjWj}
\|W_j^{}W_{j'}^{\ast}\|
& \lesssim &
\Big[ |\lambda|  2^{jm}  2^{jN} \Big]^{-\theta}
\Big(  2^{jm} 2^{js}  \Big)^{\Re(z)} \Big(  2^{j'm} 2^{j's} \Big)^{\Re(z)}\times
\nonumber  \\
& &
~2^{j(1-\theta)/2} \Big(  |\alpha_N|^{-1}2^{j-2}  \Big)^{\frac{1-\theta}{2\eta}}
2^{j'(1-\theta)/2} \Big(  |\alpha_N|^{-1}2^{j'-2}  \Big)^{\frac{1-\theta}{2\eta}}.
\end{eqnarray}
Recall that $n=0$. In the above estimate, the exponent of $2^j$ equals
\begin{eqnarray*}
& &-(m+N)\theta+(m+s)\Re(z)+\frac{1-\theta}{2}+\frac{1-\theta}{2\eta}\\
&=&
-(m+N)\theta
+\frac{m+s}{2}\theta
-\frac{1}{2}(1-\theta)
+\frac{1}{2}(1-\theta)
+\frac{1}{2\eta}(1-\theta)\\
&=&
-(m+N)\theta
+\frac{m+s}{2}\theta
+\frac{N-s}{2}\theta\\
&=&
-\frac{m+N}{2}\theta,
\end{eqnarray*}
and the exponent of $2^{j'}$ is given by
\begin{eqnarray*}
(m+s)\Re(z)+\frac{1-\theta}{2}+\frac{1-\theta}{2\eta}
&=&
(m+s)\left( \frac{\theta}{2}-\frac{1-\theta}{2(m+s)}\right)
+\frac{1-\theta}{2}+\frac{1-\theta}{2\eta}\\
&=&
\frac{m+N}{2}\theta.
\end{eqnarray*}
On the other hand, we see that $(1-\theta)/2\eta=(N-s)\theta$ and \begin{equation*}
|\alpha_N|^{-\frac{1-\theta}{2\eta}}\leq \l(\prod_{k\notin \{i_1,\cdots,i_s\}} |\alpha_k|\r)^{-\theta}.
\end{equation*}
All of these results imply that
\begin{eqnarray*}
\|W_j^{}W_{j'}^{\ast}\|
& \leq &
C (1+|z|^2)|\lambda|^{-\theta} \l( \prod_{k\notin \{i_1,i_2,\cdots,i_s\}} |\alpha_k| \r)^{-\theta} 2^{-|j-j'|\delta}
\end{eqnarray*}
with $\delta=\frac{m+N}{2[(N-s)\eta+1]}\geq 0$. If $\delta=0$, then $m=N=0$. Recall that we assumed $n=0$. This implies that $S_{xy}''=1$ and the $L^2$ estimate (\ref{damping decay estimate}) is a consequence of Plancherel's theorem. Here we need only to consider the case $\delta>0$. Then the desired estimate in \textrm{$\mathbf{Case~(ii)}$} follows from the almost orthogonality principle.

Now we turn to the estimate for $n>0$ in \textrm{$\mathbf{Case~(ii)}$}. For $|k-k'|\geq 2$, $W_{j,k}^{}W_{j',k'}^{\ast}=0$. Assume $|k-k'|\leq 1$. The operator $W_{j,k}$ is supported in the rectangle $R_{j,k}$: $2^{j-1}\leq x \leq 2^{j+1}$, $2^{k-1}\leq y \leq 2^{k+1}$. One can verify that $S_{xy}''$ is comparable to a fixed value on an expanded region
$R_{j,k}^{\ast}$ which can be defined as in $\mathbf{Case~(i)}$. Other assumptions in Lemma \ref{Lemma almost orth van der Corput} are also true.

Assume $j\geq j'$. By Lemma \ref{Lemma almost orth van der Corput}, there holds
\begin{equation*}
\|W_{j,k}^{}W_{j',k'}^{\ast}\|
\leq
C \l[ |\lambda| 2^{jm} 2^{kn} 2^{jN}  \r]^{-1}
\l[ 2^{jm} 2^{js} \r]^{\Re(z)} \l[ 2^{j'm} 2^{j's}  \r]^{\Re)(z)}.
\end{equation*}
The size estimate is, using $\|W_{j,k}^{}W_{j',k'}^{\ast}\|\leq \|W_{j,k}\|\|W_{j',k'}\|$,
\begin{equation*}
\|W_{j,k}^{}W_{j',k'}^{\ast}\|
\leq
C \l[ 2^{jm} 2^{js} \r]^{\Re(z)} 2^{j/2} 2^{k/2}
\l[ 2^{j'm} 2^{j's}  \r]^{\Re)(z)} 2^{j'/2} 2^{k'/2}.
\end{equation*}
We take a convex combination with $\theta=2\gamma$ and obtain
\begin{eqnarray*}
\|W_{j,k}^{}W_{j',k'}^{\ast}\|
& \lesssim &
\l[ |\lambda| 2^{jm} 2^{kn} 2^{jN}  \r]^{-\theta}
\l[ 2^{jm} 2^{js} \r]^{\Re(z)} 2^{j(1-\theta)/2} 2^{k(1-\theta)/2} \times \\
& & \l[ 2^{j'm} 2^{j's}  \r]^{\Re(z)} 2^{j'(1-\theta)/2} 2^{k'(1-\theta)/2}.
\end{eqnarray*}
All terms and their exponents are given as follows:
\begin{equation*}
2^j:~~-(m+N)\theta+(m+s)\Re(z)+\frac{1-\theta}{2},~~~~~~~
2^{j'}:~~(m+s)\Re(z)+\frac{1-\theta}{2};
\end{equation*}
\begin{equation*}
2^{k}:~~-n\theta+\frac{1-\theta}{2},~~~~~~~
2^{k'}:~~\frac{1-\theta}{2}.
\end{equation*}
Since $|k-k'|\leq 1$ and $2^k$, $2^{k'}\lesssim (|\alpha_N|^{-1}2^{j'})^{1/\eta}$, the product of two terms involving $2^k$ and $2^{k'}$ is bounded by $C(|\alpha_N|^{-1}2^j)^{(-n\theta+1-\theta)/\eta}$. Then the new exponent of $2^{j'}$ becomes
\begin{eqnarray*}
& &(m+s)\Re(z)+\frac{1-\theta}{2}+\frac{1}{\eta}\l( -n\theta +1-\theta \r)\\
&=& \frac{(m+s)}{2}\theta-\frac{1-\theta}{2}+\frac{1-\theta}{2}
-\frac{n}{\eta}\theta+\frac{n}{\eta}\theta+(N-s)\theta\\
&=& \l(  N+\frac{m}{2}-\frac{s}{2}  \r)\theta.
\end{eqnarray*}
To obtain the desired almost orthogonality, we shall verify the exponent of $2^j$ equals $-(N+\frac{m}{2}-s)\theta$. In fact, this is true. By direct computation,
\begin{eqnarray*}
& &-(m+N)\theta+(m+s)\Re(z)+\frac{1-\theta}{2}\\
&=&-(m+N)\theta+(m+s)\frac{\theta}{2}-\frac{1-\theta}{2}+\frac{1-\theta}{2}\\
&=&-\l(  N+\frac{m}{2}-\frac{s}{2}  \r)\theta.
\end{eqnarray*}
On the other hand, the bound concerning $\alpha_i$ is equal to $(|\alpha_N|^{-1})^{(-n\theta+1-\theta)/\eta}$. Observe that $(-n\theta+1-\theta)/\eta=(N-s)\theta$. By the assumption $|\alpha_1|\leq \cdots \leq |\alpha_N|$, we see that
$(|\alpha_N|^{-1})^{(-n\theta+1-\theta)/\eta}$ is less than or equal to $\l(\prod_{k\notin \{i_1,\cdots,i_s\}}|\alpha_k|\r)^{-\theta}.$ Combining this fact with above results, we obtain
\begin{equation*}
\|W_{j,k}^{}W_{j',k'}^{\ast}\|
\leq
C\l(|\lambda| \prod_{k\notin \{i_1,\cdots,i_k\}} |\alpha_k|\r)^{-\theta}
2^{-|k-k'|\delta}
\end{equation*}
with $\delta$ given by
\begin{equation*}
\delta=\l(  N+\frac{m}{2}-\frac{s}{2}  \r)\theta.
\end{equation*}

\noindent A similar argument also shows that $W_{j,k}^{\ast}W_{j',k'}^{}$ satisfies the same estimate.\\

\textrm{$\mathbf{Case~(iii)~~|\alpha_1|2^{(k-1)\eta-2}< 2^j < |\alpha_N|2^{(k+1)\eta+2}}$}.\\

For some large number $N_0=N_0(\eta)>0$, consider the intervals $I_i=(|\alpha_i|2^{-2N_0},|\alpha_i|2^{2N_0})$ with $1\leq i \leq N$. Since $|\alpha_1|\leq |\alpha_2|\leq \cdots |\alpha_N|$, we see that $I_i\cap I_{i+1}=\emptyset$ if and only if $|\alpha_{i+1}|/|\alpha_{i}|\geq 2^{4N_0}$.
For clarity, define
\begin{equation}\label{sec4 exceptional index set}
\Big\{t_1,t_2,\cdots,t_a \Big\}:=\Big\{1\leq i \leq N: |\alpha_{i+1}|/|\alpha_{i}|\geq 2^{4N_0} \Big\}.
\end{equation}

\noindent Of course, this set may very well be empty. We shall divide our proof into two subcases.\\

\textrm{$\mathbf{Subcase~(a)~~|\alpha_{t_r}|2^{(k+1)\eta+2}\leq 2^j \leq |\alpha_{t_r+1}|2^{(k-1)\eta-2}~for~some~1\leq r \leq a.}$}\\

Define $\Xi_1=\{1,2,\cdots,t_r\}$ and $\Xi_2=\{t_r+1,t_r+2,\cdots,N\}$. Then on the support of $W_{j,k}$, $|S_{xy}''(x,y)|$ is comparable to
\begin{eqnarray*}
|S_{xy}''(x,y)|
&\approx&
2^{jm} 2^{kn} \l(  \prod_{t\in \Xi_1} 2^j \r)
\l(  \prod_{t\in  \Xi_2} |\alpha_t|2^{k{\eta}} \r)
\end{eqnarray*}
and the damping factor $D(x,y)$ has size
\begin{eqnarray*}
|D(x,y)|
&\approx&
2^{jm} \l(  \prod_{t\in \Xi_1 \cap \Theta  } 2^j \r)
\l(  \prod_{t\in  \Xi_2 \cap \Theta  } |\alpha_t|2^{k{\eta}} \r)
\end{eqnarray*}
where $\Theta=\{i_1,i_2,\cdots,i_s\}$.
Since $W_{j,k}^{}W_{j',k'}^{\ast}=0$ for $|k-k'|\geq 2$, we assume now $|k-k'|\leq 1$ and $j\geq j'$ without loss of generality. By Lemma \ref{Lemma almost orth van der Corput}, there holds
\begin{eqnarray*}
\|W_{j,k}^{}W_{j',k'}^{\ast}\|
&\lesssim &\l[ |\lambda| 2^{jm} 2^{kn}  2^{j|\Xi_1|}
\l(  \prod_{t\in  \Xi_2} |\alpha_t|2^{k\eta} \r)  \r]^{-1}
\l[  2^{jm} 2^{j|\Xi_1\cap\Theta|}
\l(  \prod_{t\in  \Xi_2 \cap \Theta  } |\alpha_t|2^{k\eta} \r)  \r]^{\Re(z)} \times \\
& &~\l[  2^{j'm} 2^{j'|\Xi_1\cap\Theta|}
\l(  \prod_{t\in  \Xi_2 \cap \Theta  } |\alpha_t|2^{k'\eta} \r)  \r]^{\Re(z)}.
\end{eqnarray*}
By the size estimate of each $W_{j,k}$, we have, using $\|W_{j,k}^{}W_{j',k'}^{\ast}\|\leq \|W_{j,k}^{}\| \|W_{j',k'}^{\ast}\|$,
\begin{eqnarray*}
\|W_{j,k}^{}W_{j',k'}^{\ast}\|
&\lesssim&
\l[  2^{jm} 2^{j|\Xi_1\cap\Theta|}
\l(  \prod_{t\in  \Xi_2 \cap \Theta  } |\alpha_t|2^{k\eta} \r)  \r]^{\Re(z)} 2^{j/2} 2^{k/2}  \times \\
& &\l[  2^{j'm} 2^{j'|\Xi_1\cap\Theta|}
\l(  \prod_{t\in  \Xi_2 \cap \Theta  } |\alpha_t|2^{k'\eta} \r)  \r]^{\Re(z)} 2^{j'/2} 2^{k'/2}.
\end{eqnarray*}
With $\theta=2\gamma$, a convex combination of the above oscillation and size estimates yields
\begin{equation}\label{Estiamte UjkUjk}
\|W_{j,k}^{}W_{j',k'}^{\ast}\|
\leq
C |\lambda|^{-\theta}2^{ja_j} 2^{kb_k} 2^{ j'a'_{j'} } 2^{ k' b'_{k'}  },
\end{equation}
where the above exponents are given as follows:
\begin{eqnarray*}
& &a_j=
-(m+|\Xi_1|)\theta+(m+|\Xi_1\cap\Theta|)\Re(z)
+\frac{1-\theta}{2},\\
& &a'_{j'}=
(m+|\Xi_1\cap\Theta|)\Re(z)+\frac{1-\theta}{2},\\
& &b_k=
-n\theta-|\Xi_2|\theta \eta+|\Xi_2\cap\Theta|\Re(z)\eta+\frac{1-\theta}{2},\\
& &b'_{k'}=
|\Xi_2\cap\Theta|\Re(z)\eta+\frac{1-\theta}{2}.
\end{eqnarray*}
By the assumption $|k-k'|\leq 1$, we see that
\begin{equation}\label{Bound 2kbk2kbk}
2^{kb_k}2^{k'b'_{k'}}\leq C 2^{k(b_k+b_{k'}')}
=
C2^{k[2|\Theta\cap\Xi_2|\Re(z)\eta+1-\theta-|\Xi_2|\eta\theta
-n\theta]}.
\end{equation}
Define $s_1=|\Xi_1\cap\Theta|$ and $s_2=|\Xi_2\cap\Theta|$. Then $s=s_1+s_2$. Using $1-\theta=[n+(N-s)\eta]\theta$,
\begin{eqnarray*}
2|\Theta\cap\Xi_2|\Re(z)\eta+1-\theta-|\Xi_2|\eta\theta-n\theta
&=&
2s_2\eta\Re(z)+(1-\theta)-(N-t_r)\eta\theta-n\theta\\
&=&
2s_2\eta\Re(z)+(1-\theta)-(N-s)\eta\theta-(s-t_r)\eta\theta-n\theta\\
&=&
2s_2\eta\Re(z)+(t_r-s)\eta\theta\\
&=&
-\eta(\theta-2\Re(z))s_2+(t_r-s_1)\eta\theta.
\end{eqnarray*}
Since $|\alpha_{t_r}|2^{k\eta}\lesssim 2^j, 2^{j'} \lesssim |\alpha_{t_r+1}|2^{k\eta}$ and $|k-k'|\leq 1$, the right side of (\ref{Bound 2kbk2kbk}) can be divided into two terms, and each term can treated as follows:
\begin{eqnarray}\label{Bound for 2j and 2k}
2^{k(t_r-s_1)\eta\theta}
\lesssim
\l( \frac{1}{|\alpha_{t_r}|} 2^{j'} \r)^{(t_r-s_1)\theta}~~~\textrm{and}~~~~
2^{-k\eta(\theta-2\Re(z))s_2}
\lesssim
\l( \frac{1}{|\alpha_{t_{r}+1}|} 2^{j} \r)^{-(\theta-2\Re(z))s_2}.
\end{eqnarray}
Substituting these estimates into (\ref{Estiamte UjkUjk}) and using $\Re(z)=\frac{\theta}{2}-\frac{1}{2(m+s)}(1-\theta)$, we see that the new exponent of $2^j$ is
\begin{eqnarray*}
-(\theta-2\Re(z))s_2+a_j
&=& -(\theta-2\Re(z))s_2-(m+t_r)\theta+(m+s_1)\Re(z)
+\frac{1-\theta}{2}\\
&=&
-\frac{s_2(1-\theta)}{m+s}-(m+t_r)\theta+
\frac{m+s_1}{2}\theta-\frac{m+s_1}{2(m+s)}(1-\theta)
+\frac{1-\theta}{2}\\
&=&
-\left( \frac{m}{2}+t_r-\frac{s_1}{2} \right)\theta
-\frac{s_2}{2(m+s)}(1-\theta).
\end{eqnarray*}
With insertion of (\ref{Bound for 2j and 2k}) into (\ref{Estiamte UjkUjk}), it is clear that the new exponent of $2^{j'}$ is
\begin{eqnarray*}
(t_r-s_1)\theta+a_{j'}'
&=&
(t_r-s_1)\theta+\frac{m+s_1}{2}\theta
-\frac{m+s_1}{2(m+s)}(1-\theta)
+\frac{1-\theta}{2}\\
&=&
\left( \frac{m}{2}+t_r-\frac{s_1}{2}  \right)\theta
+\frac{s_2}{2(m+s)}(1-\theta).
\end{eqnarray*}
Combining the above estimates, we obtain
\begin{equation*}
\|W_{j,k}^{}W_{j',k'}^{\ast}\|
\leq
C(1+|z|^2)^2|\lambda|^{-\theta} 2^{-|j-j'|\delta}
\end{equation*}
with $\delta$ given by
\begin{eqnarray*}
\delta
&=&
\left( \frac{m}{2}+t_r-\frac{s_1}{2}  \right)\theta
+\frac{s_2}{2(m+s)}(1-\theta).
\end{eqnarray*}
Note that if $\{t_1,\cdots,t_a\}=\emptyset$ then we need only to consider $\mathbf{Subcase}$ (b) below. Now $t_r\geq 1$. It follows immediately that $\delta>0$. The constant $C$ is bounded by a constant multiple of $(\prod_{k\notin\Theta}|\alpha_k|)^{-\theta}$. Its dependence on $\alpha_i$ comes from (\ref{Estiamte UjkUjk}) and (\ref{Bound for 2j and 2k}), and we can verify this
claim as follows.
\begin{eqnarray*}
\l( \prod_{\Xi_2} |\alpha_t| \r)^{-\theta}
\l( \prod_{\Xi_2\cap\Theta} |\alpha_t| \r)^{2\Re(z)}
\l( \frac{1}{|\alpha_{t_r}|}  \r)^{(t_r-s_1)\theta}
\l( \frac{1}{|\alpha_{t_{r+1}}|} \r)^{-(\theta-2\Re(z))s_2}
\leq
\l(  \prod_{t\notin\Theta}|\alpha_t|  \r)^{-\theta}
\end{eqnarray*}
where we have used the following two inequalities:
$$|\alpha_{t_r}|^{-(t_r-s_1)\theta}
\leq
\left( \prod_{t\notin \Xi_1\cap\Theta} |\alpha_t| \right)^{-\theta}
~~{\textrm {and}}~~~
|\alpha_{t_{r}+1}|^{(\theta-2\Re(z))s_2}
\leq \left(\prod_{t\in \Xi_2\cap\Theta} |\alpha_t|\right)^{\theta-2\Re(z)}.$$
Therefore we have obtained the desired estimate in \textrm{$\mathbf{Subcase~(a)}$}.\\

\textrm{$\mathbf{Subcase~(b)~~|\alpha_{t_r+1}|2^{(k-1)\eta-2}< 2^j < |\alpha_{t_{r+1}}|2^{(k+1)\eta+2}~for~some~0\leq r \leq a.}$}\\
Here $t_0=0$. Let $G_0=\{t_r+1,t_r+2,\cdots,t_{r+1}\}$ and $\Theta=\{i_1,i_2,\cdots,i_s\}$. Choose $e_0\in G_0$ arbitrarily. For clarity, set $e_0=t_r+1$, the least integer in $G_0$. Define $W_{j,k,l_0}^{\sigma_0}$ as $W_{j,k}$, but with insertion of $\Phi\l( \sigma_0 \frac{x-\alpha_{e_0}y^{\eta}}{ 2^{l_0} } \r)$ into the cut-off of $W_{j,k}$. Here $\sigma_0$ takes either $+$ or $-$.

Note that $|\alpha_{t_{r+1}}|/|\alpha_{t_r+1}|\leq 2^{4N_0}$. By Lemma \ref{simple almost orthogonality principle}, there exists a constant $C$, depending on $N_0$, such that
\begin{equation*}
\Big\| \sum W_{j,k}  \Big\|
\leq
C \mathop{\sup}_{j,k} \|W_{j,k}\|,
\end{equation*}
where both the summation and the supremum are taken over all $(j,k)$ satisfying \textrm{$\mathbf{Subcase~(b)}$}. Associated with fixed $j,k,l_0$, we can decompose $G_0$ into the following three subsets:
\begin{eqnarray*}
G_{1,1}
&=&\{ i\in G_0 \mid  |\alpha_i-\alpha_{e_0}|2^{k\eta}\geq 2^{l_0+N_0}  \}\\
G_{1,2}
&=&\{ i\in G_0 \mid  |\alpha_i-\alpha_{e_0}|2^{k\eta}\leq 2^{l_0-N_0}  \}\\
G_{1,3}
&=&\{ i\in G_0 \mid  2^{l_0-N_0}<|\alpha_i-\alpha_{e_0}|2^{k\eta}< 2^{l_0+N_0}  \}.
\end{eqnarray*}
It is clear that $e_0\in G_{1,2}$. If $G_{1,3}$ is empty, then our decomposition is finished. Otherwise, choose the least number $e_1$ in $G_{1,3}$, and define $W_{j,k,l_0,l_1}^{\sigma_0,\sigma_1}$ as $W_{j,k,l_0}^{\sigma_0}$ by inserting $\Phi\l( \sigma_1\frac{ x-\alpha_{e_1}y^{\eta} }{ 2^{l_1} } \r)$ into the cut-off of $W_{j,k,l_0}^{\sigma_0}$, where $\sigma_1=+$ or $-$. Since the number of $l_0$ satisfying $|\alpha_{e_1}-\alpha_{e_0}|2^{k\eta}\approx 2^{l_0}$ is bounded by a constant $C(N_0)$, there exists a constant $C$, depending only on $N_0$, such that
\begin{equation*}
\Big\| \sum_{l_0} W_{j,k,l_0}^{\sigma_0}  \Big\|
\leq
C \mathop{\sup}\limits_{l_0}
\l\| W_{j,k,l_0}^{\sigma_0}  \r\|,
\end{equation*}
where $l_0$ satisfies the restriction in $G_{1,3}$. Further decompose $G_{1,3}$ as follows:
\begin{eqnarray*}
G_{2,1}
&=&\{ i\in G_{1,3} \mid  |\alpha_i-\alpha_{e_1}|2^{k\eta}\geq 2^{l_1+N_0}  \}\\
G_{2,2}
&=&\{ i\in G_{1,3} \mid  |\alpha_i-\alpha_{e_1}|2^{k\eta}\leq 2^{l_1-N_0}  \}\\
G_{2,3}
&=&\{ i\in G_{1,3} \mid  2^{l_1-N_0}<|\alpha_i-\alpha_{e_1}|2^{k\eta}< 2^{l_1+N_0}  \}.
\end{eqnarray*}
Then $e_1\in G_{2,2}$. If $G_{2,3}$ is nonempty, we continue this decomposition procedure and obtain three subsets $G_{3,1}$, $G_{3,2}$, $G_{3,3}$ of $G_{2,3}$.

Generally, if $G_{i,1},G_{i,2},G_{i,3}$ are given with $G_{i,3}\neq \emptyset$, then we shall choose the least integer $e_i$ in $G_{i,3}$ and decompose $G_{i,3}$, for each fixed $l_i$, as follows:
\begin{eqnarray*}
G_{i+1,1}
&=&\{ t\in G_{i,3} \mid  |\alpha_t-\alpha_{e_i}|2^{k\eta}\geq 2^{l_i+N_0}  \}\\
G_{i+1,2}
&=&\{ t\in G_{i,3} \mid  |\alpha_t-\alpha_{e_i}|2^{k\eta}\leq 2^{l_i-N_0}  \}\\
G_{i+1,3}
&=&\{ t\in G_{i,3} \mid  2^{l_i-N_0}<|\alpha_t-\alpha_{e_i}|2^{k\eta}< 2^{l_i+N_0}  \}.
\end{eqnarray*}
Then we obtain three disjoint subsets $G_{i+1,1},G_{i+1,2},G_{i+1,3}$ of $G_{i,3}$. This process will continue until $G_{w,3}=\emptyset$ for some $w$. We shall prove later that this process will terminate in finite steps.

Although our construction of $G_{i,1},G_{i,2},G_{i,3}$ depends on $l_{i-1}$, it is more convenient to regard $\{G_{i,1},G_{i,2},G_{i,3}\}$ as a three-tuple satisfying three properties:\\

(i)~$G_{i,u}\bigcap G_{i,v}=\emptyset$ for $u\neq v$;\\

(ii) if $G_{i,3}$ is nonempty then so is $G_{i+1,2}$;\\

(iii)~$G_{i-1,3}=G_{i,1} \cup G_{i,2} \cup G_{i,3}$ with $G_{0,3}=G_0=\{t_r+1,t_r+2,\cdots,t_{r+1}\}$.\\

\noindent{These} properties imply $|G_{i+1,3}|\leq |G_{i,3}|-1$ provided that $G_{i,3}$ is nonempty. Thus the above decomposition process stops in finite steps.

Since $G_0$ is a finite set, $|G_0|=t_{r+1}-t_r\leq N$, the number of all three-tuples $(G_{i,1},G_{i,2},G_{i,3})$, for each $i\geq 1$, is bounded by a constant depending only on $N$. Therefore, in each step, we can divide the summation over $l_i$ into a finite summation by restricting $l_i$ in $(G_{i,1},G_{i,2},G_{i,3})$. In other words, $l_i$ satisfies inequalities associated with $G_{i,u}$; for example, inequalities associated with $G_{i,1}$ are $|\alpha_t-\alpha_{e_i}|2^{k\eta}\geq 2^{l_i+N_0}$, where $t\in G_{i,1}$ and $e_i$ is the least member in $G_{i,2}$. This observation will simplify our proof of the decay estimate of $W_{j,k,l_0,l_1,\cdots,l_{w-1}}^{\sigma_0,\sigma_1,
\cdots,\sigma_{w-1}}$.

Assume $W_{j,k,l_0,\cdots,l_{w-1}}^{\sigma_0,\cdots,\sigma_{w-1}}$ and $W_{j,k,l_0,\cdots,l'_{w-1}}^{\sigma_0,\cdots,\sigma_{w-1}}$ are two operators such that $l_{w-1}$ and $l'_{w-1}$ satisfy inequalities associated with the same three-tuple
$(G_{w-1,i})_{i=1}^3$. It should be pointed out that $G_{w,3}$ is empty. Now we shall verify the assumptions in Lemma \ref{Lemma almost orth van der Corput}. First observe that $W_{j,k,l_0,\cdots,l_{w-1}}^{\sigma_0,\cdots,\sigma_{w-1}}$ is supported in the intersection of $\Omega$ and the following horizontally (also vertically) convex domain
\begin{equation}\label{sec3 HV domain support of Wjkl}
\Omega_{j,k,l_0,\cdots,l_{w-1}}^{\sigma_0,\cdots,\sigma_{w-1}}
=
\mathbf{Closure~~of~~}{\left\{(x,y):\Phi\l( \frac{x}{2^j} \r)\Phi\l( \frac{y}{2^k} \r)
\prod_{t=0}^{w-1}\Phi\l( \sigma_t \frac{x-\alpha_{e_t}y^{\eta}}{2^{l_t}} \r)\neq 0\right\}}
\end{equation}
provided that $\Phi\in C_0^{\infty}(\bR)$ further satisfies $\Phi(x)>0$ on $(\frac{1}{2},2)$.

Define an expanded region $\Omega_{j,k,l_0,\cdots,l_{w-1}}^{\sigma_0,\cdots,\sigma_{w-1}~\ast}$ by
\begin{eqnarray}\label{sec3 expanded domain general}
& &2^{j-1}-\epsilon 2^{j}\leq x \leq 2^{j+1}+\epsilon 2^{j},~~~
2^{k-1}-\epsilon 2^{k}\leq y \leq 2^{k+1}+\epsilon 2^{k},~~~\nonumber\\
& &~2^{l_t-1}-\epsilon 2^{l_{t}}\leq \sigma_t(x-\alpha_{e_t}y^{\eta}) \leq 2^{l_t+1}+\epsilon 2^{l_{t}},~~~0\leq t \leq w-1
\end{eqnarray}
for some sufficiently small $\epsilon=\epsilon(\eta)>0$. Observe that $W_{j,k,l_0,\cdots,l_{w-1}}^{\sigma_0,\cdots,\sigma_{w-1}}=0$ unless
$l_{w-1}\leq \min\{j,l_0,\cdots,l_{w-2}\}+C(\eta)$. One can verify that all conditions in Lemma \ref{Lemma almost orth van der Corput} are true for $W_{j,k,l_0,\cdots,l_{w-1}}^{\sigma_0,\cdots,\sigma_{w-1}}$ and $W_{j,k,l_0,\cdots,l'_{w-1}}^{\sigma_0,\cdots,\sigma_{w-1}}$.

Assume $l_{w-1}\geq l'_{w-1}$. By Lemma \ref{Lemma almost orth van der Corput}, $\| W_{ j,k,l_0,l_1,\cdots, l_{w-1} }^{\sigma_0,\cdots,\sigma_{w-1}}   W^{\sigma_0,\cdots,\sigma_{w-1}~\ast}_{j,k,l_0,l_1,\cdots, l'_{w-1} } \|$ is bounded by a constant multiple of
\begin{eqnarray*}
\l[ |\lambda|  2^{jm}  2^{kn}  2^{jt_r}
\prod_{i=1}^{w}\l( \prod_{t\in G_{i,1}}  |\alpha_t-\alpha_{e_{i-1}}| 2^{k\eta}
\prod_{t\in G_{i,2}}  2^{l_{i-1}} \r)
\prod_{t=t_{r+1}+1}^{N}\left(|\alpha_t|2^{k\eta}\right) \r]^{-1} \times
\end{eqnarray*}
\begin{eqnarray*}
\l[  2^{jm}  2^{j|\Xi_1\cap\Theta|}
\prod_{i=1}^{w-1}\l( \prod_{t\in G_{i,1} \cap \Theta}  |\alpha_t-\alpha_{e_{i-1}}| 2^{k\eta}
\prod_{t\in G_{i,2} \cap \Theta }  2^{l_{i-1}} \r)
\prod_{t\in \Xi_2\cap \Theta} \left(|\alpha_t|2^{k\eta}\right) \r]^{2\Re(z)} \times
\end{eqnarray*}
\begin{eqnarray*}
\l[  \l( \prod_{t\in G_{w,1} \cap \Theta}  |\alpha_t-\alpha_{e_{w-1}}| 2^{k\eta}
\prod_{t\in G_{w,2} \cap \Theta }  2^{l_{w-1}} \r)
\l( \prod_{t\in G_{w,1} \cap \Theta}  |\alpha_t-\alpha_{e_{w-1}}| 2^{k\eta}
\prod_{t\in G_{w,2} \cap \Theta }  2^{l'_{w-1}} \r)  \r]^{\Re(z)}
\end{eqnarray*}
where $\Xi_1=\{1,2,\cdots,t_r\}$, $\Theta=\{i_1,i_2\cdots,i_s\}$, and $\Xi_2=\{t_{r+1}+1,\cdots,N\}$.

By size estimates, $\| W_{  j,k,l_0,l_1,\cdots, l_{w-1} }^{\sigma_0,\cdots,\sigma_{w-1}} W_{  j,k,l_0,l_1,\cdots, l'_{w-1} }^{\sigma_0,\cdots,\sigma_{w-1}~\ast}\|$ is less than a constant multiple of
\begin{eqnarray*}
\l[  2^{jm}  2^{j|\Xi_1\cap\Theta|}
\prod_{i=1}^{w-1}\l( \prod_{t\in G_{i,1} \cap \Theta}  |\alpha_t-\alpha_{e_{i-1}}| 2^{k\eta}
\prod_{t\in G_{i,2} \cap \Theta }  2^{l_{i-1}} \r)
\prod_{t\in \Xi_2\cap \Theta} \left(|\alpha_t|2^{k\eta}\right) \r]^{2\Re(z)} \times
\end{eqnarray*}
\begin{eqnarray*}
\l[  \l( \prod_{t\in G_{w,1} \cap \Theta}  |\alpha_t-\alpha_{e_{w-1}}| 2^{k\eta}
\prod_{t\in G_{w,2} \cap \Theta }  2^{l_{w-1}} \r)
\l( \prod_{t\in G_{w,1} \cap \Theta}  |\alpha_t-\alpha_{e_{w-1}}| 2^{k\eta}
\prod_{t\in G_{w,2} \cap \Theta }  2^{l'_{w-1}} \r)  \r]^{\Re(z)}\times
\end{eqnarray*}
\begin{equation*}
2^{ l_{w-1}/2 } \l( \frac{2^{ l_{w-1} }}{ |\alpha_{ e_{w-1} }| 2^{k(\eta-1)} } \r)^{1/2}
2^{ l'_{w-1}/2 } \l( \frac{2^{ l'_{w-1} }}{ |\alpha_{ e_{w-1} }| 2^{k(\eta-1)} } \r)^{1/2},
\end{equation*}
where the last two terms are upper bounds for the measure of $\left\{y\mid
\Phi\l(\frac{y}{2^k}\r)
\Phi\l(\sigma_{w-1}\frac{x-\alpha_{e_{w-1}}y^{\eta}}{2^{u}}\r)\neq 0\right\}$
with $u\in \{l_{w-1},l'_{w-1}\}$. On the other hand, our decomposition of $G_0$ implies
\begin{eqnarray*}
|\alpha_t-\alpha_{e_{i-1}}|2^{k\eta}
&\geq &
C(\eta) 2^{l_{w-1}},~~t\in G_{i,1},~~1\leq i \leq w;\\
2^{l_i}
&\geq &
C(\eta,|G_0|) 2^{l_{w-1}},~~~1\leq i \leq w-2.
\end{eqnarray*}
Using these inequalities and the observation that $\theta=2\gamma$ satisfies $\theta-2\Re(z)\geq0$, we deduce from a convex combination of the above oscillation and size estimates that $\| W_{  j,k,l_0,l_1,\cdots, l_{w-1} }^{\sigma_0,\cdots,\sigma_{w-1}} W_{  j,k,l_0,l_1,\cdots, l'_{w-1} }^{\sigma_0,\cdots,\sigma_{w-1}~\ast}\|$ is not greater than a constant multiple of
\begin{eqnarray*}\label{sec3 almost ortho esti 2}
\l[ |\lambda|  2^{jm}  2^{kn}  2^{jt_r}
2^{l_{w-1}|G_0|}
\prod_{\Xi_2}\left(|\alpha_t|2^{k\eta}\right) \r]^{-\theta}
\l[  2^{jm}  2^{j|\Xi_1\cap\Theta|}
2^{l_{w-1}|G_0\cap \Theta|}
\prod_{\Xi_2\cap \Theta} \left(|\alpha_t|2^{k\eta}\right) \r]^{2\Re(z)}\times
\end{eqnarray*}
\begin{eqnarray}
2^{-l_{w-1}|G_{w,2}\cap\Theta|\Re(z)}
2^{l'_{w-1}|G_{w,2}\cap\Theta|\Re(z)}
2^{l_{w-1}(1-\theta)}
2^{l'_{w-1}(1-\theta)}\l( |\alpha_{e_{w-1}}| 2^{k(\eta-1)} \r)^{-(1-\theta)}.
\end{eqnarray}
The procedure for decomposition of $G_0$ also implies
\begin{eqnarray*}
2^j \geq C(\eta,|G_0|) 2^{l_{w-1}},~~~~~
|\alpha_{t_{r+1}}|2^{k\eta} \geq C(\eta,|G_0|) 2^{l_{w-1}}.
\end{eqnarray*}
Inserting these estimates into the above bound in (\ref{sec3 almost ortho esti 2}), we see that the resulting exponent of $2^{l_{w-1}}$ is
\begin{eqnarray*}
& &-\l(m+\frac{n}{\eta}+|\Xi_1|+|G_0|+|\Xi_2|\r)\theta+
\Big(m+|\Xi_1\cap\Theta|+|G_0\cap\Theta|+\\
& &|\Xi_2\cap\Theta|\Big)\cdot 2\Re(z)-|G_{w,2}\cap\Theta|\Re(z)+
1-\theta
-\frac{\eta-1}{\eta}(1-\theta)\\
&=&
-\l(m+\frac{n}{\eta}+N\r)\theta+2(m+s)\Re(z)-|G_{w,2}\cap\Theta| \Re(z)
+\frac{1-\theta}{\eta}\\
&=&
-\l(m+\frac{n}{\eta}+N\r)\theta+(m+s)\l(\theta-\frac{1-\theta}{m+s}\r)
-|G_{w,2}\cap\Theta|\cdot \Re(z)+\frac{n+(N-s)\eta}{\eta}\theta\\
&=&
-(1-\theta)-|G_{w,2}\cap\Theta| \Re(z),
\end{eqnarray*}
and the bound involving the coefficients $\alpha_{i_t}$ is
\begin{eqnarray*}
& &
\Big|\alpha_{t_{r+1}}\Big|^{\frac{n}{\eta}\theta}
\l(\prod_{\Xi_2}\frac{|\alpha_t|}{|\alpha_{t_{r+1}}|}\r)^{-\theta}
\l(\prod_{\Xi_2\cap\Theta}\frac{|\alpha_t|}{|\alpha_{t_{r+1}}|}\r)^{2\Re(z)}
\Big|\alpha_{e_{w-1}}\Big|^{-(1-\theta)}
\Big|\alpha_{t_{r+1}}\Big|^{\frac{\eta-1}{\eta}(1-\theta)}\\
&\leq&
C(\eta)
\Big|\alpha_{t_{r+1}}\Big|^{\frac{n}{\eta}\theta}
\l(\prod_{\Xi_2\backslash\Theta}\frac{|\alpha_t|}{|\alpha_{t_{r+1}}|}\r)^{-\theta}
\Big|\alpha_{t_{r+1}}\Big|^{-(1-\theta)}
\Big|\alpha_{t_{r+1}}\Big|^{\frac{\eta-1}{\eta}(1-\theta)}\\
&\leq&
C(\eta)
\l(\prod_{t \in \Xi_2\backslash\Theta}|\alpha_t|\r)^{-\theta}
\Big|\alpha_{t_{r+1}}\Big|^{|\Xi_2\backslash\Theta|\theta-(N-s)\theta}
\leq
C(\eta)\l(\prod_{t \notin \Theta}|\alpha_t|\r)^{-\theta},
\end{eqnarray*}
where we have used $|\alpha_{e_{w-1}}|\approx |\alpha_{t_{r+1}}|$, $\theta-2\Re(z)\geq0$, and $N-s=|\Xi_1\backslash\Theta|+|G_0\backslash\Theta|+|\Xi_2\backslash\Theta|$.
Therefore we obtain
\begin{equation}\label{sec3 almost ortho esti}
\l\| W_{  j,k,l_0,l_1,\cdots, l_{w-1} }^{\sigma_0,\cdots,\sigma_{w-1}}   W^{\sigma_0,\cdots,\sigma_{w-1}~\ast}_{j,k,l_0,l_1,\cdots, l'_{w-1} } \r \|
\leq
C\l( |\lambda| \prod_{k\notin \Theta}|\alpha_k| \r)^{-\theta}
2^{-|l_{w-1}-l'_{w-1}|\delta}
\end{equation}
with $\delta=|G_{w,2}\cap\Theta|\Re(z)+1-\theta$. It remains to show $\delta>0$. In fact, we see that if $1-\theta=0$ then $n=0,s=N$ and hence $\delta=|G_{w,2}|\theta/2>0$. On the other hand, if $1-\theta>0$ then $\delta=\frac{a\theta}{2}+\left(1-\frac{a}{2(m+s)}\right)(1-\theta)>0$ with $a=|G_{w,2}\cap\Theta|\leq s$. By the same argument as above, one can obtain a similar estimate for
$W_{  j,k,l_0,l_1,\cdots, l_{w-1} }^{\sigma_0,\cdots,\sigma_{w-1}~\ast}   W^{\sigma_0,\cdots,\sigma_{w-1}}_{j,k,l_0,l_1,\cdots, l'_{w-1} }$.\\

$\mathbf{ Step~2.~~Existence~of~complex~roots }$\\

In what follows, we want to show how the above arguments in Step 1 are modified in the presence of complex roots. In Step 1, a resolution of singularities is needed only in Subcase (b) of Case (iii). For this reason, it is clear that the earlier arguments can apply without any change to Case (i), Case (ii) and Subcase (a) of Case (iii) now. Hence we need only to consider Subcase (b) of Case (iii).

Assume $j,k\in \bZ$ satisfy the above restriction in Subcase (b).
It is easy to see that for those $t\in G_0:=\{t_r+1,t_r+2,\cdots,t_{r+1}\}$ satisfying $|\Re(\alpha_t)| \leq |\Im(\alpha_t)|$, we have
\begin{equation}\label{sec3 equiv size for complex roots}
|x- \alpha_{t} y^{\eta}|
\approx
|x- \Re(\alpha_{t}) y^{\eta}|
+
|\Im(\alpha_{t})| y^{\eta}
\approx |\alpha_t|2^{k\eta}
\end{equation}
in the support of $W_{j,k}$. Let $G_0^{(1)}$ consist of $t\in G_0$ such that $|\Re(\alpha_t)| \leq |\Im(\alpha_t)|$, and $G_0^{(2)}$ be the complement of $G_0^{(1)}$ in $G_0$, i.e.,
\begin{equation}
G_0^{(1)}=\{t\in G_0:~|\Re(\alpha_t)| \leq |\Im(\alpha_t)|\},~~~
G_0^{(2)}=\{t\in G_0:~|\Re(\alpha_t)| > |\Im(\alpha_t)|\}.
\end{equation}
We shall first give a resolution of singularities for $\prod_{t\in G_0^{(2)}}(x-\Re(\alpha_t)y^{\eta})$. The method in Step 1 can apply here with $G_0$ and $\alpha_t$ replaced by $G_0^{(2)}$ and $\Re(\alpha_t)$, respectively.
For clarity, we shall now describe this procedure. Choose the least number $e_0$ in $G_0^{(2)}$, and define $W_{j,k,l_0}^{\sigma_0}$ by insertion of the cut-off function $\Phi\Big(\sigma_0\frac{ x-\Re(\alpha_{e_0})y^{\eta} }{2^{l_0}}\Big)$ into $W_{j,k}$. Then we define a three-tuple $\{G_{1,1},G_{1,2},G_{1,3}\}$ as in Step 1 with $G_0^{(2)}$ in place of $G_0$. If $\{G_{i-1,1},G_{i-1,2},G_{i-1,3}\}$ is given and $G_{i-1,3}$ is not empty, we can define $W_{j,k,l_0,\cdots,l_i}^{\sigma_0,\cdots,\sigma_i}$ by insertion of $\Phi\Big(\sigma_i\frac{x-\Re(\alpha_{e_{i-1} } ) y^{\eta}}{2^{l_i}}\Big)$ into $W_{j,k,l_0,\cdots,l_{i-1}}^{\sigma_0,\cdots,\sigma_{i-1}}$. As in Step 1, the integer $e_{i-1}$ is the least number in $G_{i-1,3}$, and a new three-tuple $\{G_{i,1},G_{i,2},G_{i,3}\}$ can be defined in the same way.

Assume $\mathcal{G}:=\{G_{i,1},G_{i,2},G_{i,3}\}_{i=1}^w$ is a sequence of three-tuples satisfying all of the above three properties in Step 1. Now for two operators $W_{j,k,l_0,\cdots,l_{w-1}}^{\sigma_0,\cdots,\sigma_{w-1}}$ and $W_{j,k,l_0,\cdots,l_{w-1}'}^{\sigma_0,\cdots,\sigma_{w-1}}$ of class $\mathcal{G}$, our purpose is to establish the almost orthogonality estimate (\ref{sec3 almost ortho esti}). It should be pointed out that the above oscillation and size estimates in Step 1 are still true, except for the size of the factor $\prod_{t=t_r+1}^{t_{r+1}}(x-\alpha_ty^{\eta})$. In the support of $W_{j,k,l_0,\cdots,l_{w-1}}^{\sigma_0,\cdots,\sigma_{w-1}}$, the size of $\prod_{t=t_r+1}^{t_{r+1}}(x-\alpha_ty^{\eta})$ is comparable to
$$
\prod_{G_0^{(1)}}\Big( |\alpha_t|2^{k\eta} \Big)
\prod_{i=1}^w\left(\prod_{G_{i,1}}|\Re(\alpha_t)
-\Re(\alpha_{e_{i-1}})|2^{k\eta}\vee |\Im(\alpha_t)|2^{k\eta}\right)
\left(\prod_{G_{i,2}}2^{l_{i-1}}\vee |\Im(\alpha_t)|2^{k\eta}\right),
$$
where (\ref{sec3 equiv size for complex roots}) is used for $t\in G_0^{(1)}$.
Now we shall divide our discussion into several cases.\\

$\bullet$ $\Re(z)=\frac{\theta}{2}-\frac{1-\theta}{2(m+s)}\leq 0$

In the convex combination of oscillation and size estimates for $W_{j,k,l_0,\cdots,l_{w-1}}^{\sigma_0,\cdots,\sigma_{w-1}~\ast}
W_{j,k,l_0,\cdots,l_{w-1}'}^{\sigma_0,\cdots,\sigma_{w-1}}$
(with $\theta=2\gamma$ as in Step 1), all exponents of terms of the forms $|\Re(\alpha_t)
-\Re(\alpha_{e_{i-1}})|2^{k\eta} \vee |\Im(\alpha_t)|2^{k\eta}$ and
$2^{l_{i-1}}\vee |\Im(\alpha_t)|2^{k\eta}$ are nonpositive. Hence we can obtain a bigger bound by replacing these two terms by $|\Re(\alpha_t)
-\Re(\alpha_{e_{i-1}})|2^{k\eta}$ and
$2^{l_{i-1}}$, respectively. On the other hand, as in the case of real roots, we also have $|\alpha_t|2^{k\eta}\gtrsim 2^{l_{w-1}}$ for $t\in G_0^{(1)}$, $|\Re(\alpha_t)
-\Re(\alpha_{e_{i-1}})|2^{k\eta}\gtrsim 2^{l_{w-1}}$ for $t\in G_{i,1}$, and $2^{l_i}\gtrsim 2^{l_{w-1}}$ for $0\leq i \leq w-1$. By the same argument as in Step 1, we can show that the estimate (\ref{sec3 almost ortho esti}) is alo true.\\

$\bullet$ $\Re(z)\geq 0$ and $0\leq \theta <1$

In the resulting $L^2$ estimate (with a convex combination of oscillation and size estimates as above), terms concerning $2^{l_{w-1}}$ and $2^{l_{w-1}'}$ are given respectively by
$$
\left(\prod\nolimits_{G_{w,2}}2^{l_{w-1}}\vee |\Im(\alpha_t)|2^{k\eta}\right)^{-\theta}
\left(\prod\nolimits_{G_{w,2}\cap\Theta}2^{l_{w-1}}\vee |\Im(\alpha_t)|2^{k\eta}\right)^{\Re(z)}
2^{l_{w-1}(1-\theta)}
$$
and
\begin{equation*}
\left(\prod\nolimits_{G_{w,2}\cap\Theta}2^{l_{w-1}'}\vee |\Im(\alpha_t)|2^{k\eta}\right)^{\Re(z)}
2^{l_{w-1}'(1-\theta)}.
\end{equation*}
Since $\Re(z)\geq0$ and $l_{w-1}\geq l_{w-1}'$ (assumed as in Step 1), we have
\begin{equation*}
\left(\prod\nolimits_{G_{w,2}\cap\Theta}2^{l_{w-1}}\vee |\Im(\alpha_t)|2^{k\eta}\right)^{-\Re(z)}
\left(\prod\nolimits_{G_{w,2}\cap\Theta}2^{l_{w-1}'}\vee |\Im(\alpha_t)|2^{k\eta}\right)^{\Re(z)}\leq 1.
\end{equation*}
Thus the $L^2$ operator norm is dominated by a constant multiple of the bound in Step 1, where the exponents of $2^{l_{w-1}}$ and $2^{l_{w-1}'}$ should be replaced by $-(1-\theta)$ and $1-\theta$, respectively. Therefore we obtain the estimate (\ref{sec3 almost ortho esti}) with $\delta=1-\theta$.\\

$\bullet$ $\Re(z)\geq 0$ and $\theta=1$

Our assumption $\theta=1$ implies $n=0$ and $s=N$. Hence $\Re(z)=\frac12$. Write $G_{w,2}=\{u_1,u_2,\cdots,u_{\kappa}\}$ with $\kappa\geq 1$. Assume the absolute value of $\alpha_{u_i}$ increases as $i$. If $2^{l_{w-1}}\leq |\Im(\alpha_{u_1})|2^{k\eta}$, we shall consider $\sum_{l_{w-1}}W_{j,k,l_0,\cdots,l_{w-1}}^{\sigma_0,\cdots,\sigma_{w-1}}$, where this summation is taken over those $l_{w-1}$ such that $W_{j,k,l_0,\cdots,l_{w-1}}^{\sigma_0,\cdots,\sigma_{w-1}}$ is of class $\mathcal{G}$ and $2^{l_{w-1}}\leq |\Im(\alpha_{u_1})|2^{k\eta}$. By Lemma \ref{Lemma almost orth van der Corput} (with $T_{\lambda}^{(1)}=T_{\lambda}^{(2)}$), we obtain
\begin{equation}
\left\|\sum\nolimits_{l_{w-1}}W_{j,k,l_0,\cdots,l_{w-1}}^{\sigma_0,\cdots,\sigma_{w-1}} \right\|
\leq
C(deg(S),\varphi)|\lambda|^{-\frac12}.
\end{equation}
Otherwise if $2^{l_{w-1}},2^{l_{w-1}'}\geq |\Im(\alpha_t)|2^{k\eta}$ for some $1\leq t \leq \kappa$, then the almost orthogonality estimate (\ref{sec3 almost ortho esti}) is true with $\delta=t$.\\

It remains to show that all constants $C$ appearing in above estimates have a more precise form as in the theorem. Indeed, we can deduce from Lemma \ref{Lemma almost orth van der Corput} as well as Remark \ref{remark almost ortho esti} that all above constants $C$ can take the following form
\begin{equation*}
C(deg(S))
\mathop{\sup}_{\Omega}
\sum_{k=0}^{2}
\left( \big(\delta^{}_{\Omega,h}(x) \big)^k |\partial^k_y\varphi(x,y)|
+
\big( \delta^{}_{\Omega,v}(y) \big)^k |\partial^k_x\varphi(x,y)|\right).
\end{equation*}
Here $C(deg(S))$ is a constant depending only on the degree of $S$.

Combining all above results, we have completed the proof of the theorem.
\end{proof}\\

As mentioned in Remark \ref{Remark Theorem L2 damping estimate}, the assumption $\eta \geq 1$ is necessary in the proof of Theorem \ref{Theorem L2 damping estimate}. Now we address the case $\eta<1$. Let $\nu=\eta^{-1}>1$. Since $S$ is a polynomial, (\ref{Hessian of phase}) can be rewritten as
\begin{equation}\label{Hessian of phase II}
S_{xy}''(x,y) = d_0 x^m y^n \prod_{i=1}^{M}\l( x^{\nu}-\beta_i y \r)
\end{equation}
with $d_0,\beta_i\in \mathbb{C}\backslash \{0\}$. Without loss of generality, we may assume $c_0=1$. For some $0\leq s \leq M$, we first choose indices $1\leq i_1<\cdots<i_s\leq N$. Then define the damping factor $D$ as
\begin{equation*}
D(x,y) = x^m \prod_{t=1}^s\l( x^{\nu}-\beta_{i_t} y \r).
\end{equation*}
Here we may take $s=0$ and then define $D(x,y)=x^m$. The following theorem can be regarded as a variant of Theorem \ref{Theorem L2 damping estimate}.

\begin{theorem}\label{Theorem Duality of damping L2 estimate}
Assume $S$ is a real-valued polynomial such that its Hessian is given by {\rm(\ref{Hessian of phase II})} with $\nu\geq 1$. Let $W_z$ be defined as in {\rm(\ref{damping OIO})}. Then there exists a constant $C$, depending only on $deg(S)$ and $\varphi$, such that
\begin{equation}\label{damping decay estimate II}
\|W_zf\|_{L^2}
\leq
C(1+|z|^2)\l(|\lambda|\prod_{k\notin \{i_1,\cdots,i_s\}}|\beta_k|\r)^{-\gamma}\|f\|_{L^2},
~~~\gamma=\frac{1}{2(n+M-s+1)},
\end{equation}
where $z\in\mathbb{C}$ has real part
\begin{equation}\label{L2 damping exponent II}
\Re(z)=\frac{m+s\nu-n-(M-s)}{2(n+M-s+1)}\cdot\frac{1}{m+s\nu}.
\end{equation}
\end{theorem}
\begin{remark}
The proof of this theorem is the same as that of Theorem \ref{Theorem L2 damping estimate}. For this reason, we omit the details here. It should be pointed out that the constant $C$ in the above theorem can also take the form in Theorem \ref{Theorem L2 damping estimate}. By the same argument as above, we can also prove Theorem \ref{Uniform Damped OIO} with uniformity on both the phases and the cut-off functions.
\end{remark}

\noindent{We} conclude this section with a $L^2$ damping estimate for $W_{j,k}$ in (\ref{definition of Ujk}). Assume $$\Theta:=\{i_1,i_2,\cdots,i_s\}=\{t_r+1,t_r+2,\cdots,t_{r+1}\}$$
for some $0\leq r \leq a-1$, where $\{t_1,t_2,\cdots,t_a\}$ is defined as in (\ref{sec4 exceptional index set}). Here we set $t_0=0$. In section 4, we will need an analogue of the damping estimate in Subcase (b) of Case (iii) with the modified damping factor
\begin{equation}\label{defn of Djk}
D(x,y)
=
\Big(
|\lambda| |\alpha_{i_s}|^{-1} 2^{-k(\eta-1)}
A\Big)^{ -\frac{s}{s+2} }
+
\prod_{t=1}^{s}\big|x-\alpha_{i_t}y^{\eta}\big|,
\end{equation}
where $A$ is a positive number given by
\begin{equation}\label{defn of A}
A=2^{jm} 2^{kn} 2^{jt_r } \left(\prod_{t=t_{r+1}+1}^N |\alpha_t|2^{k\eta} \right).
\end{equation}

\begin{theorem}\label{Theorem L2 damping esti revision}
Assume $m=0$ in {\rm(\ref{Hessian of phase})} and $\{i_1,i_2,\cdots,i_s\}=\{t_r+1,t_r+2,\cdots,t_{r+1}\}$ for some $0\leq r \leq a-1$ and $|\alpha_{i_s}|2^{(k-1)\eta-2}\leq 2^j \leq |\alpha_{i_s}|2^{(k+1)\eta+2}$. Let $W_{j,k}$ be the damped oscillatory integral operator in {\rm (\ref{definition of Ujk})} with $D(x,y)$ defined by {\rm (\ref{defn of Djk})}. Then there exists a constant $C$ as in {\rm (\ref{uniform constant})} such that the decay estimate {\rm (\ref{damping decay estimate})} is still true for $W_{j,k}$ with the damping exponent $z\in\mathbb{C}$ having real part {\rm (\ref{L2 damping exponent})}.
\end{theorem}
\begin{proof}
With $2^j \approx |\alpha_{i_s}|2^{k\eta}$, $S''_{xy}$ behaves like
$ A\prod_{t=1}^{s}(x-\alpha_{i_t}y^{\eta})$ on the support of $W_{j,k}$.
As in the proof of Theorem \ref{Theorem L2 damping estimate}, we can decompose $W_{j,k}$ as
$$W_{j,k}
=\sum W_{j,k,l_0,l_1,\cdots,l_{w-1}}^{\sigma_0,\sigma_1,\cdots,\sigma_{w-1}}
.$$
On the support of each operator $W_{j,k,l_0,l_1,\cdots,l_{w-1}}^{\sigma_0,\sigma_1,\cdots,\sigma_{w-1}}$,
the Hessian $S''_{xy}$ (also $D$) is bounded from both below and above by the same bound up to a multiplicative constant.

If $D$ has size comparable to $\prod_{t=1}^s|x-\alpha_{i_t}y^{\eta}|$ on the support of $W_{j,k,l_0,l_1,\cdots,l_{w-1}}^{\sigma_0,\sigma_1,
\cdots,\sigma_{w-1}}$, then all its partial derivatives also have the same upper bounds as that of $\prod_{t=1}^s(x-\alpha_{i_t}y^{\eta})$. With this observation, the desired estimate can be proved by the same argument as in our proof of Theorem \ref{Theorem L2 damping estimate}.

It remains to consider these operators $W_{j,k,l_0,l_1,\cdots,l_{w-1}}^{\sigma_0,\sigma_1,\cdots,\sigma_{w-1}}$
for which $D$ has size
$$|D(x,y)|\approx \Big(
|\lambda| |\alpha_{i_s}|^{-1} 2^{-k(\eta-1)}
A\Big)^{ -\frac{s}{s+2} }.$$
In other words, it suffices to prove that the desired estimate holds if
\begin{equation}\label{size of damping factor}
\prod_{t=1}^{s}|x-\alpha_{i_t}
y^{\eta}|
\leq
\Big(
|\lambda| |\alpha_{i_s}|^{-1} 2^{-k(\eta-1)}
A\Big)^{ -\frac{s}{s+2}}.
\end{equation}
Observe that the cut-off function of $W_{j,k,l_0,l_1,\cdots,l_{w-1}}^{\sigma_0,\sigma_1,
\cdots,\sigma_{w-1}}$ contains a factor $\Phi\left( \sigma_{w-1}\frac{x-\alpha_{i_t}y^{\eta}}{2^{l_{w-1}}} \right)$
for some $\alpha_{i_t}$. Since $\alpha_{i_1},\cdots,\alpha_{i_s}$ have equivalent sizes and $y\in [2^{k-1},2^{k+1}]$, we have
$$|x-\alpha_{i_t}y^{\eta}|\leq
\Big(
|\lambda| |\alpha_{i_s}|^{-1} 2^{-k(\eta-1)}
A\Big)^{ -\frac{1}{s+2}}$$
in the support of $\sum_{l_{w-1}}W_{j,k,l_0,l_1,\cdots,l_{w-1}}^{\sigma_0,\sigma_1,
\cdots,\sigma_{w-1}}$. Here the summation is taken over all $l_{w-1}$ for which (\ref{size of damping factor}) holds in the support of $W_{j,k,l_0,l_1,\cdots,l_{w-1}}^{\sigma_0,\sigma_1,
\cdots,\sigma_{w-1}}$. Hence we have the following estimates of cross sections for $x$ and $y$
\begin{equation*}
\Delta x
\lesssim
\Big(|\lambda| |\alpha_{i_s}|^{-1} 2^{-k(\eta-1)}
A\Big)^{ -\frac{1}{s+2}},~~~~~~
\Delta y
\lesssim
\Big(|\lambda| |\alpha_{i_s}|^{-1} 2^{-k(\eta-1)}
A\Big)^{ -\frac{1}{s+2}}
\left( |\alpha_{i_s}| 2^{k(\eta-1)} \right)^{-1/2}.
\end{equation*}
The Schur test gives
\begin{equation*}
\left\|
\sum\nolimits_{l_{w-1}}
W_{j,k,l_0,l_1,\cdots,l_{w-1}}^{\sigma_0,\sigma_1,\cdots,\sigma_{w-1}}
f
\right\|_{L^2}
\leq
C \Big(
|\lambda| |\alpha_{i_s}|^{-1} 2^{-k(\eta-1)}
A\Big)^{ -\frac{1}{s+2}(1+s\Re(z))}
\Big( |\alpha_{i_s}| 2^{k(\eta-1)}\Big)^{ -\frac{1}{2} }
\|f\|_{L^2}.
\end{equation*}
On the other hand, $x\approx 2^j$ and $y\approx 2^k$ in the support of the operator considered above. Thus we also have
\begin{equation*}
\l\|
\sum_{l_{w-1}}
W_{j,k,l_0,l_1,\cdots,l_{w-1}}^{\sigma_0,\sigma_1,\cdots,\sigma_{w-1}}
f
\r\|_{L^2}
\leq
C \Big(
|\lambda| |\alpha_{i_s}|^{-1} 2^{-k(\eta-1)}
A\Big)^{ -\frac{s}{s+2}\Re(z)}
2^{j/2}2^{k/2}
\|f\|_{L^2}.
\end{equation*}
Let $\gamma$ be given by (\ref{damping decay estimate}) and $\theta=\frac12+\gamma$. Note that $\Re(z)=\gamma-\frac{n+(N-s)\eta}{s}\gamma
=\gamma-\frac1s(\frac12-\gamma)$. Hence $\gamma=\frac{s}{s+2}\Re(z)+\frac{\theta}{s+2}$.
A convex combination of the above estimates yields
\begin{equation*}
\l\|
\sum_{l_{w-1}}
W_{j,k,l_0,l_1,\cdots,l_{w-1}}^{\sigma_0,\sigma_1,\cdots,\sigma_{w-1}}
f
\r\|_{L^2}
\leq
C
\Big(
|\lambda| |\alpha_{i_s}|^{-1} 2^{-k(\eta-1)} A
\Big)^{-\gamma}
\Big(
|\alpha_{i_s}| 2^{k(\eta-1)}
\Big)^{-\theta/2}
\Big(
2^{j/2} 2^{k/2}
\Big)^{1-\theta}\|f\|_{L^2},
\end{equation*}
where the exponents of $2^j$ and $2^k$ are given by\\

\noindent{~~~~~$2^j$}:~~$a_j=-t_r\gamma+\frac12(1-\theta)$,\\

\noindent{~~~~~$2^k$}:~~$b_k=[(\eta-1)-n-(N-t_{r+1})\eta]\gamma
-(\eta-1)\theta/2
+(1-\theta)/2.$\\

\noindent By direct calculation, we have $\eta a_j+b_k=0$. In fact, using $\theta=\frac12+\gamma$, $1-\theta=\frac12-\gamma$ and $s=t_{r+1}-t_r$, one can see that $\eta a_j+b_k$ equlas
\begin{eqnarray*}
&&
-t_r\gamma\eta+\frac12(1-\theta)\eta+[(\eta-1)-n-(N-t_{r+1})\eta]\gamma
-(\eta-1)\theta/2
+(1-\theta)/2\\
&=&
[(\eta-1)-n-(N-t_{r+1})\eta-t_r\eta]\gamma
+
\frac12 \cdot \Big(\frac12-\gamma \Big)\cdot \eta
-
(\eta-1)\cdot \Big(\frac12+\gamma \Big)\cdot\frac12
+
\Big(\frac12-\gamma \Big)\cdot\frac12\\
&=&
[(\eta-1)-n-(N-s)\eta]\gamma
+
\frac12 \cdot \Big(\frac12-\gamma \Big)\cdot \eta
-
\frac12\cdot \Big(\frac12+\gamma \Big)\cdot \eta
+
\frac12\cdot \Big(\frac12+\gamma \Big)
+
\frac12\cdot \Big(\frac12-\gamma \Big)\\
&=&
\eta\gamma-\frac12-\gamma\eta+\frac12\\
&=&
0.
\end{eqnarray*}
Since sizes of $\alpha_{i_t}$ are equivalent, the bound concerning $\alpha_{i_t}$ in the resulting estimate is given by
\begin{eqnarray*}
& &|\alpha_{i_s}|^{\gamma}
|\alpha_{i_s}|^{-t_r\gamma}
\left(\prod_{t=t_{r+1}+1}^N |\alpha_t|\right)^{-\gamma}
|\alpha_{i_s}|^{-\theta/2}
|\alpha_{i_s}|^{(1-\theta)/2}\\
&\leq&
|\alpha_{i_s}|^{\gamma}
|\alpha_{i_s}|^{-t_r\gamma}
\left(\prod_{t=t_{r+1}+1}^N |\alpha_t|\right)^{-\gamma}
|\alpha_{i_s}|^{1/2-\theta}~~~~~\textrm{\rm with~~$\theta=\frac12+\gamma$}\\
&\leq&
\left(\prod_{t\notin\Theta} |\alpha_t|\right)^{-\gamma}
~~~~~\textrm{since}~~\Theta=\{i_1,i_2,\cdots,i_s\}=\{t_r+1,t_r+2,\cdots,t_{r+1}\}.
\end{eqnarray*}
Combining above results, we have completed the proof of the theorem.
\end{proof}

\section{Damped Oscillatory Integral Operators on $H_E^1$}

In this section, we shall establish uniform $H_E^1\rightarrow L^1$ estimates for damped oscillatory integral operators in Section 3. In general, these operators are not bounded from $H_E^1$ into $L^1$. However, we can decompose them into three parts such that each part has desired properties.

Assume $S$ is a real-valued polynomial in $\bR^2$ and its Hessian $S''_{xy}$ can be written as
$$
S''_{xy}(x,y)=x^my^n\prod_{i=1}^{N}\big(x-\alpha_iy^{\eta}\big)
$$
with $\alpha_1,\alpha_2,\cdots,\alpha_N\in \mathbb{C}\backslash\{0\}$ and $\eta>0$. Here we also assume $|\alpha_1|\leq |\alpha_2|\leq \cdots |\alpha_N|$. Let $W_z$ be a damped oscillatory integral operator of the form (\ref{damping OIO}). In this section, we shall modify the damping factor in Theorem \ref{Theorem L2 damping estimate}. First choose $s$ indices $1\leq i_1<i_2<\cdots<i_s \leq N$ such that $\prod_{t=1}^s(x-\alpha_{i_t}y^{\eta})$ is conjugate invariant in the sense of Remark \ref{sec2 remark on factorization}. The damping factor $D$ is defined as follows.

\noindent{$\bullet$
If either $m>0$ or $\mathop{\max}_{ t }|\alpha_{i_t}-\alpha_{i_s}|\geq \frac{|\alpha_{i_s}|}{4}$ is true, we assume $z$ has real part $\Re(z)=-\frac{1}{m+s}$. The damping factor $D$ is defined as in Theorem \ref{Theorem L2 damping estimate}, i.e.,}
$D(x,y)=
x^m\prod_{t=1}^{s}\big(x-\alpha_{i_t}y^{\eta}\big).$

\noindent{$\bullet$
If $m=0$ and $\mathop{\max}_{t}|\alpha_{i_t}-\alpha_{i_s}|< \frac{|\alpha_{i_s}|}{4}$,
we must have $\{i_1,i_2,\cdots,i_s\} \subseteq \{t_r+1,t_r+2,\cdots,t_{r+1}\}$ for some $0\leq r \leq a$ with $t_0=0$. Here $\{t_1,t_2,\cdots,t_{a}\}$ is the set of indices defined by (\ref{sec4 exceptional index set}). In this case, it is more convenient to consider the following special case \begin{equation}\label{New index set Theta}
\Theta :=\{i_1,i_2,\cdots,i_s\}=\{t_r+1,t_r+2,\cdots,t_{r+1}\}.
\end{equation}
With this revision, it should be pointed out that $\max_{t}
|\alpha_{i_t}-\alpha_{i_s}|< |\alpha_{i_s}|/4$ may not still hold. But this does not affect our final results. Now we shall define damping factors $D_{j,k}$ for each $W_{j,k}$, defined by (\ref{definition of Ujk}), as follows. }
\begin{equation}\label{sec4 def Djk}
D_{j,k}(x,y)=
\begin{cases}
~~~~~\prod_{t=1}^{s}\big(x-\alpha_{i_t}y^{\eta}\big),
~~~\textrm{if~~$2^j \geq  |\alpha_{i_s}|2^{\eta(k+1)+2}$~or~
$2^j \leq |\alpha_{i_1}|2^{\eta(k-1)-2}$};\\
\Big(
|\lambda| |\alpha_{i_s}|^{-1} 2^{-k(\eta-1)}
A\Big)^{ -\frac{s}{s+2} }
+
\prod_{t=1}^{s}\big|x-\alpha_{i_t}y^{\eta}\big|,~~~\textrm{otherwise;}
\end{cases}
\end{equation}
where $A$ is the constant in (\ref{defn of A}).\\

As a variant of the classical Hardy space $H^1$, we shall define the space $H_E^1$ associated with the phase $\lambda S$ and the set $\Theta$ considered above; see Phong-Stein \cite{PS1986}, Pan \cite{pan}, Greenleaf-Seeger \cite{greenleafseeger1}, Shi-Yan \cite{ShiYan} and Xiao \cite{Xiao2017} for earlier work related to this space.
\begin{defn}\label{sec4 def a variant of H1}
Let $I_k=[2^{k-1},2^{k+1}]$ and fix an index $u\in\Theta$. Associated with $\lambda S$ and $\Theta$, we say that a Lebesgue measurable function $\mathbf{a}$ is an atom in $H_E^1(I_k)$, if there exists an interval $I\subseteq I_k$ with the following three properties:\\

{\rm(i)} \supp$(\mathbf{a})$ $\subseteq I_k$;\\

{\rm(ii)} $|\mathbf{a}(x)|\leq |I|^{-1}$, a.e. $x\in I$; \\

{\rm(iii)} $\int e^{i\lambda S(\Re(\alpha_{u})c_I^{\eta},y)}\mathbf{a}(y)dy=0$ with $c_I$ the center of $I$;\\

\noindent where $\Re(\alpha_{u})$ is the real part of $\alpha_u$. The space $H_E^1(I_k)$ consists of $f\in L^1$ of form $f=\sum_j \lambda_j\mathbf{a}_j$. Here $\{\mathbf{a}_j\}$ is a sequence of atoms in $H_E^1(I_k)$, and $\{\lambda_j\}$ satisfies $\sum|\lambda_j|<\infty$. The norm of $f$ in $H_E^1(I_k)$ is defined by
\begin{equation*}
\|f\|_{H_E^1}
:=
\inf\l\{
\sum|\lambda_j|:~f=\sum_{j\in\mathbb{Z}} \lambda_j\mathbf{a}_j,~\mathbf{a_j}~atoms~in~H_E^1(I_k)
\r\}.
\end{equation*}
\end{defn}

Now we state our main result in this section.
\begin{theorem}\label{Theorem main HE1-L1 estimate}
Let $W_{j,k}$ be defined as in {\rm (\ref{definition of Ujk})}, where $D$ is defined as above, and $z\in\mathbb{C}$ has real part $\Re(z)=-\frac{1}{m+s}$. If $\eta\geq 1$, then we can decompose $\mathbb{Z}^2$ into three disjoint subsets $\Delta_1,\Delta_2$ and $\Delta_3$ such that each operator $W_i=\sum_{(j,k)\in \Delta_i}W_{j,k}$ satisfies
\begin{equation*}
\|W_1f\|_{L^{1,\infty}}\leq C\|f\|_{L^1},~~~
\|W_3f\|_{L^1}\leq C\|f\|_{L^1}
\end{equation*}
and, for $f\in H_E^1(I_k)$ and $(j,k)\in \Delta_2$,
\begin{equation*}
\|W_{j,k}f\|_{L^1}\leq C(1+|z|)\|f\|_{ H_E^1  },
\end{equation*}
where the above constants $C$ depend only on $deg(S)$ and $\varphi$.
\end{theorem}

\begin{proof}
Assume $\varphi$ is supported in $|x|,|y|<1/2$. Our proof will be divided into in two steps.\\

$\mathbf{Step~1.~~m>0~or~|\alpha_{i_t}-\alpha_{i_s}|\geq \frac{|\alpha_{i_s}|}{4}~for~some~1\leq t \leq s.}$\\

$\mathbf{{Case~(i)}}$~$2^j>|\alpha_{i_s}|2^{(k+1)\eta+2}.$\\

The set of all these $(j,k)$ is denoted by $\Delta_1$. Then $W_1$ satisfies
$\|W_1f\|_{L^{1,\infty}}\leq C\|\varphi\|_{\infty}\|f\|_{L^1}.$\\

$\mathbf{{Case~(ii)}}$~$|\alpha_{i_s}| 2^{(k-1)\eta-2}\leq 2^j \leq |\alpha_{i_s}| 2^{(k+1)\eta+2}.$\\

Let $\Delta_2$ be the set of these $(j,k)$. Then $\|W_2f\|_{L^{1}}\leq C\|\varphi\|_{\infty}\|f\|_{L^1}$. This can be verified by Fubini's theorem.

In fact, we have
\begin{eqnarray*}
\int_{\bR}|W_2f(x)|dx
&\leq&
C\int_{\bR}\int_{|\alpha_{i_s}||y|^{\eta}\approx |x|}|D(x,y)|^{-1/(m+s)}
|\varphi(x,y)||f(y)|dydx\\
&\leq&
C\|\varphi\|_{\infty}\int_{|y|<1/2}\l(\int_{|x|\approx|\alpha_{i_s}||y|^{\eta}}
|D(x,y)|^{-1/(m+s)}dx\r)|f(y)|dy\\
&\leq&
C\|\varphi\|_{\infty}\int_{|y|<1/2}\l(\int_{|x|\approx 1}
|D(x,|\alpha_{i_s}|^{-1/\eta})|^{-1/(m+s)}dx\r)
|f(y)|dy\\
&\leq &
C\|\varphi\|_{\infty}\|f\|_{L^1},
\end{eqnarray*}
where we have used the assumption that either $m>0$ or $|\alpha_{i_t}-\alpha_{i_s}|>|\alpha_{i_s}|/4$ for some $t$. All of the above constants $C$ depend only on $m,s,\eta$.\\

$\mathbf{{Case~(iii)}}$~$2^j<|\alpha_{i_s}|2^{(k-1)\eta-2}.$\\

Let $\Delta_3$ be the set of these pairs $(j,k)$. By Fubini's theorem, we can prove that each $W_{j,k}$ is bounded on $L^1$. A change of variables yields
\begin{eqnarray*}
\int |W_{j,k}f(x)|dx
&\leq&
\int\int_{|\alpha_{i_s}|y^{\eta}>2|x|}|D(x,y)|^{ -\frac{1}{m+s}} |\varphi(x,y)||f(y)|dydx\\
&\leq &
C \|\varphi\|_{\infty} \int_{|y|<1/2}\l(\int_{|x|<1/2}
|D(x,|\alpha_{i_s}|^{-1/\eta})|^{ -\frac{1}{m+s}} dx\r) |f(y)| dy\\
&\leq&
C \|\varphi\|_{\infty} \|f\|_{L^1}.
\end{eqnarray*}

Combining all results in the above three cases, we have completed the proof of the theorem for Step 1.\\

$\mathbf{Step~2.~~\Theta=\{t_r+1,\cdots,t_{r+1}\}.}$\\

Let $\Delta_1,\Delta_2,\Delta_3$ be defined as in Step 1. By the same argument as above, we have $\|W_1f\|_{L^1}\lesssim \|f\|_{L^1}$ and $\|W_3f\|_{L^1}\lesssim \|f\|_{L^1}$, where the implicit constants depending only on $deg(S)$ and the cut-off $\varphi$.

Now we turn to the desired estimate associated with $\Delta_2$. As in Section 3, we shall first treat real roots, and then consider complex roots.\\

\textbf{Case (i) All $a_{i_t}$'s are real numbers.}

By the method of resolution of singularities in Section 3, we have
\begin{equation*}
W_{j,k}
=
\sum W_{j,k,l_0,\cdots,l_{w-1}}^{\sigma_0,\cdots,\sigma_{w-1}},~~~(j,k)\in\Delta_2,
\end{equation*}
where each term in the summation is of a class $\mathcal{G}:=\{G_{i,1},G_{i,2},G_{i,3}\}_{i=0}^w$. Similarly, these three tuples satisfy the properties in Section 3, i.e.,
(i) $G_{i,u}\cap G_{i,v}=\emptyset$ for $u\neq v$;
(ii) If $G_{i-1,3}$ is not empty, then $G_{i,2}\neq \emptyset$;
(iii) $G_{i-1,3}=G_{i,1}\cup G_{i,2}\cup G_{i,3}$ and $G_{w,3}=\emptyset$.

For any three-tuple $\mathcal{G}:=\{G_{i,1},G_{i,2},G_{i,3}\}_{i=0}^w$ with length $w\geq 2$, we claim that
\begin{equation}\label{sec4 w not less than 2 bounded on L1}
\Big\|
\sum\nolimits_{\mathcal{G}}
W_{j,k,l_0,\cdots,l_{w-1}}^{\sigma_0,\cdots,\sigma_{w-1}}f
\Big\|_{L^1}
\leq
C(deg(S),\varphi) \|f\|_{L^1}.
\end{equation}
Here the summation is taken over all operators of class $\mathcal{G}$. This method of resolution of singularities, described as in Section 3, implies that (i) the number of $l_0,l_1,\cdots,l_{w-2}$ is bounded by a constant $C=C(\eta)$, and (ii) $l_{i-1}\leq l_i+C(\eta)$. Hence for $w\geq 2$ we have
\begin{equation}\label{sec4 lower bound for Djk}
|D_{j,k}(x,y)|\geq \prod_{t=1}^s |x-\alpha_{i_t}y^{\eta}|
\geq
C 2^{l_{w-2}(s-|G_{w,1}|-|G_{w,2}|)}
2^{l_{w-1}(|G_{w,1}|+|G_{w,2}|)}
\end{equation}
in the support of $W_{j,k,l_0,\cdots,l_{w-1}}^{\sigma_0,\cdots,\sigma_{w-1}}$. Our assumption $w\geq 2$ also implies $0<|G_{w,1}|+|G_{w,2}|<s$. It follows immediately that
\begin{equation*}
\left|\sum_{\mathcal{G}} W_{j,k,l_0,\cdots,l_{w-1}}^{\sigma_0,\cdots,\sigma_{w-1}}f(x) \right|
\leq
\sum_{\mathcal{G}} \int_{\bR}
\Phi\l(\sigma_{w-1}\frac{x-\alpha_{e_{w-1}}y^{\eta}  }{2^{l_{w-1}}}\r)
2^{-l_{w-2}a}
2^{-l_{w-1}b}|f(y)|dy
\end{equation*}
with $a=\frac{s-|G_{w,1}|-|G_{w,2}|}{s}>0$ and $b=\frac{|G_{w,1}|+|G_{w,2}|}{s}>0$. As a consequence, we see that
\begin{equation*}
\left\|\sum_{\mathcal{G}} W_{j,k,l_0,\cdots,l_{w-1}}^{\sigma_0,\cdots,\sigma_{w-1}} f \right\|_{L^1}
\lesssim
\sum_{l_{w-1}\leq l_{w-2}+C(\eta)}2^{-l_{w-2}a}
2^{l_{w-1}a}\|f\|_{L^1}
\lesssim
\|f\|_{L^1}.
\end{equation*}

Assume $\mathcal{G}:=\{G_{i,1},G_{i,2},G_{i,3}\}_{i=0}^w$ is a three-tuple with length $w=1$. If $|G_{1,2}|<s$, then the above estimate (\ref{sec4 w not less than 2 bounded on L1}) is also true. Under this assumption, we have
\begin{equation*}
\min_{t\in G_{1,1}}|\alpha_t-\alpha_{e_0}|2^{k\eta}
\geq
2^{l_0+N_0(\eta)}.
\end{equation*}
This implies that $D_{j,k}$ has a lower bound similar to (\ref{sec4 lower bound for Djk}), i.e,,
\begin{equation*}
|D_{j,k}(x,y)|\geq \prod_{t=1}^s |x-\alpha_{i_t}y^{\eta}|
\geq
C \left(\min_{t\in G_{1,1}} |\alpha_t-\alpha_{e_0}|2^{k\eta} \right)^{|G_{1,1}|}
2^{l_{0}|G_{1,2}|}.
\end{equation*}
The same argument as in the case $w\geq2$ gives the desired $L^1$ estimate.

To establish the $H_E^1\rightarrow L^1$ estimate, it suffices to consider the simplest three-tuple $\mathcal{G}$ for which $w=1$ and $G_{1,2}=G_{0,3}=\Theta$. In other words, we have to show
\begin{equation}\label{sec4 reduced H1L1 esti}
\left\|  \sum_{\mathcal{G}} W_{j,k,l_0}^{\sigma_0} f    \right\|_{L^1}
\leq
C(deg(S),\varphi)
\|  f  \|_{H_E^1(I_k)}.
\end{equation}
For those small $l_0$ satsifying
\begin{equation}\label{sec4 def of E}
2^{l_0s}\leq E:=\left(  |\lambda| |\alpha_{i_s}|^{-1} 2^{-k(\eta-1)} A \right)^{-\frac{s}{s+2}  },
\end{equation}
we shall prove $\sum_{\mathcal{G},2^{l_0s}\leq E} W_{j,k,l_0}^{\sigma_0}$ is also bounded on $L^1$. The proof is similar to the above case $w\geq 2$. In fact, the smallness condition on $l_0$ implies
\begin{equation*}
\left| \sum_{\mathcal{G},2^{l_0s}\leq E} W_{j,k,l_0}^{\sigma_0} f(x) \right|
\lesssim
\sum_{2^{l_0s}\leq E}\int_{\bR}\Phi\l(\sigma_{0}\frac{x-\alpha_{e_{0}}y^{\eta}  }{2^{l_{0}}}\r)
E^{-1/s}|f(y)|dy
\end{equation*}
from which it follows that
\begin{equation*}
\left\| \sum_{\mathcal{G},2^{l_0s}\leq E} W_{j,k,l_0}^{\sigma_0} f \right\|_{L^1}
\lesssim
\sum_{2^{l_0s}\leq E} 2^{l_0} E^{-1/s} \|f\|_{L^1}
\lesssim \|f\|_{L^1}.
\end{equation*}
With this estimate, we see that it remains to show (\ref{sec4 reduced H1L1 esti})
with the summation being taken over $\mathcal{G}$ and $2^{l_0}> E^{1/s}$.

First, we claim that
\begin{equation}\label{sec4 an L2 esti}
\left\| \sum_{\mathcal{G},2^{l_0s}> E} W_{j,k,l_0}^{\sigma_0} f \right\|_{L^2}
\leq
C(deg(S),\varphi)\Big(
|\alpha_{i_s}|^{-1}2^{-k(\eta-1)}
\Big)^{1/2}
\|f\|_2.
\end{equation}
Observe that each operator $W_{j,k,l_0}^{\sigma_0}$ is supported in the horizontally and vertically convex domain $\Omega_{j,k,l_0}^{\sigma_0}$, defined as in (\ref{sec3 HV domain support of Wjkl}). Here we also assume $\Phi(x)>0$ on $(\frac12,2)$, as in the definition (\ref{sec3 HV domain support of Wjkl}). The above $L^2$ estimate follows from both Lemma \ref{operator van der Corput} and \ref{Lemma almost orth van der Corput}. For example, we now deduce the $L^2$ estimate from Lemma \ref{operator van der Corput}. Recall that $\Omega_{j,k,l_0}^{\sigma_0}$ consists of all points $(x,y)$ satisfying
\begin{equation*}
2^{j-1}\leq x \leq 2^{j+1},~~~
2^{k-1}\leq y \leq 2^{k+1},~~~
2^{l_0-1}\leq \sigma_0(x-\alpha_{e_0}y^{\eta}) \leq 2^{l_0+1}.
\end{equation*}
It is easy to see that $\Omega_{j,k,l_0}^{\sigma_0}$ is a curved-trapezoid in Definition \ref{sec2 curved trapezoid} with the roles of $x$ and $y$ interchanged. For example, if $\sigma_0=+$ then $\Omega_{j,k}^{l_0}$ can be written as $2^{k-1}\leq y \leq 2^{k+1}$ and $g(y)\leq x \leq h(y)$ with $g$ and $h$ given by
\begin{equation*}
g(y)=\max\{2^{j-1},2^{l_0-1}+\alpha_{e_0}y^{\eta}\}~~~{\rm and}~~~
h(y)=\min\{2^{j+1},2^{l_0+1}+\alpha_{e_0}y^{\eta}\}.
\end{equation*}
One easily verifies that all assumptions in Lemma \ref{operator van der Corput} are true for each operator appearing in the above $L^2$ estimate. Hence we have
\begin{eqnarray*}
\sum_{\mathcal{G},2^{l_0s}> E} \|W_{j,k,l_0}^{\sigma_0} f\|_{L^2}
&\leq&
C\sum_{\mathcal{G},2^{l_0s}> E}\Big(|\lambda| 2^{l_0s} A \Big)^{-1/2}2^{-l_0} \|f\|_{L^2}\\
&\leq&
 C\Big(|\lambda|  A\Big)^{-1/2} E^{-\frac{s+2}{2s}  } \|f\|_{L^2}\\
&=&C\Big(|\alpha_{i_s}|^{-1}2^{-k(\eta-1)}\Big)^{1/2}
\|f\|_2.
\end{eqnarray*}

Assume $\textbf{a}$ is an $H_E^1(I_k)$ atom. Then for some interval $I\subseteq I_k$, the properties (i), (ii) and (iii) in Definition \ref{sec4 def a variant of H1} are true. Let $J$ be the interval defined by
\begin{equation*}
J=\{x\in\bR:~|x-\alpha_{e_0}c_I^{\eta}| \leq C(\eta)|\alpha_{i_s}| 2^{k(\eta-1)} |I| \},
\end{equation*}
where $c_I$ is the center of $I$, and $C(\eta)$ is a suitably large constant depending only on $\eta$.

By the Cauchy-Schwarz inequality, we can make use of the $L^2$ estimate (\ref{sec4 an L2 esti}) to obtain
\begin{eqnarray*}
\left\| \sum_{\mathcal{G},2^{l_0s}> E} W_{j,k,l_0}^{\sigma_0} a \right\|_{L^1(J)}
&\leq&
|J|^{1/2}
\left\| \sum_{\mathcal{G},2^{l_0s}> E} W_{j,k,l_0}^{\sigma_0} a \right\|_{L^2}\\
&\leq&
C\Big(|\alpha_{i_s}| 2^{k(\eta-1)} |I|\Big)^{1/2}   \Big(|\alpha_{i_s}| 2^{k(\eta-1)} \Big)^{-1/2}|I|^{-1/2}\\
&\leq& C.
\end{eqnarray*}
With this $L^1$ estimate over $J$, it is enough to prove a weaker form of the desired estimate,
\begin{equation}\label{sec 4 last reduced H1L1 esti}
\left\| \sum_{\mathcal{G},~2^{l_0}> \max\{E^{1/s},F\}} W_{j,k,l_0}^{\sigma_0} a \right\|_{L^1}
\leq C(deg(S),\varphi),~~~F:=|\alpha_{i_s}| 2^{k(\eta-1)}|I|.
\end{equation}
The reason for this statement is that if $2^{l_0}\leq |\alpha_{i_s}| 2^{k(\eta-1)} |I|$ then $W_{j,k,l_0}^{\sigma_0} a$ is supported in $J$, provided that the constant $C(\eta)$, appearing in the definition of $J$, is sufficiently large. Let $K_{l_0}$ be the following kernel:
\begin{equation}\label{sec 4 defn of K}
K_{l_0}(x,y)= \Phi\Big( \frac{x}{2^j} \Big)
\Phi\Big( \frac{y}{2^k} \Big)
\Phi\Big(\sigma_0 \frac{x-\alpha_{e_0}y^{\eta}  }{2^{l_0}  } \Big)
\Big( D_{j,k}(x,y) \Big)^z \varphi(x,y).
\end{equation}
Now we are going to prove the following estimate:
\begin{equation}\label{sec4 Hormander cond}
\mathop{\sup}_{y\in I}
\int_{J^c} \sum_{\mathcal{G},~2^{l_0}> \max\{E^{1/s},F\}} |K_{l_0}(x,y)-K_{l_0}(x,c_I)|dx
\leq
C(deg(S),\varphi).
\end{equation}
Set
\begin{equation}\label{sec4 defn of K}
M_{j,k,l_0}(x,y)= \Phi\Big( \frac{x}{2^j} \Big)
\Phi\Big( \frac{y}{2^k} \Big)
\Phi\Big(\sigma_0 \frac{x-\alpha_{e_0}y^{\eta}  }{2^{l_0}  } \Big)
\varphi(x,y).
\end{equation}
Then in the support of $W_{j,k,l_0}^{\sigma_0}$, we can apply the mean value theorem to obtain
\begin{eqnarray*}
|M_{j,k,l_0}(x,y)-M_{j,k,l_0}(x,c_I)|
&\leq& \mathop{\sup}_{\xi\in I} |\partial_yM_{j,k,l_0}(x,\xi)|~|I|\\
&\leq&
C\left( 2^{-k} + |\alpha_{i_s}|2^{k(\eta-1)}2^{-l_0}  \right)|I|\\
&\leq&
C \frac{|\alpha_{i_s}|2^{k(\eta-1)} |I|}{2^{l_0}}
\end{eqnarray*}
since $k\leq 0$, $|\alpha_{e_0}|\approx |\alpha_{i_s}|$ and $l_0\leq |\alpha_{i_s}|2^{k\eta}+C(\eta)$. As in Section 2, we use $I_{\Omega_{j,k,l_0}^{\sigma_0}}(x)$ to denote the vertical cross-section $\{y:(x,y)\in \Omega_{j,k,l_0}^{\sigma_0}\}$. Without loss of generality, we assume $\Omega_{j,k,l_0}^{\sigma_0}$ is nonempty. Define an expanded domain $\Omega_{j,k,l_0}^{\sigma_0\ast}$ as in (\ref{sec3 expanded domain general}) with some small $\epsilon=\epsilon(\eta)>0$. A simple but useful observation is that if $I_{\Omega_{j,k,l_0}^{\sigma_0}}(x)$ is nonempty, then we have
\begin{equation*}
{\delta}_{\Omega_{j,k,l_0}^{\sigma_0\ast}}(x)
:=\Big|I_{\Omega_{j,k,l_0}^{\sigma_0\ast}}(x)\Big|
\approx\frac{2^{l_0}}{|\alpha_{i_s}|2^{k(\eta-1)}}.
\end{equation*}
Now we shall prove this observation. Assume $(x,y_0)\in I_{\Omega_{j,k,l_0}^{\sigma_0}}(x)$. For this fixed $x$, the expanded interval $I_{\Omega_{j,k,l_0}^{\sigma_0\ast}}(x)$ is the intersection of $I_1=(2^{k-1}-\epsilon 2^k,~2^{k+1}+\epsilon2^k)$ and the following interval
\begin{eqnarray*}
I_2
&=&
\{y>0:~2^{l_0-1}-\epsilon 2^{l_0}\leq \sigma_0(x-\alpha_{e_0}y^{\eta}) \leq 2^{l_0+1}
+\epsilon 2^{l_0}\}.
\end{eqnarray*}
If one of the endpoints of $I_1$, say $2^{k-1}-\epsilon 2^k$, belongs to $I_2$, then $(2^{k-1}-\epsilon 2^k,y_0)\subseteq I_1\cap I_2$. Thus $I_{\Omega_{j,k,l_0}^{\sigma_0\ast}}(x)$ has length $\gtrsim 2^k\gtrsim 2^{l_0}/(|\alpha_{e_0}|2^{k(\eta-1)})$. Now assume both endpoints of $I_1$ do not lie in $I_2$. It follows from $I_1\cap I_2\neq \emptyset$ that $I_2\subseteq I_1$. On the other hand, the $y$-derivative of the function $\sigma_0(x-\alpha_{e_0}y^{\eta})$ has size comparable to $|\alpha_{e_0}|2^{k(\eta-1)}$ on $I_1$. This implies $$|I_2|\approx 2^{l_0}/(|\alpha_{e_0}|2^{k(\eta-1)}).$$

Since $\prod_{t\in\Theta}(x-\alpha_{t}y^{\eta})\in \bR[x,y]$ and $|D_{j,k}(x,y)|\approx 2^{l_0s}$ for $l_0$ satisfying $2^{l_0s}>E$, we can apply Lemma \ref{sec2 poly type func lemma} to obtain
\begin{eqnarray*}
& &\left|\Big(D_{j,k}(x,y)\Big)^z - \Big(D_{j,k}(x,c_I)\Big)^z \right|\\
&\leq&
C |z| \mathop{\sup}_{\xi\in I\cap I_{\Omega_{j,k,l_0}^{\sigma_0}}(x)}
\Big(D_{j,k}(x,\xi)\Big)^{\Re(z)-1}
\mathop{\sup}_{\xi\in I\cap I_{\Omega_{j,k,l_0}^{\sigma_0}}(x)}
|\partial_y D_{j,k}(x,\xi)|~|I|\\
&\leq&
C |z| \mathop{\sup}_{\xi\in I_{\Omega_{j,k,l_0}^{\sigma_0\ast}}(x)}\Big(D_{j,k}(x,y)\Big)^{\Re(z)-1}
\Big({\delta}_{\Omega_{j,k,l_0}^{\sigma_0\ast}}(x) \Big)^{-1}
\mathop{\sup}_{\xi\in I_{\Omega_{j,k,l_0}^{\sigma_0\ast}}(x)} |D_{j,k}(x,\xi)|~|I|\\
&\leq&
C|z|\frac{|\alpha_{i_s}|2^{k(\eta-1)}|I|}{2^{2l_0}}.
\end{eqnarray*}
It follows from the above two estimates of $M_{j,k,l_0}$ and $D_{j,k}$ that
\begin{eqnarray*}
\sum_{\mathcal{G},~2^{l_0}> \max\{E^{1/s},F\}} \int_{J^c}|K_{l_0}(x,y)-K_{l_0}(x,c_I)|dx
&\leq&
C\sum_{\mathcal{G},~2^{l_0}> F} 2^{-l_0}\Big(|\alpha_{i_s}|2^{k(\eta-1)} |I| \Big)\\
&\leq& C.
\end{eqnarray*}
Here $F=|\alpha_{i_s}|2^{k(\eta-1)} |I|$ has been defined as in (\ref{sec 4 last reduced H1L1 esti}).

With the estimate (\ref{sec4 Hormander cond}), we can deduce
(\ref{sec 4 last reduced H1L1 esti}) from the following estimate:
\begin{equation*}
\sum_{\mathcal{G},~2^{l_0}> \max\{E^{1/s},F\}}
\int_{J^c}|K_{l_0}(x,c_I)|\l| \int_{I} e^{i\lambda S(x,y)} \textbf{a}(y) dy \r|dx
\leq C(deg(S),\varphi).
\end{equation*}
By the vanishing property of the atom $\textbf{a}$,
\begin{eqnarray*}
\l|\int_Ie^{i\lambda S(x,y)}\textbf{a}(y)dy\r|
&=&
\l|\int_I\l(e^{i\lambda[S(x,y)-S(x,c_I)]}
-
e^{i\lambda[S(\alpha_{i_u}c_I^{\eta},y)-S(\alpha_{i_u}c_I^{\eta},c_I)]}\r)
\textbf{a}(y)dy\r|\\
&\leq&
C|\lambda| \int_{I}
\Big|\int_{\alpha_{i_u}c_I^{\eta}}^x \int_{c_I}^y \partial_{x}\partial_{y}S(\xi,\zeta)d\xi d\zeta \Big|~|\textbf{a}(y)|dy\\
&\leq&
C|\lambda|~A~|x-\alpha_{i_u}c_I^{\eta}|^{s+1}~|I|,
\end{eqnarray*}
where $x\in J^c$, $y\in I$, and $(x,y)$ also lies in the support of $W_{j,k,l_0}^{\sigma_0}$. Choose a positive number $\mu$ such that
$|\lambda|A\mu^{s+1}|I|=1$. Then it is clear that
\begin{eqnarray*}
\sum_{\mathcal{G},~2^{l_0}> \max\{E^{1/s},F\},2^{l_0}\leq \mu}
\int_{J^c}|K_{l_0}(x,c_I)|\l| \int_{I} e^{i\lambda S(x,y)} \textbf{a}(y) dy \r|dx
\leq
C\sum_{2^{l_0}\leq \mu} |\lambda|2^{l_0(s+1)}A|I|\leq C.
\end{eqnarray*}

To estimate the above summation with $2^{l_0}\leq \mu$ replaced by $2^{l_0}>\mu$, we shall apply the operator version of van der Corput lemma. We shall mention that both Lemma \ref{operator van der Corput} and Lemma \ref{Lemma almost orth van der Corput} are applicable now. First observe that if $2^{l_0}>F=|\alpha_{i_s}|2^{k(\eta-1)}|I|$, then $|x-\alpha_{e_0}y^{\eta}|\approx |x-\alpha_{e_0}c_I^{\eta}|$ in $\Omega_{j,k,l_0}^{\sigma_0}$. This implies $|K_{l_0}(x,c_I)|\approx |K_{l_0}(x,y)|$. On the other hand, for $2^{l_0}>F=|\alpha_{i_s}|2^{k(\eta-1)}|I|$, we see that all horizontal cross-sections of $\Omega_{j,k,l_0}^{\sigma_0}$ is contained in the interval $(-2^{l_0+N_0}+\alpha_{e_0}c_I^{\eta},~~2^{l_0+N_0}+\alpha_{e_0}c_I^{\eta})$ for some large $N_0=N_0(\eta)\geq 1$. By Cauchy-Schwarz's inequality, we can apply Lemma \ref{operator van der Corput} to obtain
\begin{eqnarray*}
& &\sum_{l_0}
\int_{J^c}|K_{l_0}(x,c_I)|\l| \int_{I} e^{i\lambda S(x,y)} \textbf{a}(y) dy \r|dx\\
&\lesssim&
\sum_{l_0}
\int_{\bR}\l|\int_{\bR} e^{i\lambda S(x,y)} K_{l_0}(x,y) \textbf{a}(y) dy\r|dx\\
&\lesssim&
\sum_{l_0} 2^{l_0/2}
\left(\int_{\bR}\l|\int_{\bR} e^{i\lambda S(x,y)} K_{l_0}(x,y) \textbf{a}(y) dy\r|^2dx \right)^{\frac12}\\
&\lesssim&
\sum_{l_0}
2^{l_0/2} 2^{-l_0} \Big(  |\lambda| 2^{l_0s} A \Big)^{-1/2}|I|^{-\frac12}\\
&\leq& C,
\end{eqnarray*}
where all of the above summations are taken over $l_0$ in the same range described as above, i.e., $2^{l_0}>\max\{E^{1/s},F,\mu\}$.\\

$\mathbf{Case~(ii).~At~least~one~of~\alpha_{i_1},\cdots,\alpha_{i_s}~is~complex. }$\\

If $|\Re(\alpha_{i_t})| \leq |\Im(\alpha_{i_t})|$ for some $t$, say $t_0$, then $W_{j,k}$ is bounded on $L^1$. In fact, the kernel of
$W_{j,k}$ is bounded by $C\prod_{t=1}^{s}(x-\alpha_{i_t}y^{\eta})$. By the Riesz rearrangement inequality, for each $y\in[2^{k-1},2^{k+1}]$, we have
\begin{eqnarray*}
\int_{2^{j-1}}^{2^{j+1}}
\prod_{t=1}^{s}|x-\alpha_{i_t}y^{\eta}|^{-\frac1s}
dx
&\lesssim&
\l(|\Im(\alpha_{i_{t_0}})|2^{k\eta}\r)^{-\frac1s}
\int_{2^{j-1}}^{2^{j+1}}
\prod_{t\neq t_0} |x-\alpha_{i_t}y^{\eta}|^{-\frac1s}dx\\
&\lesssim&
2^{-j/s}
\int_{|x|\leq 2^{j+1} }
|x|^{-\frac{s-1}{s}}dx<\infty.
\end{eqnarray*}
Hence we have $\|W_{j,k}f\|_{L^1}\lesssim \|f\|_{L^1}$.

Now assume $|\Re(\alpha_{i_t})| > |\Im(\alpha_{i_t})|$ for all $t$.
The treatment is similar to that of $\mathbf{Step~2}$ in our proof of Theorem \ref{Theorem L2 damping estimate}. We first apply the method of resolution of singularities for complex roots in Section 3 to obtain $W_{j,k}=\sum W_{j,k,l_0,\cdots,l_{w-1}}^{\sigma_0,\cdots,\sigma_{w-1}}$.
Let $\mathcal{G}=\{G_{i,1},G_{i,2},G_{i,3}\}_{i=0}^w$ be a sequence of three tuples satisfying the three properties described in Case (i). By the same argument as above, we are able to prove the $L^1$ estimate (\ref{sec4 w not less than 2 bounded on L1}) under one of the following assumptions:\\

\noindent
$\bullet$ $w\geq 2$.\\
$\bullet$ $w=1$, $|G_{1,1}|>0$, and $|G_{1,2}|>0$.\\
$\bullet$ $w=1$, $|G_{1,1}|=0$, and $2^{l_0s}\leq E$, where $E$ is given by (\ref{sec4 def of E}).\\
$\bullet$ $w=1$, $|G_{1,1}|=0$, $2^{l_0s}>E$, and $2^{l_0}\leq \max_{t}|\Im(\alpha_{i_t})|2^{k\eta}$.\\

\noindent It is enough to consider the case $w=1$, $|G_{1,1}|=0$, $2^{l_0s}>E$, and $2^{l_0}> \max_{t}|\Im(\alpha_t)|2^{k\eta}$. Then it is easy to see that $|D_{j,k}(x,y)|\approx 2^{l_0s}$. Moreover, the argument in Case (i) applies without essential change. Therefore the proof of the theorem is complete.
\end{proof}

\begin{remark}\label{Remark H1L1 estimate for small eta}
As in Section 3, we can treat the case $\eta<1$ by writing $S''_{xy}$ as in {\rm(\ref{Hessian of phase II})}. Similarly, we shall define the damping factor $D$ in the first case, i.e., either $m>0$ or $\max_t|\beta_{i_t}-\beta_{i_s}|\geq \frac{|\beta_{i_s}|}{4}$, as
$D(x,y)=x^m\prod_{t=1}^s(x^{\nu}-\beta_{i_t}y)$
with $\Re(z)=-1/(m+s\nu)$.
While for the second case, i.e., $m=0$ and $\max_t|\beta_{i_t}-\beta_{i_s}|< \frac{|\beta_{i_s}|}{4}$, we take $z$ with $\Re(z)=-\frac{1}{s\nu}$ and define $D_{j,k}$ by
\begin{equation*}
D_{j,k}(x,y)=
\prod_{t=1}^{s}\big(x^{\nu}-\beta_{i_t}y\big).
\end{equation*}
It should be pointed out that our treatment for $\eta<1$ is much simpler than that for $\eta\geq1$. In fact, the following integral of Hilbert type is finite:
\begin{equation*}
\int_{|x|^{\nu}\approx |\beta_{i_s}y|}\left|x^{\nu}-\beta_{i_t}y\right|^{-1/\nu}dx
\leq
C(\nu),~~~\nu=\frac{1}{\eta}>1,
\end{equation*}
provided that $|\beta_{i_t}|\approx|\beta_{i_s}|$ for all $t$. Hence, by Fubini's theorem and H\"{o}lder's inequality, we see that $W_{j,k}$ is bounded on $L^1$ for each $(j,k)\in\Delta_2$. Other statements in Theorem {\rm\ref{Theorem main HE1-L1 estimate}} are also true for $\eta<1$.
\end{remark}

\section{Proof of the main result}
In this section, we shall prove Theorem \ref{Theorem Main result uniform Lp decay} by exploiting damping estimates established in Sections 3 and 4. We only consider $T_{\lambda}$ in the first quadrant. By a change of variables, $T_{\lambda}$ in other quadrants can be treated in the same way.\\

\begin{proof}
Throughout this section, we choose $s$ indices
$\{i_1,i_2,\cdots,i_s\}=\{1,2,\cdots,s\}.$
Then the damping factor $D$ in Theorem \ref{Theorem L2 damping estimate} is equal to
\begin{equation}\label{Special damping factor}
D(x,y)=x^m\prod_{t=1}^{s}(x-\alpha_t y^{\eta}).
\end{equation}
Recall that the coefficients $\alpha_i$ are ordered such that $|\alpha_1|\leq |\alpha_2|\leq \cdots \leq |\alpha_N|$.
As in Sections 3 and 4, we assume the integer $s$ satisfies $\prod_{t=1}^{s}(x-\alpha_t y^{\eta})\in\bR[x,y]$. This means that the above product is conjugate invariant in the sense of Remark \ref{sec2 remark on factorization}.

We first assume $\eta\geq 1$ and $m+s\geq n+(N-s)\eta$. As in Section 4, we shall divide our proof here into two steps.\\

$\mathbf{Step~1.~Either~m>0~or~|\alpha_{t}-\alpha_{s}|\geq \frac{|\alpha_{s}|}{4}~is~true~for~some~1\leq t \leq s.}$\\

Choose a bump function $\Phi$ as in the previous two sections. In other words, $\Phi\in C^{\infty}(\bR)$ satisfies (i) $\supp(\Phi)\subseteq [1/2,2]$, and (ii) $\sum_{j\in\mathbb{Z}}\Phi(x/2^j)=1$ for $x>0$. Define $W_{j,k}$ by
\begin{equation}\label{sec5 def of Wjk}
W_{j,k}f(x)=\int_{\bR}e^{i\lambda S(x,y)}|\widetilde{D}(x,y)|^z\varphi(x,y)
\Phi\l(\frac{x}{2^j}\r)\Phi\left(\frac{y}{2^k}\right)f(y)dy,
\end{equation}
where the damping factor $\widetilde{D}$ is slightly different from $D$, and $z\in \mathbb{C}$ lies in the following strip:
\begin{equation*}
-\frac{1}{m+s}
=:\alpha \leq
\Re(z)
\leq
\beta:=\frac{m+s-n-(N-s)\eta}{2[n+(N-s)\eta+1]}\cdot\frac{1}{m+s}.
\end{equation*}

$\mathbf{Case~(i)~2^j \geq |\alpha_{s}|2^{(k+1)\eta+2}.}$\\

Let $\Delta_1$ be the set of all $(j,k)\in\mathbb{Z}^2$ in $\mathbf{Case~(i)}$. In this case, we define $\widetilde{D}(x,y)=x^{m+s}$. Then $W_1=\sum_{(j,k)\in\Delta_1} W_{j,k}$ satisfies
\begin{equation*}
\|W_1f\|_{L^2}
\leq C (1+|z|^2)
\l(|\lambda| \prod_{k\notin\{1,\cdots,s\}}|\alpha_k|\r)^{-\gamma}
\|f\|_{L^2},~~~\Re(z)=\beta,~~~\gamma=\frac{1}{2(n+(N-s)\eta+1)}
\end{equation*}
and
\begin{equation*}
\|W_1f\|_{L^{1,\infty}}
\leq
C \|f\|_{L^1},
~~~\Re(z)=-\frac{1}{m+s},
\end{equation*}
where both constants $C$ depend only on $deg(S)$ and the cut-off $\varphi$. The above $L^2$ decay estimate for $W_1$ can be proved by the same argument as that of Theorem \ref{Theorem L2 damping estimate}. Indeed, for each $(j,k)\in\Delta_1$, it is easy to see that $\widetilde{D}$ (the partial derivatives of $\widetilde{D}$, respectively) has the same upper bounds, which were used in the proof of Theorem \ref{Theorem L2 damping estimate}, as the damping factor $D$ (the partial derivatives of $D$, respectively). Hence following the proof of Theorem \ref{Theorem L2 damping estimate}, we can establish the above $L^2$ decay estimate.
The $L^1\rightarrow L^{1,\infty}$ estimate for $W_1$ is obvious; see also Section 4.

By Lemma \ref{Lemma interpolaiton with change of weights}, with $T$ defined by $W_1f(x)=|x|^{(m+s)z}Tf(x)$, $a=(m+s)\beta$ and $p_0=2$, we obtain
\begin{equation}\label{Lp decay for W1}
\|W_1f\|_{L^p}\leq C (1+|z|^2)
\left(|\lambda|\prod_{k\notin\{1,\cdots,s\}}|\alpha_k|\right)
^{-\delta}\|f\|_{L^p}~~~{\rm with}~~~\Re(z)=0,
\end{equation}
where $\delta$ and $p$ are given by
\begin{equation}
\delta=\frac{1}{m+s+n+(N-s)\eta+2},
~~~~~p=\frac{m+s+n+(N-s)\eta+2}{m+s+1}.
\end{equation}
Here we choose a $\theta\in [0,1]$, appearing in Lemma \ref{Lemma interpolaiton with change of weights}, satisfies $\theta\beta+(1-\theta)\alpha=0$, i.e.,
\begin{equation}\label{sec5 convex para theta}
\theta=\frac{-\alpha}{\beta-\alpha}=2\frac{n+(N-s)\eta+1}{m+s+n+(N-s)\eta+2}.
\end{equation}

$\mathbf{Case~(ii)~|\alpha_{s}|2^{(k-1)\eta-2}<2^j
<|\alpha_{s}|2^{(k+1)\eta+2}}.$\\

The set of all $(j,k)\in\bZ^2$ satisfying $\mathbf{Case~(ii)}$ is denoted by $\Delta_2$. As in Section 4, we define $W_2=\sum_{\Delta_2}W_{j,k}$. In this case, put $\widetilde{D}(x,y)=D(x,y)$. Then it is also true that
\begin{equation*}
\|W_2f\|_{L^2}
\leq
C(1+|z|^2)\left( |\lambda| \prod_{k\notin \{1,\cdots,s\}}|\alpha_k|   \right)^{-\gamma}\|f\|_{L^2},~~~\gamma=\frac{1}{2[n+(N-s)\eta+1]}
\end{equation*}
for $z\in\mathbb{C}$ with $\Re(z)=\beta$. On the other hand, as in Section 4, it is easy to prove that $W_2$ satisfies
\begin{equation*}
\|W_2f\|_{L^1}\leq  C\|f\|_{L^1},~~~\Re(z)=\alpha=-\frac{1}{m+s}.
\end{equation*}
By the Stein complex interpolation theorem, for $z$ with real part $\Re(z)=0$, we see that $W_2$ also satisfies the estimate (\ref{Lp decay for W1}).\\

$\mathbf{Case~(iii)~2^j\leq |\alpha_{s}|2^{(k-1)\eta-2}}.$\\

We use the notation $\Delta_3$ to denote the set of all $(j,k)$ satisfying $\mathbf{Case~(iii)}$. Define $W_3=\sum_{\Delta_3}W_{j,k}$ with $\widetilde{D}(x,y)=D(x,y)$. As in the previous two cases, we can
prove that $W_3$ satisfies the same estimates as $W_2$ in $\mathbf{Case~(ii)}$. An interpolation yields the desired $L^p$ decay estimate (\ref{Lp decay for W1}) with $W_3$ in place of $W_1$.\\

$\mathbf{Step~2.~m=0~and~|\alpha_{t}-\alpha_{s}|< \frac{|\alpha_{s}|}{4}~for~all~1\leq t \leq s.}$\\

We can divide our proof into three cases as in $\mathbf{Step~1}$. The arguments for $\mathbf{Case~(i)}$ and $\mathbf{Case~(iii)}$ are the same as that of $\mathbf{Step~1}$. Now we turn to $\mathbf{Case~(ii)}$. For each fixed $(j,k)\in\Delta_2$, the number of $(j',k')\in\Delta_2$ such that the horizontal (also vertical) projections of the supports of $W_{j,k}$ and $W_{j',k'}$ have nonempty intersections is bounded by a constant $C=C(\eta)$. By the almost orthogonality in Lemma \ref{simple almost orthogonality principle}, there exists a constant $C=C(\eta)$ such that
\begin{equation*}
\|W_2f\|_{L^p}\leq C\mathop{\sup}_{(j,k)\in\Delta_2}\|W_{j,k}f\|_{L^p},~~~1\leq p \leq \infty.
\end{equation*}
Thus it suffices to prove that $W_{j,k}$ satisfies (\ref{Lp decay for W1}), for each $(j,k)\in \Delta_2$, with a constant $C$ independent of $(j,k)$.

As in Sections 3 and 4, let $\{t_1,t_2,\cdots,t_a\}$ be the set of all indices $1\leq i \leq N$ such that $|\alpha_{i+1}|/|\alpha_i|\geq 2^{4N_0}$ for some large $N_0=N_0(\eta)$, where $N_0(\eta)$ is a large positive integer.
If this set is empty, then the sizes of $\alpha_{1},~\alpha_{2},\cdots,\alpha_{s}$ are equivalent up to constants depending only on $\eta$. With some abuse of notation, we set $t_1=N$ if this set is empty, i.e., $|\alpha_{i+1}|/|\alpha_i|<2^{4N_0}$ for all $i$.

With the assumption $\max_t|\alpha_{t}-\alpha_{s}|< |\alpha_s|/4$, we claim that $\{1,2,\cdots,s\}$ is a subset of $\{1,2,\cdots,t_1\}$. If the set $\{t_1,t_2,\cdots,t_a\}$ is empty, then $t_1=N$ and our claim is obviously true. Otherwise if the set in consideration is nonempty and $t_1<s$, we would obtain $t_1,t_1+1\in \{1,2,\cdots,s\}$ and
$$\frac{|\alpha_s-\alpha_{t_1}|}{|\alpha_s|}
\geq 1-\frac{|\alpha_{t_1}|}{|\alpha_{t_1+1}|}\geq 1-2^{-4}.$$
This contradicts our assumption.

It is more convenient to consider first the special case $s=t_1$. By interpolation, we are able to show the $L^p$ estimate for general $s$.

The damping factor $\widetilde{D}$ for $W_{j,k}$ is defined by Theorem \ref{Theorem main HE1-L1 estimate}, i.e.,
\begin{equation*}
\widetilde{D}(x,y):=D_{j,k}(x,y)
=\Big(|\lambda| |\alpha_{t_1}|^{-1} 2^{-k(\eta-1)} A \Big)^{-\frac{t_1}{t_1+2} }
+
\prod_{t=1}^{t_1}|x-\alpha_ty^{\eta}|
\end{equation*}
with $A>0$ given by
\begin{equation*}
A=2^{kn} \Big( \prod_{t=t_1+1}^{N} |\alpha_t| 2^{k\eta} \Big).
\end{equation*}

By Theorem \ref{Theorem L2 damping esti revision}, we have for each $(j,k)\in \Delta_2$
\begin{equation*}
\|W_{j,k}f\|_{L^2}
\leq C(1+|z|^2)
\left(|\lambda| \prod_{t=t_1+1}^{N} |\alpha_t|\right)^{-\gamma}
\|f\|_{L^2},~~~\gamma=\frac{1}{2(n+(N-t_1)\eta+1)},
\end{equation*}
where $z$ has real part $\Re(z)=\beta$. Here the constant $C$ depends only on $deg(S)$ and the cut-off $\varphi$ as in (\ref{uniform constant}). On the other hand, it follows from Theorem \ref{Theorem main HE1-L1 estimate} that for each $(j,k)\in \Delta_2$
\begin{equation*}
\|W_{j,k}f\|_{L^1} \leq C(1+|z|)\|f\|_{H_E^1(I_k)},~~~\Re(z)=-\frac{1}{t_1},
\end{equation*}
where the constant $C$ depends only $deg(S)$ and the cut-off $\varphi$.

By interpolation between the above $L^2\rightarrow L^2$ and $H_E^1\rightarrow L^1$ estimates, we see that there exists a constant $C=C(deg(S),\varphi)$ such that
\begin{equation}\label{Lp esti for s=t1}
\mathop{\sup}\limits_{\Delta_2}\|W_{j,k}f\|_{L^p}
\leq C (1+|z|^2)
\left(|\lambda|\prod_{t=t_1+1}^{N} |\alpha_t|\right)^{-\delta}\|f\|_{L^p},~~~\Re(z)=0,
\end{equation}
where the supremum is taken over all $(j,k)$ in $\Delta_2$, and $\delta,~p$ are given by
\begin{equation*}
\delta=\frac{1}{t_1+n+(N-t_1)\eta+2},
~~~~~~~~p=\frac{t_1+n+(N-t_1)\eta+2}{t_1+1}.
\end{equation*}

Let $T_{\lambda,j,k}$ be defined as $W_{j,k}$ in (\ref{sec5 def of Wjk}) but with $z=0$. With the above estimate (\ref{Lp esti for s=t1}), we are able to prove the
following estimate:
\begin{equation}\label{sec5 desired est Tjk}
\mathop{\sup}_{(j,k)\in\Delta_2}\|T_{\lambda,j,k}f\|_{L^p}
\leq
C(deg(S),\varphi)\left(|\lambda|\prod_{t=s+1}^N |\alpha_t|\right)^{-\delta}
\|f\|_{L^p}
\end{equation}
with $\delta$ and $p$ given by
\begin{equation*}
\delta=\frac{1}{s+n+(N-s)+2},
~~~~
p=\frac{s+n+(N-s)+2}{s+1}.
\end{equation*}
However, to prove this estimate, we also need an $L^p$ estimate corresponding to $s=N$,
\begin{equation}\label{sec5 desired est Tjk with s=N}
\mathop{\sup}_{(j,k)\in\Delta_2}\|T_{\lambda,j,k}f\|_{L^p}
\leq
C\left(|\lambda|\prod_{t=1}^N |\alpha_t|\right)^{-\delta}
\|f\|_{L^p},~~~p=n+N\eta+2,~~~\delta=\frac{1}{n+N\eta+2}.
\end{equation}

Now we shall prove (\ref{sec5 desired est Tjk with s=N}), and then deduce (\ref{sec5 desired est Tjk}) from (\ref{Lp esti for s=t1}) and (\ref{sec5 desired est Tjk with s=N}). Rewrite the Hessian $S_{xy}''$ as
\begin{equation*}
S_{xy}''(x,y)=y^n\prod_{t=1}^{N}\l(x-\alpha_t y^{\eta}\r)
=(-1)^N\l(\prod_{t=1}^{N}\alpha_t\r)y^n\prod_{t=1}^{N}\l(y^{\eta}-\frac{1}{\alpha_t} x\r).
\end{equation*}
Let $T_{\lambda,j,k}^{t}$ be the transpose of $T_{\lambda,j,k}$, i.e.,
$\int T_{\lambda,j,k}f(x)g(x)dx=\int f(y)T_{\lambda,j,k}^{t}g(y)dy$ for all continuous functions $f,g$ with compact support. Let $\widetilde{W}_{j,k}$ be defined as $W_{j,k}$ in (\ref{sec5 def of Wjk}) with $\widetilde{D}(x,y)=y^n\prod_{t=1}^N(y^{\eta}-x/\alpha_t)$, and $\widetilde{W}_{j,k}^t$ the transpose of $\widetilde{W}_{j,k}$. It is clear that if $z=0$ then $T_{\lambda,j,k}^t=\widetilde{W}_{j,k}^t$. It should be pointed out that the difference between $W_{j,k}$ and $\widetilde{W}_{j,k}^t$ is that the Hessian $S_{xy}''(x,y)$, as a polynomial in $y$, is not monic in the latter case. If we incorporate the coefficient $(-1)^N\prod_{t=1}^{N}\alpha_t$ in the parameter $\lambda$, then we can apply the damping estimates in Sections 3 and 4 to the dual counterpart for (\ref{sec5 desired est Tjk with s=N}). By a duality argument, (\ref{sec5 desired est Tjk with s=N}) follows immediately.

In what follows, the estimate (\ref{sec5 desired est Tjk}) will be proved by an interpolation of (\ref{Lp esti for s=t1}) and (\ref{sec5 desired est Tjk with s=N}). Choose $\theta\in[0,1]$ such that
\begin{equation*}
\frac{s+1}{s+n+(N-s)\eta+2}
=
\frac{t_1+1}{t_1+n+(N-t_1)\eta+2}\cdot\theta
+
\frac{1}{n+N\eta+2}\cdot (1-\theta).
\end{equation*}
Set $A=n+N\eta+2$ and $B=\eta-1$. Subtracting $-\frac{1}{B}=-\frac{1}{\eta-1}$ from both sides of the above equality, we obtain
\begin{equation*}
\frac{A+B}{(A-Bs)B}=\frac{A+B}{(A-Bt_1)B}\cdot\theta
+
\frac{A+B}{AB}\cdot(1-\theta).
\end{equation*}
Hence we have
\begin{equation}\label{sec5 convexity equality}
\frac{1}{(A-Bs)}=\frac{1}{(A-Bt_1)}\cdot\theta
+
\frac{1}{A}\cdot(1-\theta).
\end{equation}
It follows that
\begin{equation*}
\frac{1}{(A-Bs)}-\frac{1}{A-Bt_1}
=
\left(\frac{1}{A}-\frac{1}{(A-Bt_1)}\right)
\cdot(1-\theta)
\end{equation*}
which implies
\begin{equation}\label{a convex equality}
\frac{t_1-s}{A-Bs}=\frac{t_1}{A}\cdot(1-\theta).
\end{equation}
By interpolation between the $L^p$ estimates (\ref{Lp esti for s=t1}) and (\ref{sec5 desired est Tjk with s=N}), we get
\begin{equation*}
\mathop{\sup}_{(j,k)\in\Delta_2}\|T_{\lambda,j,k}f\|_{L^p}
\leq
C(1+|z|^2)
\left( |\lambda| \prod_{t=t_1+1}^{N} |\alpha_t| \right)^{-\gamma_1\theta}
\left( |\lambda| \prod_{t=1}^{N} |\alpha_t|  \right)^{-\gamma_2(1-\theta)}
\|f\|_{L^p},
\end{equation*}
where $p$ is given by
\begin{equation*}
p=\frac{s+n+(N-s)\eta+2}{s+1}.
\end{equation*}
Here we use $\gamma_1$ and $\gamma_2$ to denote the decay exponents in (\ref{Lp esti for s=t1}) and (\ref{sec5 desired est Tjk with s=N}), respectively. With the $\theta$ considered above, the equality (\ref{sec5 convexity equality}) gives $\gamma_1\theta+\gamma_2(1-\theta)=1/(A-Bs)$ and
\begin{eqnarray*}
\left( \prod_{t=t_1+1}^{N} |\alpha_t| \right)^{-\gamma_1\theta}
\left( \prod_{t=1}^{N} |\alpha_t|  \right)^{-\gamma_2(1-\theta)}
&=&
\left( \prod_{t=t_1+1}^{N} |\alpha_t| \right)^{-1/(A-Bs)}
\left( \prod_{t=1}^{t_1} |\alpha_t|  \right)^{-\gamma_2(1-\theta)}\\
&\approx&
\left( \prod_{t=t_1+1}^{N} |\alpha_t| \right)^{-1/(A-Bs)}
|\alpha_{t_1}|^{-t_1\gamma_2(1-\theta)}\\
&\approx&
\left( \prod_{t=s+1}^{N} |\alpha_t| \right)^{-1/(A-Bs)},
\end{eqnarray*}
where we have used the equality (\ref{a convex equality}) to $|\alpha_{t_1}|^{-t_1\gamma_2(1-\theta)}$, and the assumption that $\alpha_1,\cdots,\alpha_{t_1}$ have equivalent sizes. This proves our claim (\ref{sec5 desired est Tjk}).

Now we turn to prove Theorem \ref{Theorem Main result uniform Lp decay}. First assume $k\geq l$. We can choose an integer $s$ such that one of the following two statements is true:

\noindent ~~(i)$~~$$m+s+1=k$~~and~~$n+(N-s)\eta+1=l$ if $\eta\geq 1$;

\noindent ~~(ii) $m+s\nu+1=k$~~and~~$n+M-s+1=l$ with $\nu=\eta^{-1}$ if $\eta<1$.

\noindent Here the notations $M$ and $\nu$ are defined as in Theorem \ref{Theorem Duality of damping L2 estimate}. It suffices to consider either $s\in\{t_1,t_2,\cdots,t_a\}$ or $s=N$. By a similar argument to our proof of (\ref{sec5 desired est Tjk}), the decay estimate in Theorem \ref{Theorem Main result uniform Lp decay} for other integers $s$ follows, by interpolation of estimates corresponding to the special integers $s$ just mentioned. Moreover, it is clear that the damping factor $D$ in (\ref{Special damping factor}) is conjugate invariant for $s\in\{t_1,t_2,\cdots,t_a\}$ and $s=N$. Thus all of the above arguments apply.

Without loss of generality, we assume that the coefficient $c_0$ ($d_0$, respectively), appearing in (\ref{Hessian of phase}) ((\ref{Hessian of phase II}), respectively), is equal to one. Since $|\alpha_1|\leq |\alpha_2|\leq \cdots\leq |\alpha_N|$ and
\begin{equation*}
a_{k,l}=(-1)^{\sigma}\cdot\frac{1}{k}\cdot\frac{1}{l}\cdot\sum_{1\leq j_1<j_2<\cdots<j_{N-s}\leq N}
\alpha_{j_1}\alpha_{j_2}\cdots\alpha_{j_{N-s}}
\end{equation*}
with $\sigma=N-s$ if $\eta\geq 1$, we see that $|a_{k,l}|$ is not greater than a constant multiple of the absolute value of the product of  $\alpha_{s+1},\alpha_{s+2},\cdots,\alpha_{N}$. By these results, $T_{\lambda}$ satisfies the $L^p$ decay estimate in Theorem \ref{Theorem Main result uniform Lp decay} for $k\geq l$. While for $k<l$, the desired estimate can be proved by a duality argument.

Combining all above results, we have established the desired $L^p$ estimate for $T_{\lambda}$ in (\ref{General OIO}) with $\eta\geq 1$. Similarly, we can deal with the case $\eta<1$ without essential change. In fact, using the above interpolation method, we are able to prove the desired decay estimate for $\eta<1$, by invoking Theorem \ref{Theorem Duality of damping L2 estimate} and Remark \ref{Remark H1L1 estimate for small eta}.
\end{proof}

\section{Higher dimensional oscillatory integral operators}
\label{Higher dimensional OIO}
In this section, we shall establish uniform sharp $L^p$ estimates for higher dimensional oscillatory integral operators. Our main tools are the one dimensional result in Theorem \ref{Theorem Main result uniform Lp decay} and a variant of Stein-Weiss interpolation with change of measures. As a consequence of the $L^p$ estimates in this section, we can extend the $L^2$ decay estimate in Tang \cite{tangwan} for $(2+1)$-dimensional oscillatory integral operators. At the end of this section, some counterexamples will be presented to clarify the necessity of our assumptions.

In this section, we shall consider higher dimensional oscillatory integral operators of the form:
\begin{equation}\label{Operator Higher dimensional OIO}
T_{\lambda}f(x)=\int_{\bR^{n_Y}}e^{  i\lambda S(x,y)  }
\varphi(x,y)
f(y)dy,~~~x\in \bR^{n_X},
\end{equation}
where $\lambda\in\bR$ is a parameter and $\varphi$ is a smooth cut-off. To establish uniform estimates for $T_{\lambda}$, we need $L^p$ decay estimates with change of cut-off functions. We state this as follows.

\begin{prop}\label{Proposition Operator with change of cut-offs}
Let $T_{\lambda}$ be an oscillatory integral operator defined as in {\rm (\ref{Operator Higher dimensional OIO})}. Assume $p\in [1,\infty]$. Then the following two statements are true.

\noindent~~{\rm (i)} Let $n_X=n_Y=1$. Assume that there exists a nonnegative number $\sigma$ such that
\begin{equation}\label{Decay estimate of Tlambda}
\|T_{\lambda}f\|_{L^p}\leq C|\lambda|^{-\sigma} \|f\|_{L^p}
\end{equation}
for all smooth cut-off functions $\varphi\in C_0^{\infty}(\bR\times\bR)$ with a constant $C$ independent of $\lambda$. Then this decay estimate is also true for $T_{\lambda}$ with the cut-off function $\varphi(x,y)=\psi(x^{\gamma_1},y^{\gamma_2})\chi_{\bR^+}(x)
\chi_{\bR^+}(y)$, where $\psi\in C_0^{\infty}(\bR^2)$, $\gamma_1,\gamma_2>0$, and $\chi_{\bR^+}$ denotes the characteristic function of $\bR^+$.

\noindent~~{\rm (ii)} Assume $n_X,n_Y\geq 1$. Then the decay estimate {\rm(\ref{Decay estimate of Tlambda})} holds for all $\varphi\in C_0^{\infty}( \bR^{n_X} \times \bR^{n_Y} )$ if and only if it is true for all $x$-radial and $y$-radial $\varphi\in C_0^{\infty}( \bR^{n_X}\times \bR^{n_Y})$.
\end{prop}
\begin{remark}
This proposition turns out to be very useful in our treatment of uniform estimates for oscillatory integral operators. Its proof is very simple and invoke the Fourier expansion for smooth periodic functions.
\end{remark}
\begin{proof}
We first prove (i). Choose a radial function $\omega\in C_0^{\infty}(\bR^2)$ such that $\omega(x,y)=1$ for $(x,y)\in [-1,1]^2$ and $\omega=0$ outside the square $[-2,2]^2$. Since $\psi$ is a smooth cut-off, we can choose the smallest positive number $a>0$ such that $\supp(\psi)\subseteq [-a,a]^2$. For convenience, we shall extend $\psi$ periodically outside the larger rectangle $[-2a,2a]^2$; we use $\widetilde{\psi}$ to denote this periodic extension. Since $\widetilde{\psi}$ is smooth, we have the following Fourier expansion:
\begin{equation*}
\widetilde{\psi}(x,y)=\sum_{k,l\in\bZ} a_{k,l}e^{i(k\frac{\pi} {2a}x+l\frac{\pi} {2a}y)},
\end{equation*}
where $a_{k,l}$ are the Fourier coefficients given by
$$a_{k,l}=\frac{1}{(4a)^2}\int_{[-2a,2a]^2}\widetilde{\psi}(x,y) e^{-i(k\frac{\pi} {2a}x+l\frac{\pi} {2a}y)}dxdy.$$
By our choice of $\omega$, we have $\psi(x,y)=\omega(x/a,y/a)\widetilde{\psi}(x,y)$. It follows immediately from the above expansion that
\begin{equation*}
\psi(x,y)
=\omega\l( \frac{x}{a},\frac{y}{a} \r)
\l(
\sum\nolimits_{k,l\in\bZ} a_{k,l}e^{i(k\frac{\pi} {2a}x+l\frac{\pi} {2a}y)}
\r).
\end{equation*}
Since $\omega$ is radial, we may write $\omega(x,y)=\omega(\sqrt{x^2+y^2})$ with some abuse of notation. For clarity, we write $T_{\lambda}$ as $T_{\lambda}(S,\varphi)$ to emphasize its dependence on the phase $S$ and the cut-off $\varphi$. Then it follows from the above equality that
\begin{equation*}\label{Operator with change of cut-offs}
T_{\lambda}\Big(S,\psi(x^{\gamma_1},y^{\gamma_2})
\Big)f(x)
=
\sum\nolimits_{k,l\in\bZ}a_{k,l}T_{\lambda}\left( S+\frac{k}{\lambda}\frac{\pi} {2a}x^{\gamma_1}+\frac{l}{\lambda}\frac{\pi} {2a}y^{\gamma_2},
\omega\l(\frac{\sqrt{x^{2\gamma_1}+y^{2\gamma_2}}}{a}\r) \right)f(x)
\end{equation*}
for $f\in L^p$ satisfying $\supp(f)\subseteq\bR^+$. On the other hand, the pure $x,y$ terms in the phase of $T_{\lambda}$ do not affect the operator norm of $T_{\lambda}$. Since $\omega=1$ near the origin, $\omega\l(\frac{\sqrt{x^{2\gamma_1}+y^{2\gamma_2}}}{a}\r)$ is a smooth cut-off function. The rapid decay of $a_{k,l}$ implies $\sum|a_{k,l}|<\infty$. By our assumption, the $L^p$ operator norms of operators in the above summation have a uniform power decay bound $C|\lambda|^{-\sigma}$. This implies that
\begin{equation*}
\l \|T_{\lambda}\Big(S,\psi(x^{\gamma_1},y^{\gamma_2})\chi_{\bR^+}(x)
\chi_{\bR^+}(y)\Big) \r \|_{L^p \rightarrow L^p}
\leq
C|\lambda|^{-\sigma}
\end{equation*}
with the constant $C$ independent of $\lambda$.

The second statement can be proved in the same way as above. Indeed, following the above argument, we can show that every $\varphi\in C_0^{\infty}( \bR^{n_X} \times \bR^{n_Y} )$ can be written as the product of an absolutely convergent Fourier series and an $x$-radial and $y$-radial smooth cut-off in $\bR^{n_X}\times\bR^{n_Y}$. On the other hand, it is easy to see that insertion of each term $e^{i \frac{2\pi}{T} \overrightarrow{k}\cdot (x,y)}$ in the cut-off of $T_{\lambda}$ does not affect the $L^p$ operator norm of $T_{\lambda}$. Thus the statement (ii) follows immediately.
\end{proof}

Assume $S\in\bR[x,y]$ is a weighted homogeneous polynomial of the form (\ref{Definition Polynomial phase}). For two real numbers $\eta_1,\eta_2\geq 1$, define
\begin{equation}\label{Operator T lambda eta12}
T_{\lambda}^{\eta_1,\eta_2}f(x)
=
\int_{\bR} e^{i\lambda S(x^{\eta_1}, y^{\eta_2})} \varphi(x,y)
\chi_{\bR^+}(x)\chi_{\bR^+}(y)f(y)dy.
\end{equation}

\begin{theorem}\label{Theorem Coro of Thm 1}
Assume $S$ and $T_{\lambda}^{\eta_1,\eta_2}$ are given as above with a smooth cut-off function $\varphi$. Then there exists a constant $C=C(deg(S),\eta_1,\eta_2,\varphi)$ such that
\begin{equation*}
\l\| T_{\lambda}^{\eta_1,\eta_2} f \r\|_{L^p}
\leq
C \Big( |a_{k,l}| |\lambda| \Big)^{-\delta} \|f\|_{L^p}, ~~~
\delta=\frac{1}{k\eta_1+l\eta_2},
~~~p=\frac{k\eta_1+l\eta_2}{k\eta_1}.
\end{equation*}
\end{theorem}

\begin{proof}
First assume $\eta_1\geq1,~\eta_2=1$. By Theorem \ref{Theorem Main result uniform Lp decay}, we see that $T_{\lambda}(S,\varphi)$ satisfies the desired estimate, but with $\delta=1/(k+l)$ and $p=(k+l)/k$. By Proposition \ref{Proposition Operator with change of cut-offs}, this decay estimate is also true if the cut-off $\varphi(x,y)$ is replaced by $\varphi(x^{1/\eta_1},y)\chi_{\bR^+}(x)\chi_{\bR^+}(y)$. Now we can apply Lemma \ref{Lemma interpolaiton with change of weights} with $a=0$ and $p_0=(k+l)/k$ to obtain
\begin{equation*}
\int_{0}^{\infty}\left|\int_{0}^{\infty}\exp(i\lambda S(x,y))\varphi(x^{1/\eta_1},y)f(y)dy\right|^{p}
x^{ -(1-\theta)p }dx
\leq C \Big(|a_{k,l}||\lambda|\Big)^{ -\frac{\theta}{k+l}p  }
\int_{0}^{\infty}|f(x)|^{p}dx
\end{equation*}
for $1<p\leq p_0$, where $\theta$ satisfies $\frac{1}{p}=1-\theta+\frac{\theta}{p_0}$. Put $p=\frac{k\eta_1+l}{k\eta_1}$. It is easy to see that $\theta=\frac{p_0'}{p'}$ and $p(1-\theta)=\frac{k\eta_1+l}{k\eta_1}\cdot\frac{k(\eta_1-1)}{k\eta_1+l}=1-\frac{1}{\eta_1}$. It follows that
\begin{equation*}
\int_{0}^{\infty}\left|\int_{0}^{\infty}\exp(i\lambda S(x^{\eta_1},y)) \varphi(x,y) f(y)dy\right|^{p}
dx \leq C \Big(|a_{k,l}||\lambda|\Big)^{ -\frac{1}{k\eta_1}  } \int_{0}^{\infty}|f(x)|^{p}dx,~~~p=\frac{k\eta_1+l}{k\eta_1}.
\end{equation*}
By a duality argument, an application of Lemma \ref{Lemma interpolaiton with change of weights} again yields
\begin{eqnarray*}
& &
\int_{0}^{\infty}
\left|\int_{0}^{\infty}\exp\l(i\lambda S(x^{\eta_1},y^{\eta_2})\r) \varphi(x,y)  f(y)dy\right|^{p}dx \\
&\leq & C \Big(|a_{k,l}||\lambda|\Big)^{ -\frac{1}{k\eta_1}  } \int_{0}^{\infty}|f(x)|^{p}dx,~~~p=\frac{k\eta_1+l\eta_2}{k\eta_1}.
\end{eqnarray*}
This completes the proof of the theorem.
\end{proof}
~~~  \\  
Now consider the following class of real-valued polynomials in $\bR^{n_X}\times \bR^{n_Y}$ with $n_X,n_Y\geq 1$:
\begin{equation}\label{Polynomial phases in higher dimensions}
S(x,y)=\sum_{i=1}^{N} P_i(x) Q_i(y),
\end{equation}
where $P_i$ and $Q_i$ are real-valued homogeneous polynomials in $\bR^{n_X}$ and $\bR^{n_Y}$, respectively. Then we can state our main result for higher dimensional oscillatory integral operators as follows.
\begin{theorem}\label{Theorem poly phase with fractional terms}
Assume $S$ is a real-valued polynomial defined as above. Suppose the following conditions hold.

\noindent~{\rm (i)} There exist two integers $m\geq n_X$ and $n\geq n_Y$ such that $deg(P_i)=k_im$ and $deg (Q_i)=l_in$ for two positive integers $k_i,l_i$.

\noindent~{\rm (ii)} For two numbers $c>0$ and $d>0$, $S$ can be written as
\begin{equation*}
S(x,y)=\sum_{k_i+cl_i=d} P_i(x) Q_i(y),
\end{equation*}
where $k_i,l_i$ are given as in {\rm (i)}, and $k_i\neq k_j,l_i\neq l_j$ for $i\neq j$.

\noindent Then there exists a constant $C=C(deg(S), n_X, n_Y, \varphi)$ such that $\|T_{\lambda} f\|_{L^p(\bR^{n_X})}$ is bounded by
\begin{equation*}
C|\lambda|^{-\gamma}
\l( \int_{S^{n_X-1}} |P_i(x')|^{-\frac{n_X}{k_im}} d\sigma(x') \r)^{1/p}
\l( \int_{S^{n_Y-1}} |Q_i(y')|^{-\frac{n_Y}{l_in}} d\sigma(y') \r)^{1/p'}
\|f\|_{L^p(\bR^{n_Y})},
\end{equation*}
where $p$ and $\gamma$ are given as follows:
\begin{equation*}
p=\frac{k_im/n_X+l_in/n_Y}{k_im/n_X},~~~~~
\gamma=\frac{1}{k_im/n_X+l_in/n_Y}.
\end{equation*}
\end{theorem}
\begin{proof}
We shall apply Theorem \ref{Theorem Main result uniform Lp decay} to prove this theorem with the rotation method. By polar coordinates, we write $x=\rho x'$ and $y=ry'$ with $\rho=|x|$ and $r=|y|$. Thus the phase $S$ can be written as
$$S(x,y)=\sum_{k_i+cl_i=d} P_i(x') Q_i(y') \rho^{k_im} r^{l_in}.$$
We first calculate the $L^p$ norm of $T_{\lambda}f$ in the radial direction.
\begin{eqnarray*}
& &\int_0^{\infty} |Tf(x)|^p\rho^{n_X-1}d\rho\\
&=&
\int_0^{\infty}
\l| \int_{S^{n_Y-1}}\int_0^{\infty} e^{i\lambda S(\rho x', ry')}
\varphi(\rho x',r y') f(ry') r^{n_Y-1}dr d\sigma(y')  \r|^p \rho^{n_X-1} d\rho\\
&=&
C\int_0^{\infty}
\l| \int_{S^{n_Y-1}}\int_0^{\infty} e^{i\lambda S(\rho^{ \frac{1}{n_X} } x', r^{ \frac{1}{n_Y} }y')}
\varphi(\rho^{ \frac{1}{n_X} } x', r^{ \frac{1}{n_X} } y')
f(r^{ \frac{1}{n_Y} }y') dr  d\sigma(y') \r|^p d\rho.
\end{eqnarray*}
By Minkowski's inequality and H\"{o}lder's inequality, we can apply Proposition \ref{Proposition Operator with change of cut-offs} and Theorem \ref{Theorem Coro of Thm 1} to obtain
\begin{eqnarray*}
& &\l( \int_0^{\infty} |Tf(\rho x')|^p\rho^{n_X-1}d\rho \r)^{1/p}\\
&\leq &
C \int_{S^{n_Y-1}} \l(  \int_0^{\infty}
\l| \int_0^{\infty} e^{i\lambda S(\rho^{1/n_X} x', r^{1/n_Y}y')}
\varphi(\rho^{ \frac{1}{n_X} } x',r^{ \frac{1}{n_X} } y')
f(r^{1/n_Y}y') dr   \r|^p  d\rho \r)^{1/p} d\sigma(y')\\
&\leq&
C \int_{S^{n_Y-1}} |\lambda|^{-\delta_i}|P_i(x')|^{-\delta_i}
|Q_i(y')|^{-\delta_i} \|f((\cdot) y')\|_{L^p(r^{n_Y-1}dr)} d\sigma(y')\\
&\leq &
C |\lambda|^{-\delta_i} |P_i(x')|^{-\delta_i}
\l(  \int_{S^{n_Y-1}}|Q_i(y')|^{-\frac{n_Y}{l_in} } d\sigma(y') \r)^{1/p'} \|f\|_{L^p(\bR^{n_Y})},~~~\delta_i=\frac{1}{k_im/n_X+l_in/n_Y}.\\
\end{eqnarray*}
In the second inequality, it should be pointed out that our application of Theorem \ref{Theorem Coro of Thm 1} produces a constant $C$ independent of $x',y'$. In fact, by Fourier expansion, this can be verified by Proposition \ref{Proposition Operator with change of cut-offs}. Taking the $L^p$ norm of the above integrals over $S^{n_X-1}$, we obtain the desired result.
\end{proof}

As a consequence of the above theorem, we can make use of standard dilation to obtain the following $L^p$ boundedness of higher dimensional oscillatory integral operators without cut-off function.
\begin{theorem}\label{Theorem Lp boundedness in whole space}
Assume $S$ is a real-valued polynomial satisfying all assumptions in Theorem \ref{Theorem poly phase with fractional terms}. Let $T$ be defined as $T_{\lambda}$ in {\rm (\ref{Operator Higher dimensional OIO})}, but with $\lambda=1$ and $\varphi\equiv 1$. If, for some $i$, $P_i$ and $Q_i$ satisfy the integrability conditions:
  \begin{equation*}
\int_{S^{n_X-1}} |P_i(x')|^{-\frac{n_X}{k_im}} d\sigma(x') <\infty~~~{\rm \mathbf{and}}~~~
\int_{S^{n_Y-1}} |Q_i(y')|^{-\frac{n_Y}{l_in}} d\sigma(y') <\infty,
\end{equation*}
then $T$ is bounded from $L^{p}(\bR^{n_Y})$ into $L^{p}(\bR^{n_X})$ with $p$ defined as in Theorem \ref{Theorem poly phase with fractional terms}.
\end{theorem}

Now we shall explain why the assumptions on degree gaps $|deg(P_i)-deg(P_{i+1})|\geq n_X$ and $|deg(Q_i)-deg(Q_{i+1})|\geq n_Y$ are necessary in Theorem \ref{Theorem poly phase with fractional terms}.

Consider the phase $S(x,y)=\l(|x|^{2M_1}-|y|^{2M_2}\r)^N$,
where $M_1,M_2,N$ are positive integers. Let $P_k(x)=C_N^{N-k}|x|^{2M_1(N-k)}$ and $Q_k(y)=(-1)^k|y|^{2M_2k}$ for $0\leq k \leq N$. Then $S(x,y)=\sum_{k=0}^NP_k(x)Q_k(y)$. If $M_1\geq \frac{n_X}{2}$ and $M_2\geq \frac{n_Y}{2}$, then the assumptions (i) and (ii) in Theorem \ref{Theorem poly phase with fractional terms} are true with $m=2M_1$ and $n=2M_2$. Since $P_k$ and $Q_k$ are nonzero constants on $S^{n_X-1}$ and $S^{n_Y-1}$ respectively, it follows from Theorem \ref{Theorem poly phase with fractional terms} that
\begin{equation*}
\|T_{\lambda}f\|_{L^p(\bR^{n_X})}
\leq
C|\lambda|^{-\delta}
\|f\|_{L^p(\bR^{n_Y})}
\end{equation*}
with $p$ and $\delta$ given by
\begin{equation*}
p=\frac{2M_1(N-k)/n_X+2M_2k/n_Y}{2M_1(N-k)/n_X}~~~~~{\rm and}~~~~~
\delta=\frac{1}{2M_1(N-k)/n_X+2M_2k/n_Y}
\end{equation*}
for $1\leq k \leq N-1$.

However, assume either (i) $M_1\leq n_X/2,M_2<n_Y/2$ or (ii) $M_1< n_X/2,M_2\leq n_Y/2$ is true. Then for each $1\leq k \leq N-1$, the above decay estimate does not hold true. Assume the converse. We shall deduce a contradiction. Indeed, if the above estimate were true for some $1\leq k \leq N-1$, by Theorem \ref{Theorem Lp boundedness in whole space}, we would have
\begin{equation}\label{sec6 Lp bounded in whole space}
\left(\int_{\bR^{n_X}}\l| \int_{\bR^{n_Y}} e^{iS(x,y)}f(y)dy\r|^pdx\right)^{1/p}
\leq C \left( \int_{\bR^{n_Y}}|f(y)|^pdy \right)^{1/p}
\end{equation}
with $p$ given as above. Let $L^p_{r}(\bR^{n_Y})$ be the space of all radial functions in $L^p(\bR^{n_Y})$. It is clear that the linear mapping $f(y)\mapsto (\omega_{n_Y})^{1/p}f(r^{1/n_Y}e_1)$ is an isometry from $L^p_{r}(\bR^{n_Y})$ onto $L^p(\bR^+)$. Here $\omega_{n_Y}$ denotes the volume of the unit ball in $\bR^{n_Y}$. Now we consider the inequality (\ref{sec6 Lp bounded in whole space}) for radial functions $f\in L^p_{r}(\bR^{n_Y})$. By polar coordinates, we deduce from (\ref{sec6 Lp bounded in whole space}) that
\begin{equation}\label{sec6 Lp boundedness in 1D}
\left(\int_0^{\infty}\l| \int_0^{\infty} e^{i(\rho^a-r^b)^{N}}g(r)dr\r|^pd\rho\right)^{1/p}
\leq C \left( \int_0^{\infty}|g(r)|^pdr \right)^{1/p}
\end{equation}
for any $g\in L^p(\bR^+)$, where $a=\frac{2M_1}{n_X}$ and $b=\frac{2M_2}{n_Y}$.

By our assumption, either $a\leq 1, b<1$ or $a<1,b\leq 1$ are true. Without loss of generality, we assume $a\leq 1$ and $b<1$. In what follows, we shall prove that the inequality (\ref{sec6 Lp boundedness in 1D}) cannot be true for any $p\in [1,\infty]$. Consider first the simplest case $p=1$. Let $\{f_n\}$ be a sequence in $L^1(\bR^+)$ such that $\|f_n\|_{L^1}=1$ and $f_n$ converges weakly to the delta function supported at the origin. If we take $f=f_n$, by Fatou's lemma, we see that the left side tends to infinity as $n\rightarrow \infty$. By a duality argument, one can verify that the above inequality is not true for $p=\infty$.

Assume $1<p<\infty$. Let $g$ be the characteristic function of the interval $[K,K+\epsilon_0K^{1-b}]$, where $K\geq 1$ and $0<\epsilon_0<1$ are to be determined later. We claim that for $\rho\in [K^{\frac{b}{a}}, K^{\frac{b}{a}}+\epsilon_0]$,
\begin{equation}\label{sec6 pointwise lower bound}
\l| \int_0^{\infty} e^{i(\rho^a-r^b)^{N}}g(r)dr\r|\gtrsim K^{1-b}
\end{equation}
provided that $K$ is sufficiently large and $\epsilon_0$ is sufficiently small. Indeed, for $\rho\in [K^{\frac{b}{a}}, K^{\frac{b}{a}}+\epsilon_0]$, we have $\rho^{a/b}\geq K$ and $\rho^{a/b}\leq K+C\epsilon_0^a K^{1-b}$ for suitable $K$ and $\epsilon_0$. The latter inequality follows from
\begin{equation*}
\rho^{a/b}
\leq K\l(1+\epsilon_0K^{-b/a}\r)^{a/b}
\leq
\begin{cases}
K\left(1+\epsilon_0^{a/b}K^{-1}\right),~~~~~~\frac{a}{b}\leq 1;\\
K\left(1+\frac{2a}{b}\epsilon_0K^{-b/a}\right),~~~\textrm{ $\frac{a}{b}>1$ and $\epsilon_0K^{-b/a}\leq 2^{\frac{1}{a/b-1}}- 1$};
\end{cases}
\end{equation*}
where we have used the simple inequality $(1+x)^c\leq 1+2cx$ if $c>1$ and $0\leq x \leq 2^{\frac{1}{c-1}}-1.$ For $\rho\in [K^{\frac{b}{a}}, K^{\frac{b}{a}}+\epsilon_0]$ and $r\in [K,K+\epsilon_0K^{1-b}]$, it is clear that
\begin{equation*}
|\rho^a-r^b|=|(\rho^{a/b})^b-r^b|\lesssim K^{b-1}|\rho^{a/b}-r|
\lesssim \epsilon_0^a.
\end{equation*}
Hence the inequality (\ref{sec6 pointwise lower bound}) is true if $K\gg 1$ and $0<\epsilon_0\ll 1$. With this inequality, we see that the left side of (\ref{sec6 Lp boundedness in 1D}) is $\gtrsim K^{1-b}$, and the right side is $\lesssim K^{(1-b)/p}$. Therefore we obtain a contradiction for $p>1$ if $K$ is sufficiently large.

~~~ \\ 

\noindent{\bf Acknowledgements.} This work was supported in part by the National Natural Science Foundation of China under Grant No. 11701573. The author would like to express his gratitude to Professor Xiaochun Li for his profitable discussions and valuable suggestions.

\end{document}